\documentclass{article}
\pdfoutput=1

\usepackage{arxiv} 

\usepackage[utf8]{inputenc} 
\usepackage{amsfonts}       
\usepackage{amsmath}        
\usepackage{amsthm}         
\usepackage{bm}             
\usepackage{bbm}        	
\usepackage{graphicx}       
\usepackage{enumitem}		

\newcommand{\prob}{\mathbb{P}}
\newcommand{\expec}{\mathbb{E}}
\newcommand{\1}{\mathbbm{1}}
\newcommand{\IRD}{\texttt{IRD}}
\newcommand{\ARD}{\texttt{ARD}}
\newcommand{\CCI}{\texttt{CCI}}
\newcommand{\bbar}[1]{\bar{\bar{#1}}}

\newtheorem{theorem}{Theorem}[section]
\newtheorem{corollary}[theorem]{Corollary}
\newtheorem{lemma}[theorem]{Lemma}
\newtheorem{prop}[theorem]{Proposition}
\newtheorem{definition}[theorem]{Definition}

\newtheorem{assumption}[theorem]{Assumption}
\theoremstyle{definition}
\newtheorem*{remark}{Remark}
\newtheorem*{notation}{Notation}

\title{From inhomogeneous random digraphs to random graphs with fixed arc counts}

\author{
  Mike van Santvoort \\
  Department of Mathematics and Computer Science\\
  Eindhoven University of Technology\\
  \texttt{m.v.santvoort@tue.nl} \\
   \And
   Pim van der Hoorn \\
Department of Mathematics and Computer Science\\
  Eindhoven University of Technology\\
  \texttt{w.l.f.v.d.hoorn@tue.nl} \\
}

\begin{document}
\maketitle

\begin{abstract}
Consider a random graph model with $n$ vertices where each vertex has a vertex-type drawn from some discrete distribution. Suppose that the number of arcs to be placed between each pair of vertex-types is known, and that each arc is placed uniformly at random without replacement between one of the vertex-pairs with matching types. In this paper, we will show that under certain conditions this random graph model is equivalent to the well-studied inhomogeneous random digraph model. 

We will use this equivalence in three applications. First, we will apply the equivalence on some well known random graph models (the Erd\H{o}s-R\'enyi model, the stochastic block model, and the Chung-Lu model) to showcase what their equivalent counterparts with fixed arcs look like. Secondly, we will extend this equivalence to a practical model for inferring cell-cell interactions to showcase how theoretical knowledge about inhomogeneous random digraphs can be transferred to a modeling context. Thirdly, we will show how our model induces a natural fast algorithm to generate inhomogeneous random digraphs.
\end{abstract}

\keywords{Inhomogeneous random digraphs \and Model equivalence \and Cell-cell interaction model \and Random graph algorithm}

\section{Introduction}\label{sec:intro}
Random graphs models are becoming a more and more standardised tool to analyse real world networks. Their usefulness shines brightest in situations where data about the real word is difficult to obtain, or entirely unavailable. In such situations, random graphs are used to mimic real-world networks, so hypotheses can still be explored and tested, without the need for much data (see \cite{Drobyshevskiy2019RandomConcepts} for a review on modelling with random graphs). Moreover, random graphs are often defined through a process, making them generally efficient and simple to generate and analyse numerically (see e.g. \cite{Miller2011EfficientDegrees}). Some examples of the wide range of applications include using the stochastic block model (see \cite{Lee2019AClustering}) as a null-model to predict missing links in a graph (see \cite{Guimera2009MissingNetworks}), or using configuration-like models (see \cite{vanderHoorn2018TypicalModel}) to mimic and test the spread of a disease during an epidemic in \cite{Ball2018EvaluationStructure}.

One special random graph model that has been used extensively in both theoretical explorations and practical applications is the \emph{inhomogeneous random graph model}. In their seminal paper \cite{Bollobas2007TheGraphs} on the topic, Bollob\'as et al. managed to theoretically prove many relevant properties of the model. Subsequently, these theoretical properties could be exploited in practical application of the model. For example, they could be used in \cite{Stanley2019StochasticAttributes, Guimera2009MissingNetworks} to facilitate link prediction, or in \cite{Fortunato2016CommunityGuide} to benchmark community detection. Recently, theoretical knowledge on inhomogeneous random graphs have been extended to a directed version of the model (called \emph{inhomogeneous random digraphs}) in \cite{Cao2019OnDigraphs}, opening the door to more modelling opportunities.

While \cite{Cao2019OnDigraphs} provides results for many key properties of graphs, these cannot always be directly used in modelling efforts, since models that are used are often similar to inhomogeneous random graph, but not exactly equivalent. For example, in \cite{vanSantvoort2023MathematicallyMicroenvironment} a model is used to infer cell-cell interactions that is ``close'' to the model in \cite{Cao2019OnDigraphs}, with the exception that it fixes the number of arcs beforehand and includes connection rules stipulating where arcs can go. Because the model in \cite{vanSantvoort2023MathematicallyMicroenvironment} deviates from the inhomogeneous random graph setting, results from inhomegeneous random digraphs cannot be readily applied to this model. Therefore, the asymptotic behavior of the resulting network remains an open question, meaning properties had to be studied numerically using costly monte-carlo simulations.

In this paper, we will bridge the gap between theoretical knowledge from inhomogeneous random digraphs, and a class of random graph models that are used in practice. We call this ``new'' class of models \emph{arc assigned random digraphs}. These arc assigned random graphs differ from inhomogeneous random digraphs in the sense that their number of arcs is fixed, only to still be assigned a random position. In essence, the relation between inhomogeneous random digraphs and arc assigned random graphs is a generalisation to the relation between the classical Erd\H{o}s-R\'enyi model and Gilbert model as outlined in \cite{Janson2000RandomGraphs} (Section 1.4). We will provide this generalisation as the main result of our paper (Theorem~\ref{thm:IRD to ARD}). The definition of the models and the main result will be given in Section~\ref{sec:main result}.

We illustrate our main result by showing equivalence between some well-known random graph models. Additionally, we will extend our main result to include the model in \cite{vanSantvoort2023MathematicallyMicroenvironment}. This will constitute an extra major result (Theorem~\ref{thm:CCI to IRD}) that shows that the main result can be extended to classes of random graph models that do not directly fall into the arc assigned random digraph class. Finally, we will show that the arc assigned random digraph model provides a natural algorithm to generate inhomogeneous random digraphs. This algorithm will be linear in the amount of operations it has to execute, and will be conceptually simpler than the canonical linear algorithm to generate inhomogeneous random digraphs in \cite{Hagberg2015FastGraphs}. These applications will be given in Section~\ref{sec: applications}.

The remainder of the paper is concerned with proving our results. In Section~\ref{sec:proof thms} we will give the heuristics and machinery behind the proofs of the main theorems. Once these have been outlined, we will also execute the proofs of the main theorems. The proofs of the technical lemmas needed for these proofs will be given later in Section~\ref{sec:tool proof}.

\section{Main results}\label{sec:main result}
As mentioned in the introduction, we will generalise the equivalence that exists between the Erd\H{o}s-R\'enyi and Gilbert random graph model. In this generalisation, we will prove equivalence between two models that assign types to vertices. The difference between the two models will express itself in the way they assign edges to vertices. In this section, we aim to formalise the two equivalent models and give the conditions under which equivalence holds true. Finally, we will discuss these conditions and give examples that show how these equivalence fails when the assumptions are not met.

\subsection{Inhomogeneous random digraphs}
The first model in the equivalence, is the \emph{inhomogeneous random digraph model}. It constructs a graph $G_n$ by first defining the vertex set $[n] := \{1, 2, \ldots, n\}$, and assigning each vertex $v \in [n]$ a type $T_v$. This type is an independent sample from some overarching distribution $T$ that takes values in some set $\mathcal{S}$. We call $T$ the \emph{type distribution} of the model and $\mathcal{S}$ the \emph{type space}. 

After all vertices are assigned a type, the vertex set of $G_n$ is completed, and the arc set can be generated. This is done by fixing two functions $\kappa : \mathcal{S} \times \mathcal{S} \to \mathbb{R}^+_0$ and $\varphi_n : \mathcal{S} \times \mathcal{S} \to \mathbb{R}^+_0$. We call $\kappa$ the \emph{kernel} or the model, and $\varphi_n$ the \emph{perturbation function} of the model. For each pair of vertices $v$ and $w$ such that $v \neq w$ the arc $(v, w)$ becomes part of the arc set with probability \begin{equation}\left( \frac{\kappa(T_v, T_w)(1 + \varphi_n(T_v, T_w))}{n} \right)\wedge 1,\label{eq:arc prob}\end{equation} independent from the other arcs. Here, we defined $x \wedge y := \min\{x, y\}$ and equivalently $x \vee y := \max\{x, y\}$. Once the arc set has been generated, the construction of $G_n$ is completed. 

By defining the arc probabilities through \eqref{eq:arc prob}, it is implictly expected that the behaviour of the inhomogeneous random digraph model is captured through $\kappa$ and $T$ alone. It is generally assumed that $\varphi_n$ converges to zero in some manner as the number of vertices tends to infinity (see e.g. \cite{Cao2019OnDigraphs} Assumption 3.1 or \cite{Bollobas2007TheGraphs} Definition 2.9). Since it is notationally inconvenient to always explicitly separate the kernel and the perturbation function, we will write $\kappa_n := \kappa \cdot (1 + \varphi_n)$. In the same spirit, we will also abbreviate the inhomogeneous random digraph model by $\texttt{IRD}_n(T, \kappa_n)$.

This description of the inhomogeneous random digraph model is flexible, and admits much freedom in the choice of kernel, perturbation function and type distribution. Therefore, many classical random graph models fall within this framework (see Section~\ref{sec:classical models} or Example 2.1 in \cite{Cao2019OnDigraphs}). In this paper we shall restrict ourselves to discrete type spaces.

\begin{assumption}[Discrete type spaces]\label{ass:type distr}
    The type distribution $T$ takes values in some set $\mathcal{S} \subseteq \mathbb{N}$ and satisfies $\expec[T^\delta] < \infty$ for some $\delta > 0$. 
\end{assumption}
\begin{notation}
    We set $q_t := \prob(T =t)$ and we denote by $N_t$ the number of vertices with type $t \in \mathcal{S}$.
\end{notation}

\begin{remark}
In principle, Assumption~\ref{ass:type distr} is stated slightly restrictively for convenience. Practically, it only stipulates that the type distribution is discrete. In case the type distribution takes values in a countable set unequal to the natural number, then we can still relabel the types in such a way that Assumption~\ref{ass:type distr} is satisfied.
\end{remark}

\subsection{Arc assigned random digraphs}\label{sec:ARD}
In $\texttt{IRD}_n(T, \kappa_n)$ it is determined for each vertex-pair separately whether an arc is drawn in between them. For the second model in the equivalence, which we will call the \emph{arc assigned random digraph model}, this arc placement procedure changes. To construct a graph $G_n$ we still initially define the vertex set $[n]$, and assign each vertex $v \in [n]$ a type $T_v$ that is drawn independently from some type distribution $T$ that takes values in $\mathcal{S}$. However, now the number of arcs to be placed between two vertex types is fixed.

To make this formal, we define a function $\Lambda_n : \mathcal{S} \times \mathcal{S} \to \mathbb{N}$. We call $\Lambda_n$ the \emph{arc-count function}. The value of $\Lambda_n(t, s)$ for a tuple of vertex-\emph{types} $(t, s) \in \mathcal{S} \times \mathcal{S}$ encodes the number of arcs that will be placed in $G_n$ from a vertex with type $t$ to a vertex with type $s$. The assignment of arcs to vertex pairs takes place after vertices have been assigned a type. Given $\Lambda_n$ the arc assignment procedure is executed in the following steps:
\begin{enumerate}\itemsep0em
    \item Fix a vertex-type $t \in \mathcal{S}$ to which no arcs have been assigned yet, and define the set \[V_t := \left\{ v \in [n] : T_v = t \right\}.\] This set encodes all vertices of type $t$, and note it is deterministic after assigning types to vertices.
    \item Choose $\Lambda_n(t, s)$ arcs uniformly at random without replacement from the set \[(V_t \times V_s)\backslash \{(v, v) : v \in [n] \}.\] If the size of this set is smaller than $\Lambda_n(t, s)$, simply take the entire set. Add all arcs to the arc set of $G_n$.
    \item Repeat steps 1 and 2 until all vertex-type tuples have been considered.
\end{enumerate}

We will denote the realisation of this model by $\texttt{ARD}_n(T, \Lambda_n)$. Note for the arc assigned random digraph model that independence exists in the arc-assignment procedure between vertex-type pairs (i.e., between separate executions of step 1 and 2). However, for each fixed pair of vertex-types there is dependence between the chosen arc locations, since these are drawn uniformly at random \emph{without replacement}. 

\begin{remark}
A well-known special case of the arc assigned random graph model is obtained when we take $T = 1$ (i.e., we assign all vertices the type $1$) and $\Lambda_n(1, 1) = N$ for some fixed $N \in \mathbb{N}$. This turns out to be the directed equivalent to the original Erd\H{o}s-R\'enyi random graph model as described in \cite{Erdos1959OnI.}. For this model it is known that it is equivalent to the so-called Gilbert model\footnote{Often, this model is referred to as the Erd\H{o}s-R\'enyi model.} (see \cite{Gilbert1959RandomGraphs}), which can be seen $\texttt{IRD}_n(T, \kappa_n)$ with $\kappa_n(1, 1) = \lambda$ for some $\lambda \in \mathbb{R}^+_0$. Specifically, in Section 1.4 of \cite{Janson2000RandomGraphs} it is argued that equivalence holds whenever $N \approx \lambda n$. We will generalise this argument to show equivalence between $\texttt{IRD}_n(T, \kappa_n)$ and $\texttt{ARD}_n(T, \Lambda_n)$. This is the main result of the next section.
\end{remark}

\subsection{Equivalence between the two models}
To link $\texttt{IRD}_n(T, \kappa_n)$ and $\texttt{ARD}_n(T, \Lambda_n)$ we need to describe $\Lambda_n$ in terms of $\kappa_n$ in a clever way. A natural choice would be choosing $\Lambda_n$ to be the expected number of arcs between fixed vertex types in $\IRD$. Although natural, we will see this choice can cause issues for vertex-types that are rare given a fixed value of $n$. We will call these vertex-types \emph{unstable}. 

\begin{definition}[Unstable vertex-types]\label{def:stable vertices}
    Fix a number $n \in \mathbb{N}$, and a tolerance level $\tau \in (0, 1)$. Define the number \[u_n^\uparrow(\tau) := \inf\{t : q_s < n^{-1 + \tau} \text{ for all } s \geq t\}.\] We say a vertex-type $t \in \mathcal{S}$ is \emph{unstable at tolerance $\tau$} if $t \geq u_n^\uparrow(\tau)$. If a vertex-type is not \emph{unstable at tolerance $\tau$} we call it \emph{stable at tolerance $\tau$}. Instability for vertices can be defined analogously through their assigned types.
\end{definition}

We will see later (in Section~\ref{sec:discussion conditions} and \ref{sec:heuristics thm 1}) that unstable vertex-types may cause issues in the link between $\IRD$ and $\ARD$. Thus, we need to make an assumption on the behaviour of the kernel $\kappa$ at unstable vertex-types. Simply put, we will assume that the kernel is relatively small when a vertex-type is unstable, so that the influence of these types on the output of our models is negligable.

\begin{assumption}[Kernel bound]\label{ass:kernel bound}
    Fix an $n \in \mathbb{N}$ large and two vertex types $t,s \in \mathcal{S}$ of which at least one is unstable at a tolerance $\tau \in (0, 1)$. We assume there exist two constants $\alpha, C > 0$ with $\alpha \in (1/2 - \tau/2, 1/2)$ such that
    \begin{equation}\label{eq:kernel bound}
    \kappa(t, s) \leq \frac{C n^{-1/2 + \alpha}}{\sqrt{q_t q_s}}.
    \end{equation}
\end{assumption}
\begin{remark}
    The value of $q_t$ and $q_s$ in \eqref{eq:kernel bound} implicitly depend on $n$ and $\tau$, since we require at least one of $t$ and $s$ to be unstable. In Definition~\ref{def:stable vertices} we can see that stability is an $n$ and $\tau$ dependent property. For finite type-spaces we have that $\min_{t \in \mathcal{S}} q_t > 0$. Hence, if $n \to \infty$, then there will be a point at which $\min_{t \in \mathcal{S}} q_t > n^{-1 + \tau}$ for any $\tau \in (0, 1)$. Thus, when the type-space is finite we automatically have that Assumption~\ref{ass:kernel bound} is satisfied. 
\end{remark}

In the link between $\IRD$ and $\ARD$ we will be showing that probabilities of certain events are asymptotically the same. Inspired by Section 1.4 in \cite{Janson2000RandomGraphs} we will show that \emph{monotone} events can be translate between our two models. Note in our models that the definition of monotonicity will be slightly different from the definition in \cite{Janson2000RandomGraphs}. This is because we consider slightly different graphs: our graphs are directed and our vertices have types (i.e., are marked).

Because we consider marked graph, we need to be careful how we define and interpret sub-graphs. Given two marked graphs $G_1$ and $G_2$, we will say $G_1$ is a sub-graph of $G_2$ (and write $G_1 \subseteq G_2$) whenever:
\begin{enumerate}[label = (\arabic*)]
    \item The vertices of $G_1$ have the same marks as the vertices of $G_2$.
    \item The arcs of $G_2$ include the arcs of $G_1$.
\end{enumerate}
Under this interpretation of sub-graphs we can define what monotonicity means for our models.

\begin{definition}[Monotone events]\label{def:monotone events}
    Let $G_1$ and $G_2$ be two graphs such that $G_1 \subseteq G_2$. We say a collection $\mathcal{Q}_n$ of events is \emph{increasing} if $G_1 \in \mathcal{Q}_n$ implies $G_2 \in \mathcal{Q}_n$. Similarly, we say $\mathcal{Q}_n$ is \emph{decreasing} if $G_2 \in \mathcal{Q}_n$ implies $G_1 \in \mathcal{Q}_n$. Finally, we say $\mathcal{Q}_n$ is \emph{monotone} if it is either increasing or decreasing.
\end{definition}

We can now formulate the main result of the paper. Heuristically, it tells us that a monotone graph property is true for the $\ARD$ model whenever it is true for a ``range'' of $\IRD$ models.

\begin{theorem}[$\IRD$ to $\ARD$]\label{thm:IRD to ARD}
Fix three numbers $\alpha, \tau, C > 0$, a vertex-type distribution $T$ and a kernel $\kappa$ such that Assumption~\ref{ass:type distr} and \ref{ass:kernel bound} is satisfied. Suppose for some monotone event $\mathcal{Q}_n$ there exists a number $p \in [0, 1]$ such that
\[
\prob(\IRD_n(T, \kappa'_n) \in \mathcal{Q}_n) \to p,
\]
as $n \to \infty$ for all sequences $\kappa'_n $ that satisfy the inequality
\begin{equation}\label{eq:kappa seq bound}
|\kappa'_n(t, s) - \kappa(t, s) | \leq \frac{C n^{-1/2 + \alpha}}{\sqrt{q_t q_s}}.
\end{equation}
Then, for $\Lambda_n$ satisfying
\[\Lambda_n(t, s) = \begin{cases} \lfloor \kappa(t, s) q_t q_s n\rfloor, & \text{if } t \text{ and } s \text{ is stable,}\\ 0, & \text{else.} \end{cases}\] we have
\[
\prob(\ARD_n(T, \Lambda_n) \in \mathcal{Q}_n) \to p.
\]
\end{theorem}

\subsection{Discussion of the conditions}\label{sec:discussion conditions}
Theorem~\ref{thm:IRD to ARD} leans on two assumptions: the event $\mathcal{Q}_n$ must be monotone, and the kernel inequality \eqref{eq:kernel bound} should be satisfied. Following \cite{Janson2000RandomGraphs}, the need for monotonicity can be seen by considering the event that the graph contains an exact number of arcs. The need for \eqref{eq:kernel bound} can be seen by looking at an event that only involves unstable vertices. We will discuss both examples below. Basically, a mismatch between $\IRD$ and $\ARD$ occurs whenever the random number of arcs between given vertex-types in $\IRD$ varies too much when compared to the fixed number of arcs in $\ARD$. There is ``less randomness'' in $\ARD$ than in $\IRD$, thus problems occur if we lean heavily on the randomness $\IRD$ has, but $\ARD$ does not.

\paragraph{Monotonicity.} Let $\mathcal{Q}_n$ be the event that the graph contains exactly $n$ arcs, and note this event is not monotone. We consider $\ARD_n(1, n)$. In other words, our model has one vertex-type and $n$ arcs will be placed between the vertices. Note that this event is non-monotone. Moreover, for this $\mathcal{Q}_n$ it will be difficult to relate $\IRD$ and $\ARD$, since the number of arcs in $\ARD$ is fixed, while in $\IRD$ it is random. Thus, $\mathcal{Q}_n$ will be true with probability $1$ for $\ARD$ if we take $n$ edges as input, while for $\IRD$ the probability of $\mathcal{Q}_n$ will always be strictly less than one. 

To see this more rigorously, note first that $\prob(\ARD_n(1, n) \in \mathcal{Q}_n) = 1$ for all $n$. However, when we look at $\IRD_n(1, 1 + \varepsilon_n)$, i.e. the $\IRD$ model with kernel $1 + \varepsilon_n$ for some $\varepsilon_n \to 0$, then we notice that the number of arcs follows a Binomial distribution with $n^2$ trials and succes probability $(1 + \varepsilon_n)/n$. Thus, the median of this distribution for large $n$ is given by either $n$, $n + 1$ or $n - 1$, implying that e.g. $\prob(\IRD_n(1, 1 + \varepsilon_n) \in \mathcal{Q}_n) < 3/4$ (since the probability of being smaller or equal to the median is roughly $1/2$).

\begin{remark}
    Note the $\IRD$ model in the above example is equivalent to the Gilbert model. Principally, if the kernel does not vary in $\IRD$, then the vertex-type serve no distinguishing purpose. However, the vertex-types do play distinguishing roles in $\ARD$ due to us fixing a (different) number of arcs per vertex-type pair. Hence, the above example also highlights how in certain cases the extra randomness in $\IRD$ makes the overarching model simpler.
\end{remark}

\begin{remark}
    As argued in \cite{Janson2000RandomGraphs} (Remark 1.14), monotonicity is a sufficient condition in Theorem~\ref{thm:IRD to ARD}, but not necessary. One could impose alternative restrictions stipulating how $\prob(\ARD_n(T, \Lambda'_n) \in \mathcal{Q}_n)$ behaves for choices of $\Lambda_n'$ ``near'' $\Lambda_n$. This, however, would make the statement and proof of Theorem~\ref{thm:IRD to ARD} more technical, so we refrained from making this generalisation. 
\end{remark}

\paragraph{Kernel inequality.} To explain the need for \eqref{eq:kernel bound} we will consider $\IRD(T, \kappa)$ with $q_t = C / t^3$ and $\kappa(t, s) = C^{-2}$. We will consider $\mathcal{Q}_n$ to be the event that there are no arcs from a vertex of type $1$ to a vertex of type $\lceil \sqrt[3]{n} \rceil$. Note that vertex-type $\lceil \sqrt[3]{n} \rceil$ is unstable, because the sequence $(q_t)_{t \geq 1}$ is decreasing and for $n$ large
\[
q_{\lceil \sqrt[3]{n} \rceil} = \frac{C}{\lceil \sqrt[3]{n} \rceil^3} < \frac{C}{n} \leq \frac{1}{n^{1 - \tau}},
\]
for any $\tau \in (0, 1)$. We will consider the sub-sequence of cubic values of $n$. We will derive a lower and upper bound on the probability that $\mathcal{Q}_n$ occurs in the $\IRD$ model. We start with the lower-bound.

In this specific $\IRD$ model, we see that the number of vertices with type $\sqrt[3]{n}$ is given by $N_{\sqrt[3]{n}} \sim \texttt{Bin}(n, C / n)$. Thus, as $n \to \infty$ we see that $N_{\sqrt[3]{n}} \to \texttt{Poi}(C)$. Particularly, this means that \begin{align*}\prob(\IRD_n(T, \kappa) \in \mathcal{Q}_n) &\geq \prob(N_{\sqrt[3]{n}} = 0) + \prob(\IRD_n(T, \kappa) \in \mathcal{Q}_n \;|\; N_{\sqrt[3]{n}} = 1) \prob( N_{\sqrt[3]{n}} = 1),\\ &= \left( \exp(- C) + \prob(\IRD_n(T, \kappa) \in \mathcal{Q}_n \;|\; N_{\sqrt[3]{n}} = 1) C \exp(- C) \right)(1 + o(1)). \end{align*}

Now, observe that conditional on $\{N_{\sqrt[3]{n}} = 1\}$ the number of vertices with type $1$ is given by $N_1 \sim \texttt{Bin}(n - 1, C / (1 - 1/n))$. In particular, we have that $N_1 < n$. Thus, we can bound
\[
\prob(\IRD_n(T, \kappa) \in \mathcal{Q}_n \;|\; N_{\sqrt[3]{n}} = 1) > \prob(\IRD_n(T, \kappa) \in \mathcal{Q}_n \;|\; N_{\sqrt[3]{n}} = 1, N_1 = n).
\]

If we denote by $A_n$ the number of arcs from vertices of type 1 to vertices of type $\sqrt[3]{n}$, then we have that $A_n \sim \texttt{Bin}(N_1 N_{\sqrt[3]{n}}, C^{-2}/n)$. Thus, knowing that $N_1 = n$ and $N_{\sqrt[3]{n}} = 1$ entails $A_n \sim \texttt{Bin}(n, C^{-2}/n)$. Note that this binomial distribution converges to $\texttt{Poi}( C^{-2})$, thus we find
\[
\prob(\IRD_n(T, \kappa) \in \mathcal{Q}_n \;|\; N_{\sqrt[3]{n}} = 1, N_1 = n) \geq \prob(\texttt{Bin}(n , C^{-2} / n) = 0) = \exp(-  C^{-2}) (1 + o(1)).
\]
All in all, this shows for $\IRD$ that we have the following lower-bound:
\begin{equation}\label{eq:kernel assumption discussion lower bound}
\prob(\IRD_n(T, \kappa) \in \mathcal{Q}_n) \geq \left( \exp(- C) + C \exp(- C^{-2} - C) \right) (1 + o(1)).
\end{equation}
We can derive a upper-bound using the same arguments. First note that
\[
\prob(\IRD_n(T, \kappa) \in \mathcal{Q}_n) \leq 1 -  \prob(\IRD_n(T, \kappa) \not\in \mathcal{Q}_n \;|\; N_{\sqrt[3]{n}} = 1) \prob( N_{\sqrt[3]{n}} = 1),
\]
We again observe that conditional on $\{N_{\sqrt[3]{n}} = 1\}$ the number of vertices with type $1$ is given by $N_1 \sim \texttt{Bin}(n - 1, C / (1 - 1/n))$. In particular, the median of this distribution is given approximately by $n C$, meaning for $\varepsilon > 0$ small we have that $\prob(N_1 > \varepsilon n) > 1/2$. Thus, we can bound
\[
\prob(\IRD_n(T, \kappa) \not\in \mathcal{Q}_n \;|\; N_{\sqrt[3]{n}} = 1) \geq \prob(\IRD_n(T, \kappa) \not\in \mathcal{Q}_n \;|\; N_{\sqrt[3]{n}} = 1, N_{1} > \varepsilon n) / 2.
\]
In this case, knowing that $N_1 > \varepsilon n$ and $N_{\sqrt[3]{n}} = 1$ entails $A_n \succeq \texttt{Bin}(\varepsilon n, C^{-2}/n)$. Because this binomial converges to a $\texttt{Poi}(\varepsilon C^{-2})$ random variable, we have that
\[
\prob(\IRD_n(T, \kappa) \not\in \mathcal{Q}_n \;|\; N_{\sqrt[3]{n}} = 1, N_{1} > \varepsilon n) / 2 \geq \prob(\texttt{Bin}(\varepsilon n, C^{-2}/n) > 0)/2 = \left( 1 - \exp(-\varepsilon C^{-2}) \right) (1 + o(1))/2. 
\]
Hence, we find the upper-bound
\begin{equation}\label{eq:kernel assumption discussion upper bound}
  \prob(\IRD_n(T, \kappa) \in \mathcal{Q}_n) \leq 1 - C  \exp(- C) \cdot \frac{ \left( 1 - \exp(-\varepsilon C^{-2}) \right) (1 + o(1))}{2}.  
\end{equation}
Now we compare \eqref{eq:kernel assumption discussion lower bound} and \eqref{eq:kernel assumption discussion upper bound} to the possible intuitive choices of inputs in $\ARD$. If we look at the $\ARD$ model with the expected number of arcs between fixed vertex types from $\IRD$ as its input, then we would consider $\ARD_n(T, \Lambda_n')$ with $\Lambda_n'(t, s) = \lfloor n / (t^3 s^3)  \rfloor$. Particularly, we have that $\Lambda_n'(1, \sqrt[3]{n}) =  1$. Thus, if there is a vertex with type $\sqrt[3]{n}$, then we will also have at least one arc from a vertex with type $1$ to one with type $\sqrt[3]{n}$. In other words using \eqref{eq:kernel assumption discussion lower bound}: \[\lim_{n\to \infty} \prob(\ARD_n(T, \Lambda_n') \in \mathcal{Q}_n) = \exp(- C) <  \exp(- C) + C \exp(- C^{-2} - C) \leq \liminf_{n \to \infty} \prob(\IRD_n(T, \kappa) \in \mathcal{Q}_n). \]

Alternatively, if we were to look at the $\ARD$ model from Theorem~\ref{thm:IRD to ARD} with $\Lambda_n(1, \sqrt[3]{n}) =  0$. In particular this means using \eqref{eq:kernel assumption discussion upper bound} that
\[\lim_{n\to \infty} \prob(\ARD_n(T, \Lambda_n) \in \mathcal{Q}_n) = 1 >   1 - C  \exp(- C) \cdot \frac{ \left( 1 - \exp(-\varepsilon C^{-2}) \right)}{2} \geq \limsup_{n \to \infty} \prob(\IRD_n(T, \kappa) \in \mathcal{Q}_n). \]

Thus, we see that the probability of $\mathcal{Q}_n$ occurring in $\IRD$ and $\ARD$ significantly differs. This makes it difficult to link the models. The mismatch occurs, since the occurrence of $\mathcal{Q}_n$ is heavily influenced by both vertex-type assignment and arc generation in $\IRD$, while it is only influenced by vertex-type generation in $\ARD$. If it is possible to link the models, then the arc generation in $\IRD$ cannot have a large influence on the final probability.

\section{Applications}\label{sec: applications}
In this section we will show some of the consequences of Theorem~\ref{thm:IRD to ARD}. Specifically, we will focus on four aspects. Firstly, we will highlight the consequences of our result in some ``classical'' random graph models. Secondly, we will show that the connection between $\IRD$ and $\ARD$ in Theorem~\ref{thm:IRD to ARD} can be adapted to forge a connection between $\IRD$ and models that fall outside of the $\ARD$ class. As a running example, we will investigate a recent model to infer cell-cell interaction networks given in \cite{vanSantvoort2023MathematicallyMicroenvironment}. Thirdly, we will show how our result can be used to compute properties of $\ARD$ based on calculations in $\IRD$. Fourthly, we will explain how our results provide an intuitive linear time algorithm to generate random graphs.

\subsection{Classical random graph models}\label{sec:classical models}

\subsubsection{Directed Erd\H{o}s-R\'enyi and Gilbert model}\label{sec:ER}
We recall that the directed counterpart of the classical Gilbert model in \cite{Gilbert1959RandomGraphs} can be seen as $\IRD_n(1, \lambda)$ for some $\lambda > 0$, while the directed counterpart of the Erd\H{o}s-R\'enyi model in \cite{Erdos1959OnI.} can be seen as $\ARD(1, m)$ for some $m \in \mathbb{N}$. It is important to observe in this simple setting that $\mathcal{S} = \{1 \}$. In particular, this means that vertex-type $1$ is always stable according to Definition~\ref{def:stable vertices}, irrespective of the choice of $\tau$. Hence, the parameter $\tau$ plays no role, and $\alpha$ can be chosen as close to zero as we like without invalidating Assumption~\ref{ass:kernel bound}.

In fact, since $\alpha$ can be chosen arbitrarily close to zero, we may replace the fixed $\alpha$ in \eqref{eq:kappa seq bound} with a sequence $\alpha_n$ that converges to zero arbitrarily slowly. This will make the condition \eqref{eq:kappa seq bound} stronger, meaning there will be less sequences that satisfy it. This, however, is good news in light of Theorem~\ref{thm:IRD to ARD}, because it means the $\IRD_n(\kappa_n')$ probability has to converge for fewer sequences $\kappa_n'$. 

In light of this discussion, writing $\texttt{Gil}_n(\lambda)$ for the (directed) Gilbert model and $\texttt{ER}_n(m)$ for the (directed) Erd\H{o}s-R\'enyi model, the consequence of Theorem~\ref{thm:IRD to ARD} for these models is as follows.

\begin{corollary}[From Erd\H{o}s-R\'enyi to Gilbert]\label{cor:ER to Gil}
Fix a constant $C > 0$, a decreasing sequence $\alpha_n \in (0, 1/2)$ and a model parameter $\lambda > 0$. Suppose for some monotone event $\mathcal{Q}_n$ there exists a $p \in [0, 1]$ such that 
\[
\prob(\texttt{Gil}_n(\lambda_n) \in \mathcal{Q}_n) \to p,
\]
as $n \to \infty$ for all sequences $\lambda_n$ satisfying
\begin{equation*}
|\lambda_n - \lambda | \leq C n^{-1/2 + \alpha_n}.
\end{equation*}
Then, also
\[
\prob(\texttt{ER}_n(\lfloor \lambda n \rfloor) \in \mathcal{Q}_n) \to p.
\]
\end{corollary}

Corollary~\ref{cor:ER to Gil} can be identified as being almost equivalent to Proposition 1.13 in \cite{Janson2000RandomGraphs}. The only difference is that they can choose $\alpha_n = 0$ while for us it is converging arbitrarily slowly to zero. The slightly stronger assumption for us emerges, because we use it in the more general cases to overcome the extra randomness introduced by $T$.

\subsubsection{Stochastic block model}
The stochastic block model, as described in e.g. \cite{Abbe2018CommunityDevelopments}, is a random graph model classically used to investigate community detection. The model creates graphs with vertex-set $[n]$, and each vertex is given a type from the set $\mathcal{S} = \{1, 2, \ldots, r\}$ for some fixed $r > 1$. Then, an arc between two vertices $v, w \in [n]$ with types $T_v$ and $T_w$ is drawn with some fixed probability $\pi(T_v, T_w)$, independently of the other arcs. Hence, the connection procedure is fully defined when vertex-types and the function $\pi : \mathcal{S} \times \mathcal{S} \to (0, 1)$ are known. We denote this model by $\texttt{SBM}_{n, r}(T, \pi)$.

Although the stochastic block model shows many similarities with $\IRD$, there are still two major differences from our context. Firstly, in $\IRD$ the connection probabilities scale with $1/n$ while in the stochastic block model $\pi$ does not have to decline as $n \to \infty$. Thus, we will assume in the stochastic block model there exists a kernel $\kappa$ such that for two vertex-types $t, s \in \mathcal{S}$ we have $\pi(t, s) = \kappa(t, s) / n$. Here, division by $n$ is needed to make the graph sparse.

Secondly, the stochastic block model usually fixes deterministic vertex-types (see e.g. \cite{Lee2019AClustering}), while we draw them randomly from some type distribution $T$. The goal of these deterministic types is to ensure each vertex-type $t \in \mathcal{S}$ covers approximately a proportion $q_t$ of vertices. In light of this goal, it is not too dissimilar to give each vertex $v \in [n]$ a \emph{random} type $T_v$ with $\prob(T_v = t) = q_t$, since due to the law of large numbers (or more precisely Lemma~\ref{lem:well-concentrated types}) we have that $N_t / n \approx q_t$ for all $t \in \mathcal{S}$ when $n$ is large. Concluding, we can see the stochastic block model as $\IRD_n(T, n \pi)$ where $\kappa = n \pi$ is some kernel, and $T$ is a discrete type-distribution taking values in $[r]$ for some fixed $r > 0$.

The $\ARD$ ``equivalent'' to the stochastic block model is also sometimes seen in literature (see \cite{Peixoto2019BayesianBlockmodeling}). It is called the microcanonical stochastic block model, and instead of fixing connection probability function $\pi$ it simply fixes a list $e_n : \mathcal{S} \times \mathcal{S} \to \mathbb{N}$ encoding the number of edges that need to be placed between vertices of two given types. We denote this model by $\texttt{MSBM}_{n, r}(T, e_n)$ and we see that it is equivalent to $\ARD_n(T, e_n)$.

When the number of vertex-types is finite, recall that Assumption~\ref{ass:kernel bound} is automatically satisfied. Similar to the discussion in Section~\ref{sec:ER}, this means $\tau$ plays no role and $\alpha$ can be chosen however we fancy. Additionally, since $q^\downarrow := \inf_t\{q_t\}> 0$, we can bound $\sqrt{q_t q_s} \geq q^\downarrow$ in \eqref{eq:kappa seq bound} and merge it with the constant $C$ to simplify the condition that all sequences $\kappa'_n$ must satisfy in Theorem~\ref{thm:IRD to ARD}. Thus, the consequence our main result for the stochastic block model can be formulated as follows.

\begin{corollary}[Fixing arcs in the stochastic block model]
    Fix a constant $C > 0$, a decreasing sequence $\alpha_n \in (0, 1/2)$, and a type distribution $T$ with $[r]$ as its support for some $r > 1$. Furthermore, fix a probability function $\pi_n = \kappa / n$ for some $n$-independent function $\kappa$. Suppose for some monotone event $\mathcal{Q}_n$ there exists a number $p \in [0, 1]$ such that 
\[
\prob(\texttt{SBM}_{n, r}(T, \pi'_n) \in \mathcal{Q}_n) \to p,
\]
as $n \to \infty$ for all sequences $\pi'_n$ satisfying
\begin{equation*}
|\pi_n'(t, s) - \pi_n(t, s) | \leq C n^{-3/2 + \alpha_n}.
\end{equation*}
Then, also for $e_n(t, s) = \lfloor \pi_n(t, s) q_t q_s n^2 \rfloor$ we have that
\[
\prob(\texttt{MSBM}_{n, r}(T, e_n)  \in \mathcal{Q}_n) \to p.
\]
\end{corollary}

\subsubsection{Chung-Lu model}\label{sec:CL}
In the Chung-Lu model, for which the undirected equivalent is described in \cite{Chung2006TheDegrees}, each vertex $v \in [n]$ is given a weight $w_v > 0$. Given the weights $w_v$ and $w_u$ of two vertices $v, u \in [n]$ and the sum of all weights $\ell_n$, the probability that an arc is drawn from $v$ to $u$ is given by $w_v w_u / \ell_n$ independent of the other arcs. This is again close to the setting of $\IRD$, but we will make the slight modification (that is often done; see e.g. Chapter 6 of \cite{Hofstad2016RandomNetworks}) that weights are drawn independently from a weight distribution $W$. Then, we find ourselves in our setting. Specifically, the Chung-Lu model can now be seen as $\IRD_n(W, \kappa_n)$ with
\[
\kappa_n(t, s) = \frac{ts}{\sum_{v \in [n]} W_v}. 
\]
Assuming the first moment of $W$ is finite, we can explicitly identify the kernel and perturbation function of the Chung-Lu model. Setting $\overline{W}$ to be the empirical mean of the weights we find
\[
\kappa_n(t, s)/n = \frac{ts/\expec[W]}{n} \cdot \frac{\expec[W]}{\overline{W}}.
\]
Hence, for the kernel we find $\kappa(t, s) = ts / \expec[W]$ and for the perturbation function we find $\varphi_n(t, s) = \expec[W] / \overline{W} - 1$. Since $W$ might have an infinite support, it is not immediately clear that Assumption~\ref{ass:kernel bound} is satisfied. Luckily, only a mild additional assumption on $W$ is needed. See Section~\ref{app:extra proofs} for the proof.

\begin{prop}[Assumption satisfaction for Chung-Lu]\label{prop:ass CL}
    Suppose there exists a $\varepsilon > 0$ for which $\expec[W^{1 + \varepsilon}] < \infty$. Then, there exists a tolerance $\tau \in (0, 1)$ for which Assumption~\ref{ass:kernel bound} is satisfied.
\end{prop}

In light of Proposition~\ref{prop:ass CL}, the consequence to Theorem~\ref{thm:IRD to ARD} for the Chung-Lu model will become:

\begin{corollary}[Fixing arcs in the Chung-Lu model]\label{cor:CL}
    Consider the Chung-Lu model with i.i.d. weights drawn from $W$, and assume that $\expec[W^{1 + \varepsilon}] < \infty$ for some $\varepsilon > 0$. Fix an $\alpha > 1/2 - \varepsilon/(4 + 3 \varepsilon)$. Suppose for some monotone event $\mathcal{Q}_n$ there exists a number $p \in [0, 1]$ such that
    \[
    \prob(\IRD_n(W, \kappa'_n) \in \mathcal{Q}_n) \to p,
    \]
    as $n \to \infty$ for all sequences $\kappa'_n$ satisfying
    \[
    \left| \kappa'_n(t, s) - \frac{t s}{\expec[W]} \right| \leq \frac{C n^{-1/2 + \alpha}}{\sqrt{q_t q_s}}.
    \]
    Then, we also have for $\Lambda_n(t, s) = \lfloor t s q_t q_s n / \expec[W] \rfloor$ (when $t$ and $s$ are stable; $\Lambda_n(t, s) = 0$ otherwise) that
    \[
    \prob(\ARD_n(W, \Lambda_n) \in \mathcal{Q}_n) \to p.
    \]
\end{corollary}

\begin{remark}
    The Chung-Lu model we considered in this section is directed, because the $\IRD$ and $\ARD$ models are directed. However, due to symmetry of the kernel $\kappa_n$, the model of this section is closely related to the undirected Chung-Lu model. To find a ``true'' generalisation of the Chung-Lu model in the directed case, we would need to specify two weights per vertex. One of these is used for the possible out-arc, while the other is used for the possible in-arc. In this setting, the state space would become $\mathcal{S} = \mathbb{N}^2$. Despite notational inconveniences, results equivalent to Proposition~\ref{prop:ass CL} and Corollary~\ref{cor:CL} would still apply in this further generalisation.  
\end{remark}

\subsection{A model for cell-cell interactions}
Section~\ref{sec:classical models} showcased some direct consequences to Theorem~\ref{thm:IRD to ARD}. Of course, many real-life network models will not fall into the category of $\ARD$. Thus, this section will show how the main result can be adapted to fit an applied setting that goes beyond the $\ARD$ class. Specifically, we will show equivalence with the model in \cite{vanSantvoort2023MathematicallyMicroenvironment} to infer cell-cell interaction networks, and use this equivalence to study the existence of the giant strongly connected component.

\subsubsection{Model definition and its relation to the biological context}
When tracking the state and severity of diseases, it is important to infer how cells communicate with one another (see e.g.  \cite{Armingol2021DecipheringExpression}). There are many different strategies to accomplish this feat. Some look at a network of cell-types and infers how it evolves over time (like \cite{Wang2019CellTranscriptomics}), while others focus more on the proteins involved in the communication (like \cite{Choi2015TranscriptomeModel}). A recent model presented in \cite{vanSantvoort2023MathematicallyMicroenvironment} takes yet another approach, and infers cellular communication through a directed random graph model where vertices represent cells, vertex-types represent cell-types, and arcs represent protein pairs (ligands and receptors).

Mathematically, the model needs the following inputs:
\begin{enumerate}[label = (\arabic*)]
    \item A vertex-type distribution $T$ supported on some discrete set $\mathcal{S}$.
    \item A colour-distribution $C = (C^{\text{out}}, C^{\text{in}})$ for the arcs supported on some discrete set $\mathcal{C}^{\text{out}} \times \mathcal{C}^{\text{in}}$.
    \item An indicator function $I : \mathcal{S} \times \mathcal{C}^{\text{out}} \to \{0, 1\}$.
    \item An indicator function $J : \mathcal{S} \times \mathcal{C}^{\text{in}} \to \{0, 1\}$.
\end{enumerate}
In \cite{Zaitsev2022PreciseTranscriptomes, Dimitrov2022ComparisonData} it turns out that these four pieces of input are relatively easy to extract from patients in practice. Biologically, (1) measures how often certain cell-types are present in a tissue, (2) measures how often certain protein pairs are present in a tissue, and finally (3)-(4) indicate which proteins ``could belong to'' a given cell-type.

Using these mathematical objects, we generate a realisation of the model with vertex set $[n]$, containing $\lfloor \mu n \rfloor$ arcs for some $\mu > 0$, with the following algorithm.
\begin{enumerate}
    \item For each vertex $v \in [n]$, assign it a type $T_v$ drawn from $T$ independently of the other vertices.
    \item For each arc number $a \in [\lfloor \mu n \rfloor]$, sample an arc-colour pair $C_a = (C^{\text{out}}_a, C^{\text{in}}_a)$ drawn from $C$ independently of the other arc-colour pairs.
    \item For each arc number $a \in [\lfloor \mu n \rfloor]$ choose one vertex $v$ uniformly from the set $\left\{ v \in [n] : I(T_v, C^\text{out}_a) = 1 \right\}$.
    \item Then, independently from Step 3, choose one vertex $w$ uniformly from the set $\left\{ w \in [n] : J(T_w, C^\text{out}_a) = 1 \right\}$.
    \item Add arc $(v, w)$ to the directed graph.
\end{enumerate}

We denote the realisation of the model by $\texttt{CCI}_{n, \mu}(T, C, I, J)$. Although $\CCI$ might seem to fall in the $\ARD$ class, it is completely distinct from it. Note in $\CCI$ it is a priory unclear between which vertex-types a given arc will be drawn. Moreover, even if we were to reveal the arc-colours, then still by virtue of $I$ and $J$ it is possible that an arc is drawn between multiple combinations of vertex-types, while in $\ARD$ this would be impossible. Thus, it is impossible to directly apply Theorem~\ref{thm:IRD to ARD} to relate $\CCI$ to $\IRD$.

\subsubsection{Identifying the model as an inhomogeneous random digraphs}\label{sec:CCI is IRD statement}
To link $\CCI$ with $\IRD$, we follow Assumption~\ref{ass:type distr} and require that $\mathcal{S}, \mathcal{C}^{\text{out}}, \mathcal{C}^{\text{in}} \subseteq \mathbb{N}$. Biologically, this assumption is not far-fetched, since in reality the number of protein-types and cell-types in your body is finite. To formulate the kernel in $\IRD$ belonging to $\CCI$ we first set $p_{ij} := \prob(C^{\text{out}} = i, C^{\text{in}} = j)$ and recall that $q_k = \prob(T = k)$. We define
\begin{subequations}
\begin{equation}
    \lambda_i = \sum_{k = 1}^\infty q_k I(k, i),\label{eq:fraction of vertices connecting to lig}
\end{equation}
\begin{equation}
    \varrho_j = \sum_{k = 1}^\infty q_k J(k, j).\label{eq:fraction of vertices connecting to rec}
\end{equation}
\end{subequations}
Note that $\lambda_i$ can be interpreted as the (asymptotic) proportion of vertices that an arc with out-colour $i$ can connect to. Biologically, it is the fraction of cells that can secrete ligand-type $i$. Similarly, $\varrho_j$ can be interpreted as the fraction of vertices that an arc with in-colour $j$ can connect to. It is the fraction of cells that can express receptor-type $j$. Under these definitions, we will show that the $\IRD$-kernel belonging to $\CCI$ for two fixed vertex-types $t, s \in \mathcal{S}$ is given by
\begin{equation}
    \kappa(t, s) =\mu \sum_{i = 1}^\infty \sum_{j = 1}^\infty \frac{p_{ij} \cdot I(t, i) J(s, j)}{\lambda_i \cdot \varrho_j}. \label{eq:asymp connection number}
\end{equation}
By rewriting \eqref{eq:asymp connection number} as
\[
\frac{\kappa(t, s)}{n} = \sum_{i = 1}^\infty \sum_{j = 1}^\infty \frac{n \mu p_{ij} \cdot I(t, i) J(s, j)}{n \lambda_i \cdot n \varrho_j},
\]
we can interpret $\kappa(t, s)/n$ as the expected number of arcs that connect a specific vertex of type $t$ to another specific vertex of type $s$. This is, because $n \mu p_{ij}$ is the expected number of arcs with colour $(i, j)$, $\lambda_i n$ is the expected number of vertices that can accept out-colour $i$, and $\varrho_j n$ is the expected number of vertices that can accept in-colour $j$. Hence, for a given arc with colour $(i, j)$ the probability of it being placed between two specific vertices with type $t$ and $s$, respectively, is $1 / (n^2 \lambda_i \varrho_j)$ due to Steps 3--5 of the $\CCI$-generation algorithm. Since there are $n \mu p_{ij}$ arcs with colour $(i, j)$ that will be placed, we conclude $(n \mu p_{ij})/(n^2 \lambda_i \varrho_j)$ shall be placed between the two fixed vertices. Summing over all possible arc-colours yields the expression of $\kappa(t, s) / n$. The indicators ensure arc-colours only contribute if vertex-types $t$ and $s$ can actually accept them. 

To prove a connection between $\IRD$ and $\CCI$, we will make a chain of two links. One link from $\CCI$ to $\ARD$, and another from $\ARD$ to $\IRD$. The second link is facilitated by Theorem~\ref{thm:IRD to ARD}, the first will require a new proof. To make the first link formal, some assumptions on $\CCI$ are required.

\begin{assumption}[\CCI~assumptions]\label{ass:CCI}
    For $\texttt{CCI}_{n, \mu}(T, C, I, J)$ we assume the following:
    \begin{enumerate}[label = \Roman*.]\itemsep0em
        \item $\expec[T^{1 + \varepsilon}] < \infty$ for some $\varepsilon > 0$.
        \item $\inf_{i}\{\lambda_i : \lambda_i > 0\} > 0$ and $\inf_{j}\{\varrho_j : \varrho_j > 0\} > 0$.
    \end{enumerate}
\end{assumption}

The first assumption ensures that there are not too many unstable vertices in $\CCI$ (cf. Definition~\ref{def:stable vertices}). The second implicitly stipulates that there cannot be any arc-colours that occur with a relatively large probability, but connect to relatively few vertices. Note the second assumption is satisfied when e.g. there are only a finite number of protein types. 

We also need to make a technical assumption on $\mathcal{Q}_n$. Since Theorem~\ref{thm:IRD to ARD} sets $\Lambda_n(t, s) = 0$ if $t$ or $s$ is unstable, we need to ensure $\CCI$ probabilities do not change when arcs from and to unstable vertices are disregarded. 

\begin{assumption}[Arcs to and from unstable vertices influence nothing]\label{ass:CCI unstable vertices}
    Denote by $\texttt{CCI}^-_{n, \mu}(T, C, I, J)$ the cell-cell interaction model after removing all arcs to and from unstable vertices at some tolerance $\tau \in (0 , 1)$. We assume the event $\mathcal{Q}_n$ is such that
    \[
    \prob(\CCI_{n, \mu}(T, C, I, J) \in \mathcal{Q}_n) = \prob(\CCI_{n, \mu}^-(T, C, I, J) \in \mathcal{Q}_n) + o(1)
    \]
\end{assumption}

Note that in all biologically relevant cases Assumption~\ref{ass:CCI unstable vertices} is automatically satisfied. The number of cell-types is finite, so unstable vertex-type will not exists. However, we will show in the case of the giant strongly connected component that one can validate Assumption~\ref{ass:CCI unstable vertices} even for an infinite number of vertex-types. We can now formulate the main result of this section.

\begin{theorem}[$\CCI$ to $\IRD$]
    \label{thm:CCI to IRD}
    Consider $\texttt{CCI}_{n, \mu}(T, C, I, J)$ such that Assumption~\ref{ass:CCI} is satisfied, and let $\kappa$ be as in \eqref{eq:asymp connection number}. Define a constant $\alpha$ such that
    \begin{equation}\label{eq:alphachoice}
    \frac38 <  \alpha < \frac12. 
    \end{equation}
    Suppose $\mathcal{Q}_n$ is a monotone event that satisfies Assumption~\ref{ass:CCI unstable vertices} for some tolerance $\tau > 1 - 2\alpha$. If there exists a number $p \in [0, 1]$ such that
    \[
    \prob(\IRD_n(T, \kappa'_n) \in \mathcal{Q}_n) \to p,
    \]
    as $n \to \infty$ for all sequences $\kappa'_n $ that satisfy the inequality
\begin{equation}\label{eq:kern sequence CCI}
|\kappa'_n(t, s) - \kappa(t, s) | \leq \frac{3 n^{-1/2 + \alpha}}{\sqrt{q_t q_s}}.
\end{equation}
Then, we also have that
\[
\prob(\CCI_{n, \mu}(T, C, I, J) \in \mathcal{Q}_n) \to p.
\]
\end{theorem}
\begin{remark}
    The condition \eqref{eq:alphachoice} might seem a bit arbitrary. In the proof of Theorem~\ref{thm:CCI to IRD}, we will see that this requirement on $\alpha$ ensures there is a stability tolerance $\tau$ (cf. Definition~\ref{def:stable vertices}) for which the link between $\CCI$ and $\ARD$ can be made. We stress that \eqref{eq:alphachoice} is sufficient and might not be necessary. Moreover, note we have taken $C = 3$ in \eqref{eq:kern sequence CCI} when compared to \eqref{eq:kappa seq bound}. We will also see this is sufficient and not necessary.
\end{remark}

\begin{remark}
    As we will see in the proof of Theorem~\ref{thm:CCI to IRD}, condition II of Assumption~\ref{ass:CCI} is not strictly necessary. It can be replaced with an explicit assumption on the distribution of $C$ together with some global growth restriction on $\kappa(t, s)$ as $t ,s \to \infty$. However, to make the proofs less technical, we opted to consider only condition II. In biological contexts, the number of protein- and cell-types is finite, so that Assumption~\ref{ass:CCI} is automatically satisfied. 
\end{remark}

\subsubsection{Using the main result to find the size of its giant strongly connected component}
One of the current mathematical downsides of the model in \cite{vanSantvoort2023MathematicallyMicroenvironment}, is its heavy reliance on Monte-Carlo simulations to attain results. Theorem~\ref{thm:CCI to IRD} can help here, since it shows that monotone properties of $\IRD$ can be translated to $\CCI$. Thus, since $\IRD$ is already well studied in e.g. \cite{Cao2019OnDigraphs}, we can use existing literature to quickly compute asymptotic properties of $\CCI$ without the need for any Monte-Carlo simulation. We will show this idea by computing the asymptotic size of the largest strongly connected component (SCC) in $\CCI$. We will start by defining what SCC.

\begin{definition}[SCC]\label{def:SCC}
    A \emph{strongly connected component} of a directed graph $G = ([n], E)$ is a subset of vertices $V \subseteq [n]$ such that for every pair $v, w \in V$ there exists a path from $v$ to $w$ and back over the arcs in $E$. Specifically, we disregard vertex-types.
\end{definition}

We note in \cite{Cao2019OnDigraphs} (Theorem 3.9) that an expression for the asymptotic size of the largest strongly connect component already exists for $\IRD$. Hence, we seek to apply Theorem~\ref{thm:CCI to IRD} to translate this result into the language of $\CCI$. To do this, we first need to show that all sequences of kernels $\kappa'_n$ that adhere to \eqref{eq:kern sequence CCI} satisfy some regularity conditions (Assumption 3.1 in \cite{Cao2019OnDigraphs}). The proof will be given in Section~\ref{app:extra proofs}.

\begin{prop}[Regularity]\label{prop:IRD regularity inf}
    Let $\kappa'_n$ be a sequence adhering to \eqref{eq:kern sequence CCI}. Define in $\IRD_n(T, \kappa'_n)$ the conditional probability measure $\prob_n(\cdot) = \prob(\cdot \;|\; (T_v)_{v \in [n]})$. Then, the following four conditions are satisfied:
    \begin{enumerate}[label = (\alph*)]
        \item There exists a Borel probability measure $\nu$ on $\mathbb{N}$ such that for all $V \subseteq \mathbb{N}$ we have in probability under $\prob$ that
        \[
        \frac1n \sum_{v = 1}^n \1\{T_v \in V\} \to \nu(A).
        \]
        \item $\kappa$ is continuous and non-negative almost everywhere on $\mathbb{N}^2$.
        \item $\varphi_n(t, s) := (\kappa'_n(t, s) - \kappa(t, s))/\kappa(t, s)$ is continuous on $\mathbb{N}^2$ and converges to zero $\prob_n$-a.s. for any $t, s \in \mathbb{N}$.
        \item For the following limits we have that
        \[
        \lim_{n \to \infty } \frac1{n^2} \expec\left[ \sum_{v = 1}^n \sum_{w = 1}^n \kappa(T_v, T_w) \right] = \lim_{n \to \infty} \frac1{n^2} \expec\left[ \sum_{v = 1}^n \sum_{w \neq v} \kappa'_n(T_v, T_w) \right] = \sum_{t = 1}^\infty \sum_{s = 1}^\infty \kappa(t, s) q_t q_s < \infty.
        \]
    \end{enumerate}
\end{prop}

Proposition~\ref{prop:IRD regularity inf} shows that Theorem 3.9 of \cite{Cao2019OnDigraphs} can be adapted to $\CCI$. Theorem 3.9 additionally assumes that $\kappa$ should be irreducible (cf. Definition 3.7 in \cite{Cao2019OnDigraphs}). Translated into the language of $\CCI$ this can be formulated as follows.

\begin{assumption}[Irreducibility]\label{ass:irreducibility}
    Let $\widehat{G} = (\mathbb{N}, \widehat{E})$ be a directed graph with arc-set \[\widehat{E} := \{(t, s) \in \mathbb{N}^2 : \kappa(t, s) > 0\}.\]
    We assume that $\widehat{G}$ is strongly connected.
\end{assumption}

We can now compute the asymptotic size of the largest strongly connected component in $\CCI$. Since it is not a direct consequence of Theorem~\ref{thm:CCI to IRD}, we will provide its proof in Section~\ref{app:extra proofs}.

\begin{prop}[Largest SCC in \CCI] \label{prop:GSCC}
    Let $|\mathcal{C}_{\max}|$ be the size of the largest strongly connected component in $\CCI$, and assume $\CCI$ satisfies Assumption~\ref{ass:irreducibility}. Denote by $\pi^-_x$ the largest fixed points to the system of equations
    \begin{equation}\label{eq:pi+}
    1 - \pi^-_x(\kappa) = \exp\left\{- \sum_{t = 1}^\infty \kappa(x, t) q_t \pi^-_x(\kappa) \right\}, \qquad x \in \mathbb{N},
    \end{equation}
    and by $\pi^+_x$ the largest fixed points to the system of equations
    \begin{equation}\label{eq:pi-}
    1 - \pi^+_x(\kappa) = \exp\left\{- \sum_{t = 1}^\infty \kappa(t, x) q_t \pi^+_x(\kappa) \right\}, \qquad x \in \mathbb{N}.
    \end{equation}
    Then, we have that $|\mathcal{C}_{\max}|/n \to \alpha$ in probability, where
    \[
    \alpha = \sum_{x = 1}^\infty \pi^+_x \pi_x^- q_x.
    \]
\end{prop}
\begin{remark}
    The values $\pi_x^\pm$ can be recognised as survival probabilities of a multi-type branching process in which the number of children with type $t$ born from a parent with type $x$ is Poisson distributed with parameter $\kappa(x, t) q_t$ for $\pi_x^-$ or $\kappa(t, x) q_t$ for $\pi_x^+$.
\end{remark}

In Figure~\ref{fig:numerical experiment GSCC} we compare the result of Proposition~\ref{prop:GSCC} to the normalised size of the largest SCC one would obtain through Tarjan's algorithm (see \cite{Tarjan1972Depth-FirstAlgorithms}) numerically from realisations the cell-cell interaction model. In these numerical experiments, we consider $\CCI_{n, \mu}(T, C, I, J)$ with $n = 10000$, varying $\mu$, and input distributions/indicator functions summarised by the following vectors/matrices:
\begin{equation}\label{eq:example input}
\bm{q} = \begin{bmatrix} 0.1 \\ 0.15 \\ 0.25 \\ 0.5 \end{bmatrix}, \quad \bm{P} = \begin{bmatrix} 0.2 & 0.2 \\ 0 & 0.1 \\ 0.5 & 0 \end{bmatrix}, \quad I = \begin{bmatrix} 0 & 1 & 1 \\ 1 & 0 & 1 \\ 1 & 1 & 0 \\ 0 & 1 & 0 \end{bmatrix}, \quad J = \begin{bmatrix} 1 & 0 \\ 0 & 1 \\ 1 & 1  \\ 0 & 1 \end{bmatrix}.
\end{equation}
Here, entry $k$ in the vector $\bm{q}$ indicates $\prob(T = k)$. Similarly, entry $(i, j)$ in the matrix $\bm{P}$ indicates $\prob(C = (i, j))$. In the numerical experiments we applied a Monte-Carlo approach with 1000 instances of each $\CCI_{n, \mu}(T, C, I, J)$ have been generated, their largest SCC size recorded, and averaged to obtain the numerical largest SCC size. We have also plotted the 95\% confidence bounds. We observe that Proposition~\ref{prop:GSCC} accurately mathches the numerical results.

\begin{figure}
    \centering
    \includegraphics[scale = 0.7]{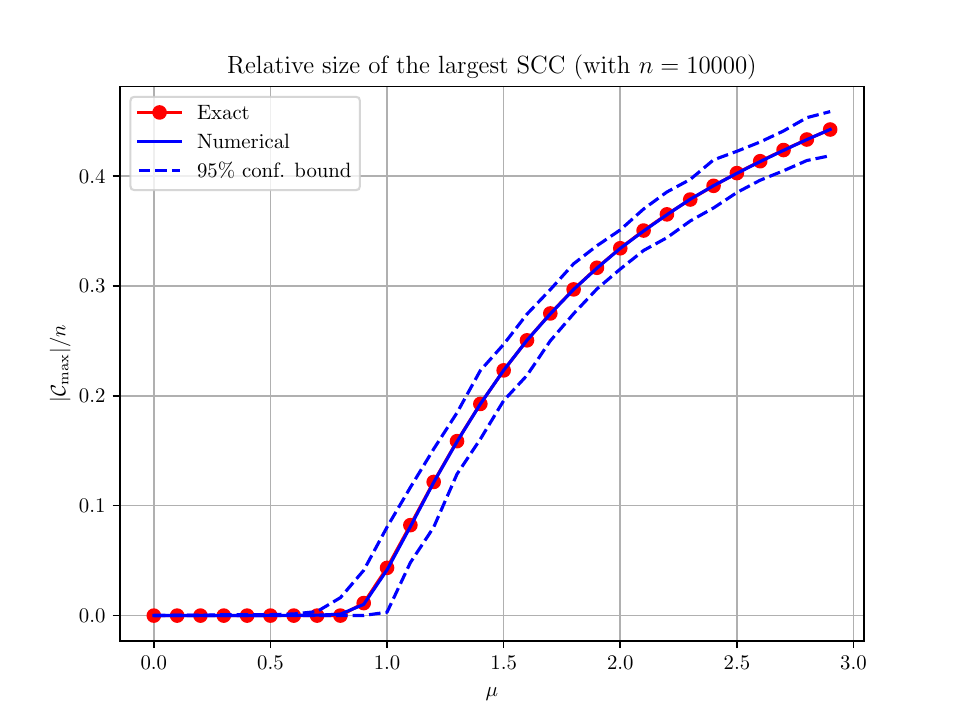}
    \caption{\emph{The size of the largest SCC computed through both Proposition~\ref{prop:GSCC} and 1000 instances of $\CCI_{n, \mu}(T, C, I, J)$. Input parameters for the model are given in \eqref{eq:example input}, and the two-sided 95\% (Monte-Carlo) confidence bounds are plotted to highlight the deviation in numerical largest SCC sizes.}}
    \label{fig:numerical experiment GSCC}
\end{figure}

\subsection{An algorithm to generate inhomogeneous random digraphs}
A final application of Theorem~\ref{thm:IRD to ARD} is that it provides simple algorithm generate realisations of $\IRD$. If one were to na\"{i}vely generate instances of on $\IRD$, first a list of vertex-types will be generated. Thereafter, we would iterate over all vertex pairs $v, w \in [n]$, look up their types $T_v, T_w \in \mathcal{S}$, and finally realise a Bernoulli random variable with parameter $\kappa(T_v, T_w) / n$ to determine whether arc $(v, w)$ will be placed in the graph. This approach will require $\mathcal{O}(n^2)$ operations to generate a realisation of $\IRD$.

The main issue with the na\"{i}ve approach to generating $\IRD$s, is the fact that $\IRD$s are sparse. Hence, a lot of time is spent on arcs that in the end will not be part of the realisation. Ideally, one would like to skip arcs that will not make it into the final realisation, to speed up the process. Our $\ARD$ model provides such an method since it fixes the number of edges upfront. By generating the arcs in one of the initial model steps, we remove the need to consider vertex pairs that ultimately have no arc placed between them. Moreover, through Theorem~\ref{thm:IRD to ARD} we know how to pick this model such that realisations are equivalent to a given $\IRD$ model. Our algorithm to generate $\ARD$ (and hence $\IRD$) is given in Section~\ref{sec:ARD}.

It should be noted that techniques have already been devised to considerably speed up graph generation of inhomogeneous random graphs (see e.g. \cite{Miller2011EfficientDegrees, Hagberg2015FastGraphs}) by skipping arcs that will not make it into the final realisation. However, while the algorithm for the Chung-Lu model~\cite{Miller2011EfficientDegrees} is exact, for general inhomogeneous graphs the algorithm also outputs a graph that is only asymptotically equivalent to the original model~\cite{Hagberg2015FastGraphs}. Moreover, the general algorithm requires a reasonable amount of computations to be implemented, as it needs to compute definite integrals involving $\kappa$ and solve related root problems. On the other hand, using the $\ARD$ only requires the formula for $\kappa$. It is therefore much more easy to implement and needs less computations to run.

If we compare our algorithm to \cite{Miller2011EfficientDegrees} in the Chung-Lu case (Section~\ref{sec:CL}), we see that it performs similar. To create the graph, we would have to execute
\[
\sum_{t = 1}^\infty \sum_{s = 1}^\infty \kappa(t, s) q_t q_s = \expec[W]
\]
samples without replacement from some set. Since these can be executed in constant time (see e.g. \cite{Vitter1987AnSampling}) we see that only $\mathcal{O}(\expec[W] n)$ random samples are needed. This is the same condition as given in \cite{Miller2011EfficientDegrees}. We do not see the extra factor $1/2$, since our Chung-Lu model is directed.

\section{Strategy, tools, and proofs of the main results}\label{sec:proof thms}
Here we will present the main strategy and tools to prove both Theorem~\ref{thm:IRD to ARD} and \ref{thm:CCI to IRD}. For both theorems we will first delve into the heuristics behind the proof, before we will outline all the technical tools needed to provide the proof. The proofs of these technical tools are postponed until Section~\ref{sec:tool proof}. Each proof section will end with a proof of its respective theorem based on the technical tools.

\subsection{Heuristics behind Theorem~\ref{thm:IRD to ARD}}\label{sec:heuristics thm 1}
To see how a parallel between an $\IRD$ and an $\ARD$ can be drawn, it is important to observe that an $\ARD$ is an $\IRD$ with the number of arcs per pair of vertex-types has been revealed. Indeed, if we know that e.g. $m$ arcs will be drawn from vertices of type $1$ to vertices of type $2$ in $\IRD$, then the only information missing is the location where the $m$ arcs will appear. Since the appearance probability is the same for each possible location (in $\IRD$ it depends only on the vertex-types involved), arcs shall be assigned through a uniform choice between the possible locations. This is a uniform choice without replacement, since each location can be chosen only once. Thus, we find ourselves in the setting of $\ARD$. To make this observation formal, we will introduce the concept of an \emph{arc-to-vertex-type function}.

\begin{definition}[Arc-to-vertex-type function]\label{def:arc to vertex type}
    In $\IRD_n(T, \kappa_n)$ we define the random arc-to-vertex-type function $A_n : \mathcal{S} \times \mathcal{S} \to \mathbb{N}$ as the function that counts the number of arcs placed between two vertex-types. Specifically, if we denote the arc-set of $\IRD_n(T, \kappa_n)$ by $\mathcal{A}_n$, then for fixed $t, s \in \mathcal{S}$ the arc-to-vertex-type function is defined as
    \[
    A_n(t, s) := \left| \{(v, w) \in \mathcal{A}_n : (T_v, T_w) = (t, s) \} \right|.
    \]
\end{definition}

In light of Definition~\ref{def:arc to vertex type}, the previous observation communicates that the law of $\IRD_n(T, \kappa_n)$ conditioned on $A_n$ is equal to the law of $\ARD(T, \Lambda_n)$ with $\Lambda_n = A_n$. The main difference between $\IRD$ and $\ARD$ now shows through the fact that $A_n$ is a \emph{random} function, while $\Lambda_n$ is not. If we seek to show that $\IRD$ and $\ARD$ are equivalent, then our hope is that the random function $A_n$ starts to become more and more deterministic as $n\to \infty$. This deterministic limit will then provide the proper scaling of $\Lambda_n$.

To find the correct scaling for $\Lambda_n$ we can be inspired by the law of large numbers. For a fixed type $t $ we know for $n$ large that $N_t \approx q_t n$. When we fix a second vertex type $s \in \mathcal{S}$, then we similarly have that $N_s \approx q_s n$. Putting both together this shows that approximately $q_t q_s n^2$ arc generation attempts in $\IRD_n(T, \kappa_n)$ will be done from a vertex of type $t$ to a vertex of type $s$. Since all these generation attempts have an independent success probability of $\kappa_n(t, s)/n$, we expect that $A_n(t, s) \approx q_t q_s \kappa_n(t, s) n$ for large $n$. Thus, if there is an equivalence between $\IRD$ and $\ARD$, we should choose $\Lambda_n(t, s) \approx q_t q_s \kappa_n(t, s) n$.

\begin{figure}
    \centering
    \includegraphics[scale = 0.5]{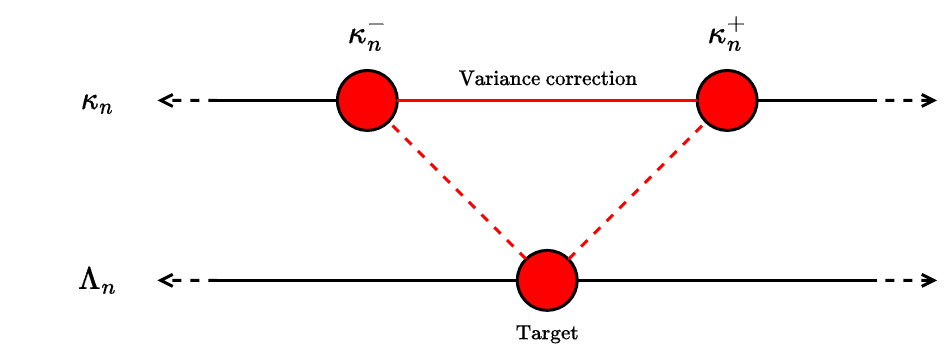}
    \caption{\emph{The idea behind the main result. When we seek to show convergence in $\ARD(T, \Lambda_n)$ for fixed $\Lambda_n$, then in $\IRD$ the same property needs to be true for a range of $\kappa_n$. This range is chosen such that with high probability $\kappa_n \in [\kappa_n^-, \kappa_n^+]$ given the arc-to-vertex-type function equal to $\Lambda_n$.}}
    \label{fig:idea main result}
\end{figure}

Although the heuristics based on the law of large numbers give a good hint at the choice of $\Lambda_n$, it hides some of the intricacies. First and foremost, it does not reveal how the variance in $A_n$ plays a role. As $n$ tends to infinity, the probability that $A_n$ equals $\Lambda_n$ for any choice of $\Lambda_n$ will become negligibly small. Thus, if any property needs to be translated from $\IRD$ to $\ARD$, its validity in $\IRD_n(T, \kappa_n)$ for one fixed $\kappa_n$ is never enough. Instead, we will require that the property is true in $\IRD_n(T, \kappa_n')$ for all $\kappa'_n$ that fall in a ``range'' from some smallest $\kappa_n^-$ to some largest $\kappa^+_n$. This range must be chosen such that almost all input $\kappa_n$ in $\texttt{IRD}_n(T, \kappa_n)$ that could produce the event $\{A_n = \Lambda_n\}$ with high probability fall between $\kappa^-_n$ and $\kappa_n^+$. See Figure~\ref{fig:idea main result} for an illustration.

Secondly, the heuristics based on the law of large numbers assumes that the asymptotic concentration is true for all vertex-types simultaneously. However, for fixed $n$ there will always be some vertex-types for which concentration has not kicked in yet. For example, if we dynamically chose the vertex-type $t_n \in \mathcal{S}$ such that $q_{t_n} \approx 1/n$, then for the number of vertices $N_{t_n}$ with type $t_n$ we roughly have that $N_{t_n} \sim \texttt{Poi}(1)$. This is more variable than the deterministic $N_{t_n} = q_{t_n} n \approx 1$ we expect. Thus, it is useful to distinguish between the vertex-types for which concentration based on the law of large numbers has started to kick in, and the vertex types that might still behave more erratic. This is why we explicitly defined stability of vertex-types (cf. Definition~\ref{def:stable vertices}).

\begin{remark}
    In Definition \ref{def:stable vertices} we need $q_t \gg 1/n$ to ensure concentration started to kick in, explaining the need for the parameter $\tau \in (0, 1)$. Moreover, keeping $\tau$ as a parameter will allow us to control how ``fast'' concentration occurs. Often, the specific choice of $\tau$ does not matter or is clear from context in which case we will omit it.
\end{remark}

Unstable vertex-types should not influence probability of a certain property being true in $\IRD$ too much. If they do, relating $\IRD$ to $\ARD$ will be impossible. Thus, when setting the aforementioned range of kernels in $\IRD$, we need to ensure that for unstable vertex-types the value zero also falls in the range. Practically, this will entail that the property's probability in $\IRD$ will stay the same irrespective of the inclusion of unstable vertex-types. Formally, this translated into Assumption~\ref{ass:kernel bound}.

The strategy to prove Theorem~\ref{thm:IRD to ARD} will exploit all previous considerations. It will employ four steps. They will coincide with the proof steps in Section~\ref{sec:proof thm I} and the main ideas behind them are given below:

\paragraph{Step I.} We fix two kernels $\kappa^+_n$ and $\kappa^-_n$ with the property that no matter the choice of $\kappa_n'$ in \eqref{eq:kappa seq bound} we have for stable vertex-types $t, s \in \mathcal{S}$ that $\kappa^-_n(t, s) \leq \kappa'_n(t, s) \leq \kappa^+(t, s)$. By using the law of total probability to condition on the realisation of $A_n(t, s)$ (cf. Definition~\ref{def:arc to vertex type}), we can ``decompose'' probabilities of $\IRD_n(T, \kappa^\pm_n)$ into probabilities of $\ARD_n(T, \Lambda_n')$ (for different values of $\Lambda_n'$) and probabilities of realisations of $A_n(t, s)$. We use monotonicity of $\mathcal{Q}_n$ to transform these decomposed probabilities in upper- and lower- bounds involving $\ARD_n(T, \Lambda_n)$ with $\Lambda_n(t, s) = \lfloor q_t q_s n \kappa(t, s) \rfloor$. 

\paragraph{Step II.} The upper- and lower bounds we get from Step I will involve error probabilities of $A_n(t, s)$ for the kernels $\kappa^\pm_n(t, s)$. We note for e.g. $\kappa^+_n$ that $A_n(t, s) \sim \texttt{Bin}(N_t N_s, \kappa^+(t, s)/n)$. For stable vertex-types, we can use the concentration of $N_t$ and $N_s$ that has started to kick in to show that $A_n(t, s) \approx \texttt{Bin}(q_t q_s n^2, \kappa^+(t, s)/n)$, and use this to conclude that the error probabilities converge to zero. For unstable vertices, we will use the realisation that their inclusion does not alter the final result of the calculations (cf. Assumption~\ref{ass:kernel bound}) to remove error terms involving them.

\subsection{Proof of Theorem~\ref{thm:IRD to ARD}}\label{sec:proof thm I}
As outlined in the strategy, an important characteristic of the proof is the distinction between stable and unstable vertex-types. Not all vertices are important to us, but as $n \to \infty$ we want all vertex types to start playing a role. Therefore, it is important to characterise approximately how may stable vertex-types there are. This is covered by the following lemma.

\begin{lemma}[Amount of stable vertex-types]\label{lem:no stable} 
    Recall Definition~\ref{def:stable vertices} and suppose that $\expec[T^\delta] < \infty$ for some $\delta > 0$. Then, for $n$ large we have \[ u_n^\uparrow(\tau) \leq \left\lceil n^{(1 - \tau)/(1 + \delta)} \right\rceil.\]
\end{lemma}

A helpful property of stable vertex-types, as highlighted in the heuristics, is that they concentrate around their mean. To make this precise, we will define what this means for pairs of stable vertex-types. We will call pairs of vertex-types that started their concentration to be \emph{well-concentrated}. 

\begin{definition}[Well-concentrated vertex-types]\label{def:well-concentrated types}
    Given two vertex types $t, s \in \mathcal{S}$ we say that they are well-concentrated around their mean if the following event occurs:
    \[
    \mathcal{V}_{ts} := \left\{|N_t - q_t n | \leq  \log(n) \sqrt{q_t n}\right\} \cap \left\{|N_s - q_s n | \leq \log(n) \sqrt{q_s n} \right\}.
    \]
\end{definition}

We will see below that $\mathcal{V}_{ts}$ occurs with high probability for every pair of vertex-types. This highlights what we meant in the heuristics with ``concentration starting to kick in'' when a vertex type is stable. For stable vertex types we have that $\log(n) \sqrt{q_t n} \ll q_t n$, while for unstable vertex-types we have that $\log(n) \sqrt{q_t n} \gg q_t$. In other words, unstable vertex-types are expected to exhibit great variability in the sense that their count could be significantly far away from their mean.

\begin{lemma}[Vertices are well-concentrated]\label{lem:well-concentrated types}
    Fix two vertex-types $t, s \in \mathcal{S}$. For any $r > 0$ we have for $n$ large that $\prob(\neg\mathcal{V}_{ts}) \leq 2 \exp(-\log(n)^2 / 2)$.
\end{lemma}

When vertex-types are well-concentrated and stable, the approximation $N_t N_s \approx n^2 q_t q_s$ is valid. Thus, if we fix a kernel $\kappa_n'$ then we know in $\IRD_n(T, \kappa_n')$ that $A_n(t, s) \sim \texttt{Bin}(N_t N_s, \kappa_n'/n) \approx \texttt{Bin}(n^2 q_t q_s, \kappa_n'(t, s)/n)$. This will in turn allow us to show that $A_n(t, s) \approx n q_t q_s \kappa_n'(t, s)$, meaning that the heuristics based on averages is valid. We can translate this into the following two lemmas:

\begin{lemma}[Undershoot mixed binomials]\label{lem:undershoot mixed binomials}
    Fix two stable vertex-types $t, s \in \mathcal{S}$ at tolerance $\tau \in (0, 1)$ and set $\Lambda_n(t, s) = \lfloor n q_t q_s \kappa(t, s) \rfloor$ for some kernel $\kappa$. Let $\kappa_n'$ be a different function, and suppose for it there exists an $\alpha \in (1/2 - \tau /2, 1/2)$ and a constant $C > 0$ such that
    \[
    \kappa'_n(t, s) \geq \kappa(t, s) + \frac{C n^{-1/2 + \alpha}}{\sqrt{q_t q_s}}.
    \]
    Then, we have for $n$ large that
    \[
    \prob(A_n(t, s) < \Lambda_n(t, s)) \leq 2 \exp(-\log(n)^2 / 2).
    \]
\end{lemma}

\begin{lemma}[Overshoot mixed binomials]\label{lem:overshoot mixed binomials}
    Fix two stable vertex-types $t, s \in \mathcal{S}$ at tolerance $\tau \in (0, 1)$ and set $\Lambda_n(t, s) = \lfloor n q_t q_s \kappa(t, s) \rfloor$ for some kernel $\kappa$. Let $\kappa_n'$ be a different function, and suppose for it there exists an $\alpha \in (1/2 - \tau /2, 1/2)$ and a constant $C > 0$ such that
    \[
    \kappa'_n(t, s) \leq \kappa(t, s) - \frac{C n^{-1/2 + \alpha}}{\sqrt{q_t q_s}}.
    \]
    Then, we have for $n$ large that
    \[
    \prob(A_n(t, s) > \Lambda_n(t, s)) \leq 2 \exp(-\log(n)^2 / 2).
    \]
\end{lemma}

The final piece of the puzzle that is missing, is the influence of monotone events. As argued in the heuristics, monotonicity of $\mathcal{Q}_n$ will provide upper- and lower-bounds on probabilities involving $\ARD$. Thus, we will need a result that actually show these events are able to give us bounds. The following lemma will tackle this.

\begin{lemma}[Monotonicity in \ARD]\label{lem:monotonicity ARD}
    Let $\mathcal{Q}_n$ be a monotone event and let $\Lambda_n$ and $\Lambda_n'$ be two functions such that for all $t, s \in \mathcal{S}$ we have that $\Lambda_n(t, s) \leq \Lambda_n'(t, s)$. Then, we have the following:
    \begin{enumerate}
        \item If $\mathcal{Q}_n$ is increasing, then $\prob(\ARD_n(T, \Lambda_n) \in \mathcal{Q}_n) \leq \prob(\ARD_n(T, \Lambda_n') \in \mathcal{Q}_n)$.
        \item If $\mathcal{Q}_n$ is decreasing, then $\prob(\ARD_n(T, \Lambda_n) \in \mathcal{Q}_n) \geq \prob(\ARD_n(T, \Lambda_n') \in \mathcal{Q}_n)$.
    \end{enumerate}
\end{lemma}

We are now in a position to proof Theorem~\ref{thm:IRD to ARD}. We will only give the proof for $\mathcal{Q}_n$ that are increasing. The proof for decreasing events will be analogous. The proofs of all the lemmas above are postponed until Section~\ref{sec:tool proof thm I}.

\begin{proof}[\textbf{Proof of Theorem~\ref{thm:IRD to ARD}}]
    Let $\mathcal{Q}_n$ be some increasing event, and define two kernels $\kappa_n^\pm$ such that for $t, s \in \mathcal{S}$ we have
    \[
    \kappa_n^\pm(t, s) = \begin{cases}\max\left\{ \kappa(t, s) \pm C n^{-1/2 + \alpha} / \sqrt{q_t q_s}, 0 \right\}, & \text{if } t \text{ and } s \text{ are stable},\\
    0, & \text{ else.} \end{cases}
    \]
Recall that the target arc-count function $\Lambda_n$ is given by \[\Lambda_n(t, s) = \begin{cases} \lfloor \kappa(t, s) q_t q_s n\rfloor, & \text{if } t \text{ and } s \text{ is stable,}\\ 0, & \text{else.} \end{cases}\] The proof will consist of four steps.
\begin{itemize}
    \item[Ia.] We show that $\prob(\ARD_n(T, \Lambda_n) \in \mathcal{Q}_n) \leq \prob(\IRD_n(T, \kappa_n^+) \in \mathcal{Q}_n) + \xi_n$ using Lemma~\ref{lem:monotonicity ARD}. Here, $\xi_n$ is an error term involving mixed-binomial deviations.
    \item[Ib.] We show that $\prob(\IRD_n(T, \kappa_n^-) \in \mathcal{Q}_n) \leq \prob(\ARD_n(T, \Lambda_n') \in \mathcal{Q}_n) + \zeta_n$ again using Lemma~\ref{lem:monotonicity ARD}. Here, $\zeta_n$ is another error-term involving mixed-binomial deviations.
    \item[IIa.] We show that $\xi_n \to 0$ using Assumption~\ref{ass:kernel bound}, Lemma~\ref{lem:no stable} and Lemma~\ref{lem:undershoot mixed binomials}.
    \item[IIb.] We show that $\zeta_n \to 0$ using Assumption~\ref{ass:kernel bound}, Lemma~\ref{lem:no stable}, and Lemma~\ref{lem:overshoot mixed binomials}.
\end{itemize}
Together, Step Ia and IIa show with the convergence assumption on $\IRD_n(T, \kappa_n^+)$ and Assumption~\ref{ass:kernel bound} on $\kappa$ that
\[
\limsup_{n\to\infty} \prob(\ARD_n(T, \Lambda_n) \in \mathcal{Q}_n) \leq \limsup_{n\to \infty} \prob(\IRD_n(T, \kappa_n^+) \in \mathcal{Q}_n) + \limsup_{n \to \infty}\xi_n = p + 0.
\]
Similarly, Steps Ib and IIb show with the convergence assumption on $\IRD_n(T, \kappa_n^-)$ and Assumption~\ref{ass:kernel bound} on $\kappa$ that
\[
\liminf_{n \to \infty} \prob(\IRD_n(T, \kappa^-_n) \in \mathcal{Q}_n) \leq \liminf_{n \to \infty} \prob(\ARD_n(T, \Lambda_n) \in \mathcal{Q}_n) + \liminf_{n \to \infty} \zeta_n,
\]
so that
\[
p \leq \liminf_{n \to \infty} \prob(\ARD_n(T, \Lambda_n) \in \mathcal{Q}_n) \leq \limsup_{n\to \infty}\prob(\ARD_n(T, \Lambda_n) \in \mathcal{Q}_n) \leq p.
\]
This shows the desired result if Steps I--II are true. We will now show this.

\paragraph{Step Ia.} We will set $A_n^+$ to be the arc-to-vertex-types function of $\IRD_n(T, \kappa_n^+)$. We will use the law of total probability to integrate over all possible realisations of $A_n^+$. 
\[
\prob(\IRD_n(T, \kappa^+_n) \in \mathcal{Q}_n) = \sum_{\Lambda_n'} \prob(\IRD_n(T, \kappa^+_n) \in \mathcal{Q}_n | A_n^+ = \Lambda_n') \prob\left( \bigcap_{t, s \in \mathcal{S}} \{A^+_n(t, s) = \Lambda_n'(t, s)\} \right)
\]
We will now define the set $\mathcal{L}^+_n := \{\Lambda'_n : \Lambda_n'(t, s) \geq \Lambda_n(t, s) \text{ for all } t, s \in \mathcal{S}\}$, and bound the sum by only considering $\Lambda_n'$ that fall in the set $\mathcal{L}_n^+$. We will also use the fact that $\IRD$ conditioned on $A_n^+$ equals $\ARD$. This yields
\[
\prob(\IRD_n(T, \kappa^+_n) \in \mathcal{Q}_n ) \geq \sum_{\Lambda_n' \in \mathcal{L}_n^+} \prob(\ARD_n(T, \Lambda_n') \in \mathcal{Q}_n ) \prob\left( \bigcap_{t,s \in \mathcal{S}} \{A^+_n(t, s) = \Lambda_n'(t, s)\} \right).
\]
In $\ARD_n(T, \Lambda_n')$ note that the value of $\Lambda_n'(t, s)$ does not matter when either $t$ or $s$ is a vertex type that does not appear in the graph. These excess arcs will be deleted in the end anyways when one generates $\ARD$. Now, we can use the fact that $\mathcal{Q}_n$ is increasing together with Lemma~\ref{lem:monotonicity ARD} to further lower bound the first probability in the product by considering $\ARD_n(T, \Lambda_n)$ instead of $\ARD_n(T, \Lambda_n')$.
\begin{align}
\prob(\IRD_n(T, \kappa^+_n) \in \mathcal{Q}_n) &\geq \prob(\ARD_n(T, \Lambda_n) \in \mathcal{Q}_n )  \sum_{\Lambda_n' \in \mathcal{L}_n^+} \prob\left(\bigcap_{t, s \in \mathcal{S}} \{A^+_n(t, s) = \Lambda_n'(t, s)\}  \right), \nonumber \\
&= \prob(\ARD_n(T, \Lambda_n) \in \mathcal{Q}_n )  \prob\left( \bigcap_{t, s \in \mathcal{S}} \{A^+_n(t, s) \geq \Lambda_n(t, s)\}  \right). \label{eq:ARD vs arc gen split}
\end{align}
Note that \eqref{eq:ARD vs arc gen split} is the split we alluded to in the strategy. We will now continue to lower bound the probability involving the arc-to-vertex-type function. To do this, we will first rewrite this probability by using the complement rule, de Morgan's laws, and the union bound
\[
\prob(\IRD_n(T, \kappa^+_n) \in \mathcal{Q}_n) \geq \prob(\ARD_n(T, \Lambda_n) \in \mathcal{Q}_n ) - \sum_{t = 1}^\infty \sum_{s = 1}^\infty\prob\left(A_n^+(t, s) < \Lambda_n(t, s)  \right).
\]
Finally, moving the error-terms to the other side of the inequality yields the result we sought to obtain from this step.
\begin{equation}\label{eq:proof I step I}
    \prob(\ARD_n(T, \Lambda_n) \in \mathcal{Q}_n ) \leq \prob(\IRD_n(T, \kappa^+_n) \in \mathcal{Q}_n) + \underbrace{\sum_{t = 1}^{\infty} \sum_{s = 1}^{\infty}\prob\left(A_n^+(t, s) < \Lambda_n(t, s) \right) }_{\xi_n} .
\end{equation}

\paragraph{Step Ib.} Similar to Step Ia we use the law of total probability to condition on all possible realisations of the arc-to-vertex-type function. For $\IRD_n(T, \kappa^-_n)$ we will denote the arc-to-vertex-type function by $A_n^-$. This yields
\begin{equation}\label{eq:law total prob step 2 thm 1}
   \prob(\IRD_n(T, \kappa^-_n) \in \mathcal{Q}_n) = \sum_{\Lambda_n'} \prob(\IRD_n(T, \kappa^-_n) \in \mathcal{Q}_n \;|\; A_n^- = \Lambda_n') \prob\left( \bigcap_{t, s} \{A^-_n(t, s) = \Lambda_n'(t, s) \} \right). 
\end{equation}
The idea of this proof step is now to start inductively chiselling away individual vertex-type pairs from \eqref{eq:law total prob step 2 thm 1} to find the desired bound. We will show the main approach by considering the induction base $t = s = 1$. In this base case, split up the sum into the part where $\Lambda'_n(1, 1) \leq \Lambda_n(1, 1)$ and the part where $\Lambda'_n(1, 1) > \Lambda_n(1, 1)$. In the computation below, we will also apply the link between $\ARD$ and $\IRD$.
\begin{align*}
     \prob(\IRD_n(T, \kappa^-_n) \in \mathcal{Q}_n) = &\sum_{\Lambda'_n(1, 1) \leq \Lambda_n(1, 1)} \prob(\ARD_n(T, \Lambda_n') \in \mathcal{Q}_n ) \prob\left( \bigcap_{t, s} \{A^-_n(t, s) = \Lambda_n'(t, s) \} \right)\\
     &+\sum_{\Lambda'_n(1, 1) > \Lambda_n(1, 1)} \prob(\ARD_n(T, \Lambda_n') \in \mathcal{Q}_n) \prob\left( \bigcap_{t, s} \{A^-_n(t, s) = \Lambda_n'(t, s) \} \right).
\end{align*}
For the second sum, we can simply bound the $\IRD$ probability by one to find
\begin{align*}
     \prob(\IRD_n(T, \kappa^-_n) \in \mathcal{Q}_n) \leq &\sum_{\Lambda'_n(1, 1) \leq \Lambda_n(1, 1)} \prob(\ARD_n(T, \Lambda_n') \in \mathcal{Q}_n) \prob\left( \bigcap_{t, s} \{A^-_n(t, s) = \Lambda_n'(t, s) \} \right)\\
     &+ \prob\left(A^-_n(1, 1) > \Lambda_n(1, 1)  \right).
\end{align*}
Next, we show the idea of the induction step by considering the case $t = 1$ and $s = 2$. Like before, we split up the sum in the part where $\Lambda'_n(1, 2) \leq \Lambda_n(1, 2)$ and the part where $\Lambda'_n(1, 2) > \Lambda_n(1, 2)$.
\begin{align*}
     \prob(\IRD_n(T, \kappa^-_n) \in \mathcal{Q}_n) \leq &\sum_{\substack{\Lambda'_n(1, 1) \leq \Lambda_n(1, 1) \\ \Lambda'_n(1, 2) \leq \Lambda_n(1, 2) }} \prob(\ARD_n(T, \Lambda_n') \in \mathcal{Q}_n) \prob\left( \bigcap_{t, s} \{A^-_n(t, s) = \Lambda_n'(t, s) \} \right)\\
     &+ \sum_{\substack{\Lambda'_n(1, 1) \leq \Lambda_n(1, 1) \\ \Lambda'_n(1, 2) > \Lambda_n(1, 2) }} \prob(\ARD_n(T, \Lambda_n') \in \mathcal{Q}_n) \prob\left( \bigcap_{t, s} \{A^-_n(t, s) = \Lambda_n'(t, s) \} \right)\\
     &+ \prob\left(A^-_n(1, 1) > \Lambda_n(1, 1)  \right).
\end{align*}
In the second sum, we will add all the removed terms involving $\Lambda_n'(1, 1)$ again, creating an upper-bound. Then, as in the base case, we bound the $\IRD$ probability by one to remove the sum. We find
\begin{align*}
     \prob(\IRD_n(T, \kappa^-_n) \in \mathcal{Q}_n) \leq &\sum_{\substack{\Lambda'_n(1, 1) \leq \Lambda_n(1, 1) \\ \Lambda'_n(1, 2) \leq \Lambda_n(1, 2) }} \prob(\ARD_n(T, \Lambda_n') \in \mathcal{Q}_n) \prob\left( \bigcap_{t, s} \{A^-_n(t, s) = \Lambda_n'(t, s) \} \right)\\
     &+ \prob(A_n^-(1, 2) > \Lambda_n(1, 2)) + \prob\left(A^-_n(1, 1) > \Lambda_n(1, 1)  \right).
\end{align*}
Similar to Step Ia, now define the set $\mathcal{L}_n^- := \{\Lambda_n' : \Lambda_n'(t, s) \leq \Lambda_n(t, s) \text{ for all } t, s \in \mathcal{S}\}$. We can repeat the previous argumentation for all pairs of vertex-types to find
\begin{align*}
     \prob(\IRD_n(T, \kappa^-_n) \in \mathcal{Q}_n) \leq &\sum_{\Lambda_n' \in \mathcal{L}_n^-} \prob(\ARD_n(T, \Lambda_n') \in \mathcal{Q}_n) \prob\left( \bigcap_{t, s} \{A^-_n(t, s) = \Lambda_n'(t, s) \} \right)\\
     &+ \sum_{t = 1}^\infty \sum_{s = 1}^\infty \prob\left(A^-_n(t, s) > \Lambda_n(t, s)  \right).
\end{align*}
Now, we apply Lemma~\ref{lem:monotonicity ARD} to upper bound the $\ARD$ probability by noting for all $\Lambda_n' \in \mathcal{L}_n^-$ that $\Lambda_n$ is greater or equal to it for all vertex-types. After bounding, we take the $\ARD$ probability out of the sum, and note that the resulting sum bounded by one. We find the desired outcome of this step.
\[
 \prob(\IRD_n(T, \kappa^-_n) \in \mathcal{Q}_n) \leq \prob(\ARD_n(T, \Lambda_n) \in \mathcal{Q}_n) + \underbrace{\sum_{t = 1}^\infty \sum_{s = 1}^\infty \prob\left(A^-_n(t, s) > \Lambda_n(t, s)  \right)}_{\zeta_n}.
\]

\paragraph{Step IIa.} To show that $\xi_n$ converges to zero, we first recall that $A_n^+(t, s) \geq 0$. Thus, through the definition of $\Lambda_n(t, s)$ we have for every $t > u_n^\uparrow$ or $s > u_n^\uparrow$ that
\[
\prob\left(A_n^+(t, s) < \Lambda_n(t, s) \right) = \prob\left(A_n^+(t, s) < 0 \right) = 0.
\]
Hence, $\xi_n$ simplifies into
\[
\xi_n = \sum_{t = 1}^{u_n^\uparrow} \sum_{s = 1}^{u_n^\uparrow}\prob\left(A_n^+(t, s) < \Lambda_n(t, s) \right).
\]
Thus, through Lemma~\ref{lem:undershoot mixed binomials} we have for any $r > 0$ that
\[
\xi_n \leq 2 u_n^\uparrow(\tau)^2 \exp(-\log(n)^2 / 2).
\]
Now, by applying Lemma~\ref{lem:no stable} we find for some $\delta > 0$ that
\[
4 u_n^\uparrow(\tau)^2 \exp(-\log(n)^2 / 2) \leq 4 n^{2(1 - \tau)/(1 + \delta)} \exp(-\log(n)^2 / 2) .
\]
Thus, we find indeed that $\xi_n \to 0$.

\paragraph{Step IIb.} We note from our definition of $\kappa^-_n$ that for an unstable vertex type $t$ or $s$ we have that $\kappa^-_n(t, s) = 0$. This implies that $A_n^-(t, s) = 0$ as well, meaning
\[
\prob(A_n^-(t, s) > \Lambda_n(t, s)) = \prob(A_n^-(t, s) > 0) = 0.
\]
Thus, we may write 
\[
\zeta_n = \sum_{t = 1}^{u_n^\uparrow} \sum_{s = 1}^{u_n^\uparrow} \prob(A_n^-(t, s) > \Lambda_n(t, s)).
\]
Similar to Step IIa, we will apply the bound from Lemma~\ref{lem:overshoot mixed binomials} and \ref{lem:no stable} to obtain
\[
\zeta_n \leq 4 n^{2(1 - \tau)/(1 + \delta)} \exp(-\log(n)^2 / 2) .
\]
Hence, we also have that $\zeta_n \to 0$. Taking Step I and II together, we have found the desired result.
\end{proof}

\begin{remark}
    Note that Lemma~\ref{lem:well-concentrated types} does not play a direct role in the proof of Theorem~\ref{thm:IRD to ARD}. This is because it is a prerequisite for the proof of Lemma~\ref{lem:undershoot mixed binomials} and \ref{lem:overshoot mixed binomials}. 
\end{remark}

\begin{remark}
    In the proof of Theorem~\ref{thm:IRD to ARD} we saw that the difficulty of each step is present in different places. For example, we saw that the strategy to derive the upper bound in Step Ia needed much more careful use of the law of total probability, while in Step Ib the difficulty lied in carefully cutting away the correct error-terms after applying the law of total probability. If we were to also provide the explicit proof for decreasing $\mathcal{Q}_n$, all difficulties would flip. So, the careful cutting would happen in Step Ia while the careful application of the law of total probability would happen in Step Ib.
\end{remark}

\begin{remark}
    Theorem~\ref{thm:IRD to ARD} shows how a result from $\IRD$ can be translated to $\ARD$. One can formulate a similar theorem that would show how results from $\ARD$ can be translated to $\IRD$. For the Erd\H{o}s-R\'enyi and Gilbert model, theorems for both directions of equivalence can be found in \cite{Janson2000RandomGraphs}. The proof for translation from $\ARD$ to $\IRD$ would use similar techniques as the proof of Theorem~\ref{thm:IRD to ARD}. Like the equivalent result in \cite{Janson2000RandomGraphs}, we expect monotonicity of $\mathcal{Q}_n$ is not required anymore if we tranlsate from $\ARD$ to $\IRD$.
\end{remark}

\subsection{Heuristics behind Theorem~\ref{thm:CCI to IRD}}\label{sec:heuristics II}
We want to relate $\CCI$ to $\IRD$ using Theorem~\ref{thm:IRD to ARD}. However, recall that in $\CCI$ there is no knowledge on which arcs are assigned to which vertex-types. We only know something about the potential vertex-types each arc can be placed in between. Thus, if we would like to apply Theorem~\ref{thm:IRD to ARD} for $\CCI$, we would need to reveal the vertex-types each arc in $\CCI$ is going to be placed in between. Then, we are back to the setting of $\ARD$, meaning our main result can be applied. To this end, it is instructive to introduce a random function $\bar{A}_n : \mathcal{S} \times \mathcal{S} \to \mathbb{N}$ that counts in $\CCI$ the amount of arcs that are placed in between two given vertex-types. We will call this function the \emph{vertex-type arc count function}.

\begin{definition}[Vertex-type arc count function]\label{def:vertex type arc count}
    In $\CCI_{n, \mu}(T, C, I, J) = ([n], E)$, after removing self-loops and multi-arcs, we define the function $\bar{A}_n : \mathcal{S} \times \mathcal{S} \to \mathbb{N}$ for two fixed vertex-types $t, s \in \mathcal{S}$ as
    \[
    \bar{A}_n(t, s) := \left| (v, w) \in E : T_v = t \text{ and } T_w = s   \right|.
    \]
\end{definition}

By noting that arcs are placed between vertices uniformly, and by noting that uniform distributions conditioned on a subset of their support are still uniform, we can conclude that $\CCI$ conditioned on $\bar{A}_n$ is equivalent to $\ARD$. Hence, the key to connecting $\CCI$ to $\IRD$ will be to first show $\bar{A}_n$ concentrates, and then applying Theorem~\ref{thm:IRD to ARD}. 

Similar to the heuristics in Section~\ref{sec:heuristics thm 1}, we want to show that $\bar{A}_n$ concentrates around its mean. Recall from Section~\ref{sec:CCI is IRD statement} that $\kappa(t, s) / n$, with $\kappa$ given by \eqref{eq:asymp connection number}, can be interpreted as the expected number of arcs that connect a specific vertex of type $t$ with a specific vertex of type $s$. Hence, by noting that there will be roughly $n q_t$ vertices with type $t$ and $n q_s$ vertices with type $s$, we can deduce from this that we expect $\kappa(t, s) n^2 q_t q_s$ arcs to be placed from vertices of type $t$ to vertices of type $s$. Thus, to apply Theorem~\ref{thm:IRD to ARD} we will show that
\begin{equation}\label{eq:A bar expec}
\bar{A}_n(t, s) \approx \sum_{i = 1}^\infty \sum_{j = 1}^\infty \frac{n p_{ij} \mu \cdot n q_t \cdot n q_s \cdot I(t, i) J(s, j)}{n\lambda_i \cdot n\varrho_j} = n \kappa(t, s) q_t q_s \approx \lfloor n \kappa(t, s) q_t q_s \rfloor. 
\end{equation}

To show the above concentration, though, there is one detail that deserves special attention. Looking at Defintion~\ref{def:vertex type arc count} we see that $\bar{A}_n$ counts the arcs between two vertex-types \emph{after removing self-loops and multi-arcs}. This highlights another big difference between $\CCI$ and $\ARD$ we will have to cope with: $\ARD$ will not produce self-loops and multi-arcs, while $\CCI$ might. The connection between the two can only be made once we remove these, and hence we will need to show that $\CCI$ after erasing self-loops and multi-arcs still approximately equals $\CCI$ before doing so.

\begin{figure}
    \centering
    \includegraphics[scale = 0.5]{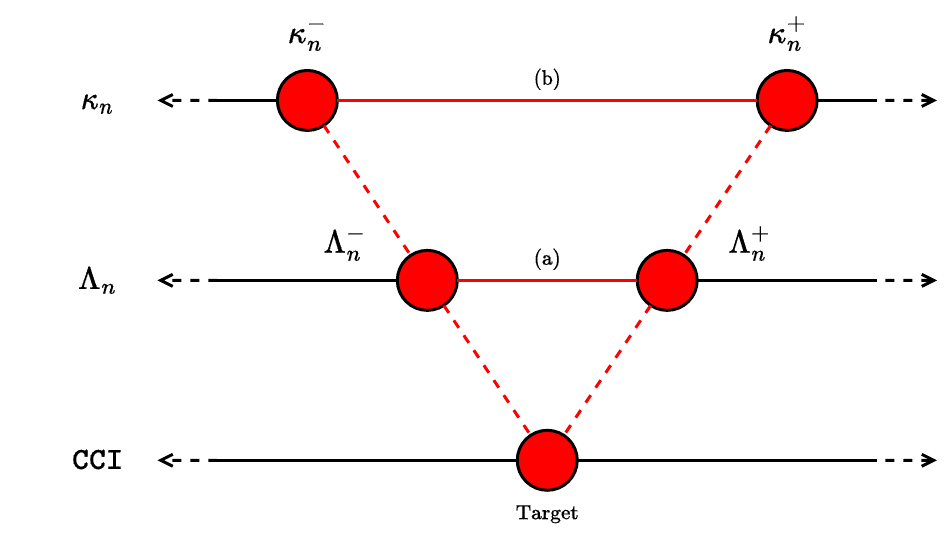}
    \caption{\emph{The heuristics behind the proof of Theorem~\ref{thm:CCI to IRD}. We want to identify $\CCI$ with given parameters as $\IRD$, so we will have to (a) remove self-loops and multi-edges, and assign arcs to vertex-types. This brings us in the setting of $\ARD$. Thereafter, we will have to (b) apply Theorem~\ref{thm:IRD to ARD} to get into the setting of $\IRD$.}}
    \label{fig:proof idea II}
\end{figure}

Figure~\ref{fig:proof idea II} showcases the structure of the proof based on the previous discussion. In essence, the new tools that we will develop create the link between $\CCI$ and $\ARD$. In this link, the concept of well-concentrated vertices will be used slightly differently. This is needed, since we lose a big benefit in $\CCI$ that $\IRD$ and $\ARD$ both had. In these two models we could consider the graph generation process for each pair of vertex-types independently of the other pairs. This is why the concept of well-concentrated vertices was defined in terms of pairs of two fixed vertex-types. However, in $\CCI$ it is possible that one arc type could possible connect to all vertex-types simultaneously. Thus, if vertices are well-concentrated in this setting, then a significantly big number of vertex-types should be close around their mean \emph{at the same time}. The strategy to prove Theorem~\ref{thm:CCI to IRD} will have the following steps. These are mirrored in the actual proof in Section~\ref{sec:proof II}.

\paragraph{Step I.} We show that the realisation of $\CCI$ restricted to \emph{super}-stable vertex-types is similar to the $\ARD$ model restricted to \emph{super}-stable vertex-types. 

\paragraph{Step II.} We show that the realisation of $\CCI$ restricted to stable vertex-types is similar to the $\ARD$ model restricted to stable vertex-types.

\paragraph{Step III.} We show that every instance of $\ARD$ we have created has an input function $\Lambda_n$ for which the corresponding form of $\kappa_n$ in Theorem~\ref{thm:IRD to ARD} adheres to \eqref{eq:kappa seq bound}.

\subsection{Proof of Theorem~\ref{thm:CCI to IRD}}\label{sec:proof II}
To prove Theorem~\ref{thm:CCI to IRD}, we will often rely on direct consequence of Assumption~\ref{ass:CCI}, which is that the kernel belonging to $\CCI$ is bounded. Although not necessarily needed to derive our results (cf. the remark after Theorem~\ref{thm:CCI to IRD}), this fact will greatly simplify the proofs of the other techniques we will use, since we always just replace kernel values by a constant.
\begin{lemma}\label{lem:kappa bounded}
    Under Assumption~\ref{ass:CCI} the kernel $\kappa$ in \eqref{eq:asymp connection number} is bounded.
\end{lemma}

As alluded to in Section~\ref{sec:heuristics II}, we will need to refine our notion of being well concentrated. Specifically, we will require that well-concentration is true for all vertex-types simultaneously until some upper-bound vertex-type that scales with $u_n^\uparrow$ (cf. Definition~\ref{def:stable vertices}). Like in Definition~\ref{def:well-concentrated types}, this expanded notion of well-concentration will be satisfied with high probability in our models, since for for each individual vertex-type well-concentration is true with super-polynomial probability (cf. Lemma~\ref{lem:well-concentrated types}), and since the number of stable vertex-types grows only polynomially (cf. Lemma~\ref{lem:no stable}). 

Apart from upgrading our notion of concentration, it will also be helpful to upgrade our notion of stability. Specifically, we will define a class of vertex-types that adhere to stricter stability requirements, which we will need for some of our proofs. We will call these vertex-types \emph{super stable}. Their definition will be based on Definition~\ref{def:stable vertices}.

\begin{definition}[Super stable vertices]\label{def:super stable vertices}
    Recall the definition of stable vertex-types (cf. Definition~\ref{def:stable vertices}), and note particularly its influence on the tolerance $\tau$. We say a vertex-type is super stable when $\tau > 1/2$. 
\end{definition}

We will see that for super stable vertex-types there are relatively few self-loops and multi-arcs in $\CCI$. Moreover, like for the previous concept of stability, we will show that the occurrence of vertex-types that are not super stable is rather rare. When analysing $\CCI$, this fact will allow us to split up the graph into super-stable vertices which exhibit nice properties, and other vertices that occur so little that they might as well be removed from the graph. The extra result on stability that states vertices which are not super stable occur only a little is given below.

\begin{lemma}[Probability of unstable vertex-types]\label{prob:unstable}
    Suppose that $\expec[T^\delta] < \infty$ for some $\delta > 0$ Then, for $n$ large and all $r \in (0, \delta)$ there exists a constant $\widehat{C}_r > 0$ such that \[\prob(T > u_n^\uparrow(\tau)) \leq \widehat{C}_r \cdot n^{\frac{(\tau - 1)(\delta - r)}{(1 + \delta)}}.\]
    In other words, vertex-types that are not stable at tolerance $\tau$ occur with only vanishing probability.
\end{lemma}

With the concept of super stability we can compute the number of self-loops (connections from a vertex to itself) and multi-arcs (more than one of the same directed connection between two vertices). We will show that both the number of self-loops and multi-arcs within each pair of vertex-types is sub-linear. Specifically, if we consider super stable vertex-types at tolerance $\tau > 1/2$, then we will show that the number of self-loops and multi-arcs is bounded by $n^{2(1 - \tau)}$. Observe that this bound becomes trivial when $\tau \leq 1/2$, since $\CCI$ will have $\lambda n$ arcs.

\begin{lemma}[Number of self-loops]\label{lem:self-loops}
    Fix a $\tau \in (0, 1)$. Denote by $S_t$ the number of self-loops on vertices with type $t \in \mathcal{S}$. We have for all $r > 0$ and $\nu < \tau$ that
    \[
    \prob\left( \bigcup_{t = 1}^{u_n^\uparrow(\tau)} \{S_t > n^{1 - \nu}\} \right) \leq 1/n^r.
    \]
\end{lemma}

\begin{remark}
    Note that for Lemma~\ref{lem:self-loops} super stability is not required.
\end{remark}

\begin{lemma}[Number of multi-arcs]\label{lem:multi-arcs}
    Fix a $\tau \in (1/2, 1)$. Denote by $M_{ts}$ the number of multi-arcs from vertices of type $t \in \mathcal{S}$ towards vertices of type $s \in \mathcal{S}$. We have for all $r > 0$ and $\nu < 2\tau - 1$ that
    \[
    \prob\left( \bigcup_{t = 1}^{u_n^\uparrow} \bigcup_{s = 1}^{u_n^\uparrow} \{M_{ts} > n^{1 - \nu}\} \right) \leq 1/n^r.
    \]
\end{lemma}

With the quantification of self-loops and multi-arcs we can start the computation of $\bar{A}_n$. This will happen in two steps. First, we will count the number of arcs between two fixed vertex-types where we do not distinguish between ``normal'' arcs and ``bad'' multi-arc/self-loops. We will do this by first estimating the probability that one arc gets assigned to two given vertex types (for super stable and other vertex-types separately). Then, we will use these probabilities to estimate the total number of arcs placed between two given vertex-types. Hereto, we will start by defining random variables that count the number of arcs that are placed between two vertex-types. Compare this with Definition~\ref{def:vertex type arc count}, and note below that we do not remove self-loops and multi-edges.
\begin{definition}[Non-unique vertex-type arc count function]
    In $\CCI_{n, \mu}(T, C, I, J) = ([n], E)$, before removing self-loops and multi-arcs, we define the function $\bbar{A}_n : \mathcal{S} \times \mathcal{S} \to \mathbb{N}$ for two fixed vertex-types $t, s \in \mathcal{S}$ as
    \[
    \bbar{A}_n(t, s) := \left| (v, w) \in E : T_v = t \text{ and } T_w = s   \right|.
    \]
\end{definition}

\begin{lemma}[Arc to vertex-type probability]\label{lem:arc vertex prob stable}
    Suppose $t$ and $s$ are two stable vertex-types at tolerance $\tau \in (0, 1)$, and assume that Assumption~\ref{ass:CCI} is satisfied. Define $\mathcal{A}_{ts}^{(a)}$ as the event that arc $a \in [\mu n]$ is placed from a vertex with type $t$ to a vertex with type $s$. Then, we have that there exists a constant $\widehat{C} > 0$ such that
    \[
   \left| \prob(\mathcal{A}^{(a)}_{ts}) - \frac{q_t q_s \kappa(t, s)}{\mu}  \right| \leq \widehat{C}  \log(n) n^{- \frac{\tau}2} .
    \]
\end{lemma}
\begin{lemma}[Arc to vertex-type probability -- no stability]\label{lem:arc vertex prob no stable}
    Suppose $t,s \in \mathcal{S}$ are two vertex-types of which at least one is not stable at tolerance $\tau \in (0, 1)$. Assume that Assumption~\ref{ass:CCI} is satisfied. Define $\mathcal{A}_{ts}^{(a)}$ as the event that arc $a \in [\mu n]$ is placed from a vertex with type $t$ to a vertex with type $s$. Then, we have that there exists a constant $\widehat{C} > 0$ such that for any fixed $r > 0$ we have that
    \[
    \prob(\mathcal{A}_{ts}^{(a)}) \leq \widehat{C} \sqrt{q_t q_s} \left(\sqrt{q_t q_s}+ \frac{\log(n)}{\sqrt{n}}  + \frac1{n^r} \right) .
    \]
\end{lemma}

Lemma~\ref{lem:arc vertex prob stable} shows that the probability of arc placement between two vertex-types $t, s \in \mathcal{S}$ is roughly $q_t q_s \kappa(t, s) / \mu$. Since we seek to place $\mu n$ arcs, we can derive from this that there will be roughly $q_t q_s \kappa(t, s) n$ arcs from vertices of type $t$ to vertices of type $s$. This is the expected number we require in \eqref{eq:A bar expec}. Furthermore, Lemma~\ref{lem:arc vertex prob no stable} shows that the probability of arcs being placed in between some unstable vertices is small. Together with Lemma~\ref{lem:kappa bounded} this will show that these types of arcs do not contribute too much to the overall count. Together, we can use these considerations to show that $\bbar{A}_n$ concentrates for all stable vertex-types.

\begin{lemma}[Non-unique arc to vertex-type count]\label{lem:deviation A double hat}
    Suppose that Assumption~\ref{ass:CCI} is satisfied. Then, for all $\alpha > 3/8$ and $C > 0$ we have that
    \[
     \prob\left(\bigcup_{t = 1}^{u^\uparrow_n(\tau)} \bigcup_{s = 1}^{u^\uparrow_n(\tau)} \left\{ \left|\bbar{A}_n(t, s) - \lfloor \kappa(t, s) q_t q_s n \rfloor \right| > C n^{1/2 + \alpha} \sqrt{q_t q_s} \right\} \right) = o(1),
    \]
    for all $\tau > 1/2 - \alpha$.
\end{lemma}

Lemma~\ref{lem:deviation A double hat} provides the final piece of the puzzle. Not only does the parameter $\alpha$ that appears in the lemma's statement coincide with the parameter given in Theorem~\ref{thm:CCI to IRD}, but with the lemma we are now also in a position to prove Theorem~\ref{thm:CCI to IRD}. First, with Lemma~\ref{lem:self-loops} and \ref{lem:multi-arcs}, we will show that $\bbar{A}_n$ is close to $\bar{A}_n$ (cf. Definition~\ref{def:vertex type arc count}) for stable vertex types. Thereafter, we apply this ``closeness'' together with Assumption~\ref{ass:CCI}, Assumption~\ref{ass:CCI unstable vertices}, and Theorem~\ref{thm:IRD to ARD} to prove Theorem~\ref{thm:CCI to IRD}. We will use the same steps as outlined in Section~\ref{sec:heuristics II}.

\begin{proof}[\textbf{Proof of Theorem~\ref{thm:CCI to IRD}}]
    Fix a constant $\alpha \in (0, 1/2)$ such that $3/8 < \alpha < 1/2$. Given the value of $\alpha$, choose a super-stable tolerance $\tau^+ > 1/2$ (cf. Definition~\ref{def:super stable vertices}) and a value $\varepsilon$ close enough to zero such that 
    \begin{equation}\label{eq:condition tau}
        \frac34 < 1 - \frac{2 + \varepsilon}{8 + 6 \varepsilon} \leq \tau^+ < 2 \alpha.
    \end{equation}
    This is possible because of two reasons. Firstly, the fraction in \eqref{eq:condition tau} converges to $3/4 = 2 \cdot 3/8 < 2 \alpha$ for $\varepsilon \to 0$. Secondly, we can take $\varepsilon \to 0$, since the requirement $\mathbb{E}[T^{1 + \varepsilon}] < \infty$ for some $\varepsilon > 0$ in Assumption~\ref{ass:CCI} implies that $\mathbb{E}[T^{1 + \varepsilon'}] < \infty$ for all $\varepsilon' < \varepsilon$ as well.
    
    Additionally, choose a stable tolerance $\tau > 1 - 2\alpha$. We will move through the following steps to prove the result:
    \begin{enumerate}[label = \Roman*.]
        \item We will use Lemma~\ref{lem:self-loops}, \ref{lem:multi-arcs} and \ref{lem:deviation A double hat} to show that \begin{equation}\prob\left(\bigcup_{t = 1}^{u^\uparrow_n(\tau^+)} \bigcup_{s = 1}^{u^\uparrow_n(\tau^+)} \left\{ \left|\bar{A}_n(t, s) - \lfloor \kappa(t, s) q_t q_s n \rfloor \right| > n^{1/2 + \alpha} \sqrt{q_t q_s} \right\} \right) = o(1). \label{eq:goal step I proof II}\end{equation}
        \item We will use Lemma~\ref{lem:kappa bounded} and \ref{lem:deviation A double hat} to upgrade the result of Step I to \begin{equation}\label{eq:upgrade step II proof II}\prob\left(\bigcup_{t = 1}^{u^\uparrow_n(\tau)} \bigcup_{s = 1}^{u^\uparrow_n(\tau)} \left\{ \left|\bar{A}_n(t, s) - \lfloor \kappa(t, s) q_t q_s n \rfloor \right| > n^{1/2 + \alpha} \sqrt{q_t q_s} \right\} \right) = o(1).\end{equation}
        \item We will use the result of Step II together with Theorem~\ref{thm:IRD to ARD} to show the desired result.
    \end{enumerate}

    \paragraph{Step I.} Fix two super-stable vertex types $t, s \in \mathcal{S}$. Denote by $S_t$ the number of self-loops between vertices of type $t$, and by $M_{ts}$ the number of multi-arcs from a vertex of type $t$ to a vertex of type $s$. First of all, we note that \[\left(\bbar{A}_n(t, s) - S_t \1_{\{t = s\}} - M_{ts} \right)\vee 0 \leq \bar{A}_n(t, s) \leq \bbar{A}_n(t,s).\] Using this, we bound \eqref{eq:goal step I proof II} from above as
    \begin{subequations}
    \begin{align}
         &\prob\left(\bigcup_{t = 1}^{u^\uparrow_n} \bigcup_{s = 1}^{u^\uparrow_n} \left\{ \left|\bbar{A}_n(t, s) - \lfloor \kappa(t, s) q_t q_s n \rfloor \right| > n^{1/2 + \alpha} \sqrt{q_t q_s} \right\} \right) \\
         &+ \prob\left(\bigcup_{t = 1}^{u^\uparrow_n} \bigcup_{s = 1}^{u^\uparrow_n} \left\{ \left|\bbar{A}_n(t, s) - S_t \1_{\{t = s\}} - M_{ts} - \lfloor \kappa(t, s) q_t q_s n \rfloor \right| > n^{1/2 + \alpha} \sqrt{q_t q_s} \right\} \right).\label{eq:step I proof II bound}
    \end{align}
    \end{subequations}
    Now, we seek to apply Lemma~\ref{lem:deviation A double hat} on both probabilities in the sum. For the first, we can directly apply it. For the second, we will have to deal with the inclusion of the $S_t$ and $M_{ts}$ random variables. To do this, define the events $\mathcal{S}_{ts} := \{S_t \leq n^{1 - \nu_1}\}\cap\{t = s\}$ for some $\nu_1 < \tau^+$ and $\mathcal{M}_{ts} := \{M_{ts} \leq n^{1 - \nu_2}\}$ for some $\nu_2 < 2 \tau^+ - 1$. We can now intersect the second probability with these events and apply the union bound to find an upper-bound for \eqref{eq:step I proof II bound}:
    \begin{subequations}
    \begin{align}
    &\prob\left(\bigcup_{t = 1}^{u^\uparrow_n} \bigcup_{s = 1}^{u^\uparrow_n} \left\{ \left|\bbar{A}_n(t, s) - S_t \1_{\{t = s\}} - M_{ts} - \lfloor \kappa(t, s) q_t q_s n \rfloor \right| > n^{1/2 + \alpha} \sqrt{q_t q_s} \right\} \cap \mathcal{S}_{ts} \cap \mathcal{M}_{ts} \right)\label{eq:intersect step I proof II bound} \\
    &+ \prob\left(\bigcup_{t = 1}^{u_n^\uparrow} \neg \mathcal{S}_{tt} \right) +  \prob\left( \bigcup_{t = 1}^{u^\uparrow_n} \bigcup_{s = 1}^{u^\uparrow_n} \neg \mathcal{M}_{ts} \right) \label{eq:event err step I proof II}
    \end{align}
    \end{subequations}
    We will now show that \eqref{eq:intersect step I proof II bound} converges to zero. To do this, first note that conditioned on the events $\mathcal{S}_{ts}$ and $\mathcal{M}_{ts}$ we may replace $S_t$ by $n^{1 - \nu_1}$ and $M_{ts}$ by $n^{1 - \nu_2}$ in order to create a larger probability. Doing this, applying the triangle inequality, and removing the events $\mathcal{S}_{ts}$ and $\mathcal{M}_{ts}$ yields the upper-bound
    \begin{equation}\label{eq:domination goal proof II}
    \prob\left(\bigcup_{t = 1}^{u^\uparrow_n} \bigcup_{s = 1}^{u^\uparrow_n} \left\{ \left|\bbar{A}_n(t, s) - \lfloor \kappa(t, s) q_t q_s n \rfloor \right| > n^{1/2 + \alpha} \sqrt{q_t q_s} - n^{1 - \nu_1} - n^{1 - \nu_2} \right\}  \right).
    \end{equation}
    Now, we seek to show that $n^{1/2 + \alpha} \sqrt{q_t q_s}$ dominates over $n^{1 - \nu_{i}}$ for $i \in \{1, 2\}$. Recall that $q_{t, s} \leq n^{-1 + \tau^+}$ (due to super-stability) and conclude that \[n^{1/2 + \alpha} \sqrt{q_t q_s} \geq n^{\alpha + \tau^+ - 1/2}.\]
    Thus, for domination we only need to show that $\alpha + \tau^+ - 1/2 > 1 - \nu_{i}$ for $i \in \{1, 2\}$. Moreover, since we can choose $\nu_1$ and $\nu_2$ such that $\nu_1 < \tau^+$ and $\nu_2 < 2 \tau^+ - 1$, we have domination when we can show both that $\alpha + \tau^+ - 1/2 > 1 - \tau^+$ and $\alpha + \tau^+ - 1/2 > 2 - 2\tau^+$. Both of these inequalities are true, since we picked $\alpha > 3/8$ and $\tau^+ > 3/4$ from \eqref{eq:condition tau}.

    Because $n^{1/2 + \alpha} \sqrt{q_t q_s}$ dominates $n^{1 - \nu_{1, 2}}$, we may further upper-bound \eqref{eq:domination goal proof II} as
    \[
    \prob\left(\bigcup_{t = 1}^{u^\uparrow_n} \bigcup_{s = 1}^{u^\uparrow_n} \left\{ \left|\bbar{A}_n(t, s) - \lfloor \kappa(t, s) q_t q_s n \rfloor \right| > \frac12 n^{1/2 + \alpha} \sqrt{q_t q_s} \right\}  \right).
    \]
    Note this upper bound equals $o(1)$ due to Lemma~\ref{lem:deviation A double hat}. Thus, indeed \eqref{eq:intersect step I proof II bound} converges to zero. We also immediately have that \eqref{eq:event err step I proof II} converges to zero from Lemmas~\ref{lem:self-loops} and \ref{lem:multi-arcs}. Hence, we can finally conclude that \eqref{eq:step I proof II bound} converges to zero, implying that \eqref{eq:goal step I proof II} is satisfied.

    \paragraph{Step II.} In this step we will bound the following probability:
    \begin{equation}\label{eq:goal step II proof II}
            \prob\left( \bigcup_{(t \vee s) > u_n^\uparrow(\tau^+)}^{u_n^\uparrow(\tau)}\left\{ \left|\bar{A}_n(t, s) - \lfloor \kappa(t, s) q_t q_s n \rfloor \right| > n^{1/2 + \alpha} \sqrt{q_t q_s} \right\} \right).
    \end{equation}
    Note that the union inside \eqref{eq:goal step II proof II} considers only the pairs vertex-types $t, s \in \mathcal{S}$ for which at least one of the two is not super-stable. This is the difference between the union in \eqref{eq:upgrade step II proof II} when compared to the one in \eqref{eq:goal step I proof II}. Thus, if \eqref{eq:goal step II proof II} converges to zero, then (due to Step I) we also have that \eqref{eq:upgrade step II proof II} converges to zero.

    To show that \eqref{eq:goal step II proof II} converges to zero, we will first show that $n^{1/2 + \alpha} \sqrt{q_t q_s}$ is greater than $\lfloor \kappa(t, s) q_t q_s n \rfloor $. To this end, it follows from Lemma~\ref{lem:kappa bounded} to show there exists a constant $\kappa^\uparrow \in \mathbb{R}$ such that
    \[
    n^{1/2 + \alpha} \sqrt{q_t q_s} - \lfloor \kappa(t, s) q_t q_s n \rfloor \geq n^{1/2 + \alpha} \sqrt{q_t q_s} \left( 1 - \kappa^\uparrow n^{1/2 - \alpha} \sqrt{q_t q_s}  \right).
    \]
    Because either vertex-type $t$ or $s$ is not super-stable we know that
    \[
    \kappa^\uparrow n^{1/2 - \alpha} \sqrt{q_t q_s} \leq \kappa^\uparrow n^{\tau^+ / 2 - \alpha}.
    \]
    From \eqref{eq:condition tau} we find that $\alpha > \tau^+ / 2$, implying for $n$ large that $\kappa^\uparrow n^{1/2 - \alpha} \sqrt{q_t q_s} < 1$. Hence, we may conclude indeed that
    \[
    n^{1/2 + \alpha} \sqrt{q_t q_s} \geq \lfloor \kappa(t, s) q_t q_s n \rfloor .
    \]
    The consequence of this, due to the fact that $\bar{A}_n \geq 0$, is that \eqref{eq:goal step II proof II} equals
    \[
    \prob\left(  \bigcup_{(t \vee s) > u_n^\uparrow(\tau^+)}^{u_n^\uparrow(\tau)}\left\{ \bar{A}_n(t, s) - \lfloor \kappa(t, s) q_t q_s n \rfloor  > n^{1/2 + \alpha} \sqrt{q_t q_s} \right\} \right).
    \]
    The idea is now to upper-bound this probability by recalling that $\bbar{A}_n \geq \bar{A}_n$ and using Lemma~\ref{lem:deviation A double hat} to show it converges to zero. We can apply Lemma~\ref{lem:deviation A double hat}, since  $\tau > 1 - 2\alpha > 1/2 - \alpha$. Putting this plan into action yields
    \begin{align*}
        &\prob\left( \bigcup_{(t \vee s) > u_n^\uparrow(\tau^+)}^{u_n^\uparrow(\tau)}\left\{ \bar{A}_n(t, s) - \lfloor \kappa(t, s) q_t q_s n \rfloor  > n^{1/2 + \alpha} \sqrt{q_t q_s} \right\} \right),\\
        &\leq \prob\left( \bigcup_{(t \vee s) > u_n^\uparrow(\tau^+)}^{u_n^\uparrow(\tau)}\left\{ \bbar{A}_n(t, s) - \lfloor \kappa(t, s) q_t q_s n \rfloor  > n^{1/2 + \alpha} \sqrt{q_t q_s} \right\} \right),\\
        &\leq \prob\left( \bigcup_{(t \vee s) > u_n^\uparrow(\tau^+)}^{u_n^\uparrow(\tau)}\left\{ \left| \bbar{A}_n(t, s) - \lfloor \kappa(t, s) q_t q_s n \rfloor \right| > n^{1/2 + \alpha} \sqrt{q_t q_s} \right\} \right) = o(1).
    \end{align*}

    \paragraph{Step III.} We will now use Assumption~\ref{ass:CCI unstable vertices} to remove the influence of unstable vertices. Then, we will condition on the realisation of $\bar{A}_n$ using the law of total probability to turn $\CCI$ into $\ARD$. This yields
    \begin{align*}
    \prob\left( \CCI_{n, \mu} \in \mathcal{Q}_n \right) &= \prob\left( \CCI_{n, \mu}^- \in \mathcal{Q}_n \right) + o(1),\\
    &= \sum_{\Lambda'_n}\prob\left( \CCI_{n, \mu}^- \in \mathcal{Q}_n \;|\; \bar{A}_n = \Lambda'_n \right) \prob\left( \bar{A}_n = \Lambda'_n \right) + o(1),\\
    &= \sum_{\Lambda'_n}\prob\left( \ARD_n(T, \Lambda_n') \in \mathcal{Q}_n\right) \prob\left( \bar{A}_n = \Lambda'_n \right) + o(1).
    \end{align*}
    Note in the expression above that $\prob\left( \bar{A}_n = \Lambda'_n \right) = 0$ if there exists an unstable vertex-type $t$ or $s$ at tolerance $\tau$ for which $\Lambda_n'(t, s) > 0$. Hence, we we can simply set $\bar{A}_n(t, s) = 0$ for these vertex-types. Now, define the following set of ``desirable'' $\bar{A}_n$ realisations:
    \[
    \mathcal{L}_n = \left\{ \Lambda_n' : \left|\Lambda_n'(t, s) - \lfloor \kappa(t, s) q_t q_s n \rfloor \right| \leq n^{1/2 + \alpha} \sqrt{q_t q_s} \text{ for all } t, s \leq u_n^\uparrow(\tau)  \right\}.
    \]
    Splitting up the above sum into the values of $\Lambda_n'$ that fall within $\mathcal{L}_n$ and the ones that do not yields
    \begin{align*}
    \prob\left( \CCI_{n, \mu} \in \mathcal{Q}_n \right) \leq &\sum_{\Lambda'_n \in \mathcal{L}_n}\prob\left( \ARD_n(T, \Lambda_n') \in \mathcal{Q}_n\right) \prob\left( \bar{A}_n = \Lambda'_n \right) \\
    &+ \prob\left(\bigcup_{t = 1}^{u_n^\uparrow(\tau)} \bigcup_{s = 1}^{u_n^\uparrow(\tau)} \left\{ \left|\bar{A}_n(t, s) - \lfloor \kappa(t, s) q_t q_s n \rfloor \right| > n^{1/2 + \alpha} \sqrt{q_t q_s} \right\} \right) + o(1).
    \end{align*}
    We can find a similar lower bound by disregarding all the value of $\Lambda_n'$ that do not fall within $\mathcal{L}_n$. Hence, by using the result of Step II, we now have upper and lower bounds of the desired probability in terms of $\ARD$ probabilities given by
    \[
    \sum_{\Lambda'_n \in \mathcal{L}_n}\prob\left( \ARD_n(T, \Lambda_n') \in \mathcal{Q}_n\right) \prob\left( \bar{A}_n = \Lambda'_n \right) \leq \prob\left( \CCI_{n, \mu} \in \mathcal{Q}_n \right) \leq \sum_{\Lambda'_n \in \mathcal{L}_n}\prob\left( \ARD_n(T, \Lambda_n') \in \mathcal{Q}_n\right) \prob\left( \bar{A}_n = \Lambda'_n \right) + o(1).
    \]
    We will end the proof by showing that
    \begin{equation}\label{eq:goal Step III proof II}
        \sum_{\Lambda'_n \in \mathcal{L}_n}\prob\left( \ARD_n(T, \Lambda_n') \in \mathcal{Q}_n\right) \prob\left( \bar{A}_n = \Lambda'_n \right) \to p.
    \end{equation}
    To achieve this, we first note for all $\Lambda_n' \in \mathcal{L}_n$ that $\Lambda_n'(t, s) = 0$ when either $t$ or $s$ is unstable, and that for all other $t, s \leq u_n^\uparrow(\tau)$ we have
    \[
    \Lambda^-_n(t, s) := \lfloor \kappa(t, s) q_t q_s n \rfloor - \lfloor n^{1/2 + \alpha} \sqrt{q_t q_s} \rfloor \leq \Lambda'_n(t, s) \leq \lfloor \kappa(t, s) q_t q_s n \rfloor + \lfloor n^{1/2 + \alpha} \sqrt{q_t q_s} \rfloor =: \Lambda^+_n(t, s).
    \]
    Here, the floor in the upper-bound is valid, since $\Lambda'_n(t, s)$ must be an integer. We will now exploit monotonicity of $\mathcal{Q}_n$ to take the $\ARD$ probability out of the sum in \eqref{eq:goal Step III proof II}. We will assume without loss of generality that $\mathcal{Q}_n$ is increasing. The decreasing case is analogous. From Lemma~\ref{lem:monotonicity ARD} we know that
    \[
    \prob\left( \ARD_n(T, \Lambda_n^-) \in \mathcal{Q}_n\right) \leq \prob\left( \ARD_n(T, \Lambda_n') \in \mathcal{Q}_n\right) \leq \prob\left( \ARD_n(T, \Lambda_n^+) \in \mathcal{Q}_n\right).
    \]
    Thus, when substituting these bounds into \eqref{eq:goal Step III proof II}, removing the $\ARD$ probabilities from the sum, and computing the remaining sum, we find
    \begin{subequations}
    \begin{align}
            \prob\left( \ARD_n(T, \Lambda_n^-) \in \mathcal{Q}_n\right)\prob\left( \bar{A}_n \in \mathcal{L}_n \right) &\leq \sum_{\Lambda'_n \in \mathcal{L}_n}\prob\left( \ARD_n(T, \Lambda_n') \in \mathcal{Q}_n\right) \prob\left( \bar{A}_n = \Lambda'_n \right)\label{eq:goal III lower proof II} \\ &\leq \prob\left( \ARD_n(T, \Lambda_n^+) \in \mathcal{Q}_n\right)\prob\left( \bar{A}_n \in \mathcal{L}_n \right).\label{eq:goal III upper proof II}
    \end{align}
    \end{subequations}
    From the result of Step II we know that $\prob\left( \bar{A}_n \in \mathcal{L}_n \right) \to 1$. Thus, we now seek to invoke Theorem~\ref{thm:IRD to ARD} to show that both $\ARD$ probabilities converge to $p$. For this, we first focus on the lower-bound in \eqref{eq:goal III lower proof II} and notice that
    \[
    \Lambda^-_n(t, s) = \lfloor \kappa(t, s) q_t q_s n - \lfloor n^{1/2 + \alpha} \sqrt{q_t q_s} \rfloor \rfloor \geq \lfloor \kappa(t, s) q_t q_s n -  n^{1/2 + \alpha} \sqrt{q_t q_s} \rfloor =   \left\lfloor \left(\kappa(t, s) - \frac{n^{-1/2 + \alpha}}{\sqrt{q_t q_s}} \right) q_t q_s n \right\rfloor.
    \]   
    Now, we consider the kernel $\kappa^-_n(t, s) = \kappa(t, s) - n^{-1/2 + \alpha}/\sqrt{q_t q_s}$. Note that our initial choice of $\alpha$ and $\tau$ implies that $\alpha > 1/2 - \tau / 2$. We also have that $\alpha > \tau / 2$, since $\tau < 1/2$ and $\alpha > 3/8$. All in all, with our choices of $\alpha$ and $\tau$ we have that Assumption~\ref{ass:kernel bound} is satisfied. This is because Lemma~\ref{lem:kappa bounded} tells us for $n$ large and one unstable vertex-type that there exists a constant $c > 0$ such that
    \[
    \kappa(t, s) \leq \frac{c \sqrt{q_t q_s}}{\sqrt{q_t q_s}} \leq  \frac{c n^{-1/2 + \tau / 2}}{\sqrt{q_t q_s}} \leq \frac{c n^{-1/2 + \alpha}}{\sqrt{q_t q_s}}.
    \]
    Moreover, if we fix a kernel $\kappa'_n$ such that $|\kappa_n'(t, s) - \kappa_n^-(t, s)| \leq n^{-1/2 + \alpha}/\sqrt{q_t q_s}$, then we particularly have that
    \begin{align*}
    \left|\kappa_n'(t, s) - \kappa(t, s) \right| &\leq  \left|\kappa_n'(t, s) - n^{-1/2 + \alpha}/\sqrt{q_t q_s} - \kappa(t, s) \right| + n^{-1/2 + \alpha}/\sqrt{q_t q_s},\\
    &\leq  \left|\kappa^-_n(t,s) - \kappa(t, s) \right| + n^{-1/2 + \alpha}/\sqrt{q_t q_s} \leq 2 n^{-1/2 + \alpha}/\sqrt{q_t q_s}.
    \end{align*}
    Through the assumption in \eqref{eq:kern sequence CCI} we may now conclude that $\prob(\texttt{IRD}_n(T, \kappa'_n) \in \mathcal{Q}_n) \to p$. Hence, when we set $\hat{\Lambda}_n^-(t, s) = \lfloor \kappa^-_n(t, s) q_t q_s n \rfloor$, then we indeed find for \eqref{eq:goal III lower proof II} through an application of Theorem~\ref{thm:IRD to ARD} (with $C = 1$) that
    \[
    p = \lim_{n \to \infty } \prob\left( \ARD_n(T, \hat{\Lambda}^-_n) \in \mathcal{Q}_n\right)\prob\left( \bar{A}_n \in \mathcal{L}_n \right) \leq \limsup_{n \to \infty} \prob\left( \ARD_n(T, \Lambda_n^-) \in \mathcal{Q}_n\right)\prob\left( \bar{A}_n \in \mathcal{L}_n \right).
    \]
    The approach to show that \eqref{eq:goal III upper proof II} converges to $p$ is the same. We first note that
    \[
    \Lambda_n^+(t, s) \leq \lfloor \kappa(t, s) q_t q_s n \rfloor + n^{1/2 + \alpha} \sqrt{q_t q_s} \leq \lfloor \kappa(t, s) q_t q_s n  + n^{1/2 + \alpha} \sqrt{q_t q_s} \rfloor =  \left\lfloor \left(\kappa(t, s) + \frac{n^{-1/2 + \alpha}}{\sqrt{q_t q_s}} \right) q_t q_s n \right\rfloor.
    \]
    Now, we consider the kernel $\kappa^+_n(t, s) = \kappa(t, s) + n^{1/2 + \alpha}/\sqrt{q_t q_s}$ and note that Assumption~\ref{ass:kernel bound} is satisfied for this kernel due to Lemma~\ref{lem:kappa bounded} if we pick e.g. $C = 2$. Now, we fix an arbitrary kernel $\kappa_n'$ such that $|\kappa_n'(t, s) - \kappa_n^-(t, s)| \leq 2 n^{-1/2 + \alpha}/\sqrt{q_t q_s}$, and we note particularly that
    \begin{align*}
    \left|\kappa_n'(t, s) - \kappa(t, s) \right| &\leq  \left|\kappa_n'(t, s) + n^{-1/2 + \alpha}/\sqrt{q_t q_s} - \kappa(t, s) \right| + n^{-1/2 + \alpha}/\sqrt{q_t q_s},\\
    &\leq  \left|\kappa^+_n(t,s) - \kappa(t, s) \right| + n^{-1/2 + \alpha}/\sqrt{q_t q_s} \leq 3 n^{-1/2 + \alpha}/\sqrt{q_t q_s}.
    \end{align*}
    Through the assumption in \eqref{eq:kern sequence CCI} we may now conclude that $\prob(\texttt{IRD}_n(T, \kappa'_n) \in \mathcal{Q}_n) \to p$. Hence, when we set $\hat{\Lambda}_n^+(t, s) = \lfloor \kappa^+_n(t, s) q_t q_s n \rfloor$, then we indeed find for \eqref{eq:goal III lower proof II} through an application of Theorem~\ref{thm:IRD to ARD} (with $C = 2$) that
    \[
        \liminf_{n \to \infty} \prob\left( \ARD_n(T, \Lambda_n^+) \in \mathcal{Q}_n\right)\prob\left( \bar{A}_n \in \mathcal{L}_n \right) \leq \lim_{n \to \infty } \prob\left( \ARD_n(T, \hat{\Lambda}^+_n) \in \mathcal{Q}_n\right)\prob\left( \bar{A}_n \in \mathcal{L}_n \right) =p.
    \]
    We have now shown that the upper- and lower-bound in \eqref{eq:goal III lower proof II} and \eqref{eq:goal III upper proof II} converge to $p$. Thus, we have now shown through \eqref{eq:goal Step III proof II} that indeed $\prob(\CCI_{n, \mu}(T, C, I, J) \in \mathcal{Q}_n) \to p$.
\end{proof}

\begin{remark}
    Lemmas~\ref{prob:unstable}, \ref{lem:arc vertex prob stable} and \ref{lem:arc vertex prob no stable} were not directly used in the proof of Theorem~\ref{thm:CCI to IRD}. These are all needed to prove the ``main technical lemma'' of this section: Lemma~\ref{lem:deviation A double hat}.
\end{remark}

\section{Proofs of propositions and lemmas}\label{sec:tool proof}

We end this paper by giving the proofs of all lemmas and propositions that were stated in the main text. We will start with the proofs of all proposition in Section~\ref{sec: applications}, then give the proofs of all lemmas in Section~\ref{sec:proof thm I}, and finally give the proofs of all lemmas in Section~\ref{sec:proof II}.

\subsection{Proofs of propositions}\label{app:extra proofs}
\begin{proof}[Proof of Proposition~\ref{prop:ass CL}]
When $\expec[W^{1 + \varepsilon}] < \infty$, we know that $\sum_{t = 1}^\infty t^{1 + \varepsilon} \cdot q_t < \infty$. This has the following two consequences:
\begin{enumerate}\itemsep0em
    \item There exists $t_1$ for which $(q_t)_{t \geq t_1}$ is decreasing sequence.
    \item There exists a $t_2$ for which $q_t < t^{-2 - \varepsilon}$ for all $t \geq t_2$.
\end{enumerate}

Define $t^\downarrow := \max\{t_1, t_2\}$, fix some $n \in \mathbb{N}$ large and two vertex-types $t, s \in \mathbb{N}$. Suppose that $t > t^\downarrow$, then from consequence 2. we have that $q_t^{-1/(2 + \varepsilon)} > t$. This yields the following lower-bound:
\begin{equation}\label{eq:unstable lower CL}
    q_t^{-1/2} = q_t^{- \frac{\varepsilon}{2(2 + \varepsilon)}} \cdot q_t^{- \frac1{2 + \varepsilon}} \geq q_t^{- \frac{\varepsilon}{2(2 + \varepsilon)}} \cdot t.
\end{equation}
Now, assume without loss of generality that vertex-type $t$ is unstable at some unspecified tolerance $\tau$. According to Definition~\ref{def:stable vertices}, for this type $t$ we must have $q_t < n^{-1 + \tau}$. Since this upper-bound converges to zero, we must also have that $u_n^\uparrow(\tau) \to \infty$ as $n \to \infty$. Hence, there exists a $n$ large enough for which $u_n^\uparrow > t^\downarrow$, implying \eqref{eq:unstable lower CL} is satisfied for every unstable vertex-type.

Now, focus on the possibly stable vertex type $s$. If it happens that $s > t^\downarrow$, then \eqref{eq:unstable lower CL} is valid for type $s$ too. It is only possible that \eqref{eq:unstable lower CL} is not valid for type $s$ when $s \leq t^\downarrow$. However, since $t^\downarrow$ is independent from $n$, we can choose a $C > 0$ large enough such that
\begin{equation}\label{eq:stable lower CL}
    C q_s^{-1/2} \geq s / \expec[W] \qquad \text{for all } s \leq t^\downarrow.
\end{equation}
By noting in \eqref{eq:unstable lower CL} that $ q_t^{- \frac{\varepsilon}{2(2 + \varepsilon)}} \cdot t \geq t$, we may conclude that \eqref{eq:stable lower CL} holds for all $s \in \mathbb{N}$. Now, it is time to lower-bound \eqref{eq:kernel bound} in Assumption~\ref{ass:kernel bound} with the aforementioned value of $C$. We find
\begin{align*}
    \frac{C n^{\alpha - 1/2}}{\sqrt{q_t q_s}} = \frac{C}{\sqrt{q_s}} \cdot \frac{n^{\alpha - 1/2}}{\sqrt{q_t}} \geq \frac{ts}{\expec[W]} \cdot n^{\alpha - 1/2} \cdot q_t^{- \frac{\varepsilon}{2(2 + \varepsilon)}} = \kappa(t, s) \cdot n^{\alpha - 1/2} \cdot q_t^{- \frac{\varepsilon}{2(2 + \varepsilon)}}.
\end{align*}
Here, we used \eqref{eq:unstable lower CL} and \eqref{eq:stable lower CL} at the inequality. We will end the proof by showing that $n^{\alpha - 1/2} \cdot q_t^{- \frac{\varepsilon}{2(2 + \varepsilon)}} \geq 1$ for some choice of $\tau, \alpha > 0$, after which Assumption~\ref{ass:kernel bound} will be satisfied. First, we will use instability of $t$ to bound
\[
n^{\alpha - 1/2} \cdot q_t^{- \frac{\varepsilon}{2(2 + \varepsilon)}} \geq n^{\alpha - 1/2} \cdot \left( n^{1 - \tau} \right)^{\frac{\varepsilon}{2(2 + \varepsilon)}} = n^{\frac{\varepsilon}{2(2 + \varepsilon)} + \alpha - 1/2 - \frac{\varepsilon \tau}{2(2 + \varepsilon)}}.
\]
If an admissable pair $\alpha, \tau > 0$ exists, then from Assumption~\ref{ass:kernel bound} it follows that $\alpha - 1/2 > \tau/2$. Now take any $0 < \tau < \varepsilon/(2+2\varepsilon) < 1/2$. Then we have that
\[
	\frac{\varepsilon}{2(2 + \varepsilon)} + \alpha - 1/2 - \frac{\varepsilon \tau}{2(2 + \varepsilon)} 
	\ge \frac{\varepsilon - \tau(2 + 2\varepsilon)}{2(2+\varepsilon)} >0,
\]
which implies that $n^{\alpha - 1/2} \cdot q_t^{- \frac{\varepsilon}{2(2 + \varepsilon)}} \geq 1$.
\end{proof}

\begin{proof}[Proof of Proposition~\ref{prop:IRD regularity inf}]
    As the probability measure $\nu$ we can simply take the Borel measure that assigns the probabilities $\nu(\{t\}) = q_t$ for all $t \in \mathbb{N}$. Then, by the weak law of large numbers condition (a) is satisfied. Moreover, since $\mathbb{N}$ is a discrete space, we also have that the continuity conditions in (b) and (c) are satisfied. Hence we are left to show that
    \begin{enumerate}[label = \Roman*.]\itemsep0em
    \item $\varphi_n(t, s) := (\kappa'_n(t, s) - \kappa(t, s))/\kappa(t, s) \to 0$ as $n \to \infty$, and
    \item $\displaystyle \lim_{n \to \infty } \frac1{n^2} \expec\left[ \sum_{v = 1}^n \sum_{w = 1}^n \kappa(T_v, T_w) \right] = \lim_{n \to \infty} \frac1{n^2} \expec\left[ \sum_{v = 1}^n \sum_{w \neq v} \kappa'_n(T_v, T_w) \right] = \sum_{t = 1}^\infty \sum_{s = 1}^\infty \kappa(t, s) q_t q_s < \infty.$
    \end{enumerate}

    \paragraph{Part I.} First note from the definition of $\kappa'_n$ that for fixed $t, s \in \mathbb{N}$ we have
    \begin{equation}\label{eq:varphi bound}
    - \frac{C n^\alpha}{\kappa(t, s) \sqrt{q_t q_s n}} \leq \varphi_n(t, s) \leq \frac{C n^\alpha}{\kappa(t, s) \sqrt{q_t q_s n}}.
    \end{equation}
    Note that for fixed $t, s \in \mathbb{N}$ the values of $q_t, q_s$ and $\kappa(t, s)$ are deterministic, so the lower- and upper-bounds are deterministic as well. We only have a problem whenever $\kappa(t, s)$, $q_t$ or $q_s$ are zero. But, if $q_t$ or $q_s$ is zero, then the vertex-type cannot exist, meaning we can remove it from the model. Alternatively, when $\kappa(t,s) = 0$, we can simply take $\varphi_n(t, s) = 0$, since $\varphi_n$ is part of a multiplicative factor in $\texttt{IRD}_n$ that is multiplied by $\kappa$. In other words, if $\kappa(t, s) = 0$, the value of $\varphi_n(t,s)$ does not matter. Finally, since $\alpha < 1/2$ we have that these bounds both converge to zero, letting us conclude that $\varphi_n(t, s) \to 0$ almost surely (in $\mathbb{P}_n$).

    \paragraph{Part II -- Rightmost sum.} 
    We first show that the last sum is finite. For this we will substitute the definition of $\kappa$ and compute the result to show convergence.
    \begin{align*}
        \sum_{t = 1}^\infty \sum_{s = 1}^\infty \kappa(t, s) q_t q_s &= \sum_{t = 1}^\infty \sum_{s = 1}^\infty \sum_{i = 1}^\infty \sum_{j = 1}^\infty \frac{\mu p_{ij}  I(t, i) J(s, j)  q_t q_s}{\lambda_i \varrho_j},\\
        &= \sum_{i = 1}^\infty \sum_{j = 1}^\infty \sum_{t = 1}^\infty \sum_{s = 1}^\infty \frac{\mu p_{ij} I(t, i) J(s, j) q_t q_s}{\lambda_i \varrho_j},\\
        &= \mu \sum_{i = 1}^\infty \sum_{j = 1}^\infty \lambda_i^{-1} \varrho_j^{-1} p_{ij} \sum_{t = 1}^\infty I(t, i) q_t \sum_{s = 1}^\infty J(s, j) q_s,\\
        &= \mu \sum_{i = 1}^\infty \sum_{j = 1}^\infty \lambda_i^{-1} \varrho_j^{-1} p_{ij} \lambda_i \varrho_j,\\
        &= \mu \sum_{i = 1}^\infty \sum_{j = 1}^\infty p_{ij} = \mu < \infty.
    \end{align*}
    Note in the second line we swapped the order of summation. This is possible due to the fact that all terms in the sum are non-negative, and Tonelli's theorem. 

    \paragraph{Part II -- Leftmost limit.} We will show the leftmost limit in II equals $\mu$ too. Using the fact that $(T_v)_{v \geq 1}$ is an i.i.d. distributed sequence, and the law of total expectation yields
    \begin{align*}
    \frac1{n^2} \expec\left[ \sum_{v = 1}^n \sum_{w = 1}^n \kappa(T_v, T_w) \right] &= \frac{n(n - 1)}{n^2} \cdot \expec[\kappa(T_1, T_2)] + \frac{1}{n} \cdot \expec[\kappa(T_1, T_1)],\\
    &= \frac{n(n - 1)}{n^2} \cdot \sum_{t = 1}^\infty \sum_{s = 1}^\infty \kappa(t, s) q_t q_s + \frac{1}{n} \cdot \sum_{t = 1}^\infty \kappa(t, t) q_t ,\\
    &= \mu \cdot \frac{n(n - 1)}{n^2} + \frac{1}{n} \cdot \sum_{t = 1}^\infty \kappa(t, t) q_t .
    \end{align*}
    Now, Lemma~\ref{lem:kappa bounded} shows that $\kappa$ is bounded. Hence, for some $\widehat{C} > 0$ we have
    \[
    0 \leq \frac{1}{n} \cdot \sum_{t = 1}^\infty \kappa(t, t) q_t \leq \frac{\widehat{C}}{n} \sum_{t = 1}^\infty q_t = \frac{\widehat{C}}{n} \to 0. 
    \]
    Thus, we indeed have that
    \[
    \lim_{n \to \infty} \frac1{n^2} \expec\left[ \sum_{v = 1}^n \sum_{w = 1}^n \kappa(T_v, T_w) \right] = \lim_{n \to \infty} \left[ \mu \cdot \frac{n(n - 1)}{n^2} + \frac{1}{n} \cdot \sum_{t = 1}^\infty \kappa(t, t) q_t  \right] = \mu + 0 = \mu.
    \]
    \paragraph{Part II -- Middle limit.} We show that the middle limit equals $\mu$ too. We split up the expectation into a sum over the kernel. Using the same considerations as above, we find
    \[
    \frac1{n^2} \expec\left[ \sum_{v = 1}^n \sum_{w \neq v} \kappa'_n(T_v, T_w) \right] = \mu \cdot \frac{n^2 - n}{n^2} + \expec[\varphi_n(T_1, T_2)] \cdot \frac{n^2 - n}{n^2}.
    \]
    We will now bound the expected value of $\varphi_n$ by using the law of total expectation and \eqref{eq:varphi bound}. 
    \begin{equation}\label{eq:exp phi bound}
    -C n^{\alpha - 1/2} \sum_{t = 1}^\infty \sqrt{q_t} \sum_{s = 1}^\infty \sqrt{q_s} \leq \expec[\varphi_n(T_1, T_2)] \leq \sum_{t = 1}^\infty \sum_{s = 1}^\infty \frac{C n^\alpha q_t q_s}{\sqrt{q_t q_s n}} = C n^{\alpha - 1/2} \sum_{t = 1}^\infty \sqrt{q_t} \sum_{s = 1}^\infty \sqrt{q_s}. 
    \end{equation}
    Now, because we have assumed that $\expec[T^{1 + \varepsilon}] < \infty$, we have for $t$ large that $q_t < t^{-2 - \varepsilon}$. Hence, $\sqrt{q_t} < t^{-1 - \varepsilon/2}$, meaning the sums in \eqref{eq:exp phi bound} are finite. Thus, for some $\widetilde{C} > 0$ we have that
    \[
    -\widetilde{C} n^{\alpha - 1/2} \leq \expec[\varphi_n (T_1, T_2)] \leq \widetilde{C} n^{\alpha - 1/2}.
    \]
    Since $\alpha < 1/2$, both these terms converge to zero. All together, this shows that
    \[
    \lim_{n \to \infty} \frac1{n^2} \expec\left[ \sum_{v = 1}^n \sum_{w \neq v} \kappa'_n(T_v, T_w) \right] = \lim_{n \to \infty} \left[ \mu \cdot \frac{n^2 - n}{n^2} + \expec[\varphi_n(T_v, T_w)] \cdot \frac{n^2 - n}{n^2} \right] = \mu \cdot 1 + 0 \cdot 1 = \mu.
    \]
\end{proof}

\begin{proof}[Proof of Proposition~\ref{prop:GSCC}]
    Denote by $\CCI_{n, \mu}^-$ the version of $\CCI_{n, \mu}$ after removing all arcs from or to an unstable vertex. In this model, let $\mathcal{C}_{(i)}^-$ denote the $i$-th largest strongly connected component. Finally, let $A_n^\uparrow$ denote the number of unstable arcs. We split up the proof in the following steps.
    \begin{enumerate}[label = \Roman*.]\itemsep0em
        \item We show that with high probability for some constant $p \in [0, 1)$ we have $A_n^\uparrow \leq n^{p}$.
        \item We show that \[\frac1n \left| \bigcup_{i = 1}^{n^{p}} \mathcal{C}_{(i)}^- \right| \to \alpha,\] in probability.
        \item We use the above two points to show that $|\mathcal{C}_{\max}|/n \to \alpha$ in probability as well.
    \end{enumerate}

    \paragraph{Step I.} Recall that the total number of arcs in $\CCI_{n, \mu}$ is $\lfloor \mu n \rfloor$. Thus, we can write (cf. Definition~\ref{def:vertex type arc count}) \[A_n^\uparrow = \lfloor \mu n \rfloor - \sum_{t = 1}^{u_n^\uparrow} \sum_{s = 1}^{u_n^\uparrow} \bar{A}_n(t, s). \] Then, using \eqref{eq:upgrade step II proof II} we can bound with high probability \[
    A^\uparrow_n \leq \mu n - \sum_{t = 1}^{u_n^\uparrow} \sum_{s = 1}^{u_n^\uparrow} \left[  \lfloor \kappa(t, s) q_t q_s n \rfloor - n^{1/2 + \alpha} \sqrt{q_t q_s} \right] \leq  \mu n - \sum_{t = 1}^{u_n^\uparrow} \sum_{s = 1}^{u_n^\uparrow} \left[   \kappa(t, s) q_t q_s n - 1 - n^{1/2 + \alpha} \sqrt{q_t q_s} \right].
    \]
    Since $\mathbb{E}[T^{1 + \varepsilon}] < \infty$, we know that the sum over $\sqrt{q_t}$-terms converges. Hence, when we compute all the negative sums, we find that there exists a constant $C > 0$ such that
    \[
    A^\uparrow_n \leq \mu n - \sum_{t = 1}^{u_n^\uparrow} \sum_{s = 1}^{u_n^\uparrow}   \kappa(t, s) q_t q_s n  + C n^{1/2 + \alpha} + (u_n^\uparrow)^2  .
    \]
    Using Lemma~\ref{lem:no stable} can now conclude that
    \[
    A^\uparrow_n \leq \mu n - \sum_{t = 1}^{u_n^\uparrow} \sum_{s = 1}^{u_n^\uparrow}  \kappa(t, s) q_t q_s n  + C n^{1/2 + \alpha} + n^{1 - \tau}  .
    \]
    We note that the remaining double sum (when summing over \emph{all} vertex-types) adds up to $\mu n$. Thus, we can bound the first two terms in the current upper-bound on $A_n^\uparrow$ to find
    \[
    A^\uparrow_n \leq \sum_{t = 1}^\infty \sum_{s = u_n^\uparrow + 1}^\infty \kappa(t, s) q_t q_s n + \sum_{t = u_n^\uparrow + 1}^\infty \sum_{s = 1}^\infty \kappa(t, s) q_t q_s n +  C n^{1/2 + \alpha} + n^{1 - \tau}.
    \]
    By applying Lemma~\ref{lem:kappa bounded} and the fact that the $q_t$-terms are probabilities, there exists a constant $\kappa^+ > 0$ such that
    \[
    A^\uparrow_n \leq 2 \kappa^+ n \prob(T > u_n^\uparrow) +  C n^{1/2 + \alpha} + n^{1 - \tau}.
    \]
    When we finally apply Lemma~\ref{prob:unstable} we find that there exists an overarching constant $\widehat{C} > 0$ such that
    \[
    A_n^\uparrow \leq \widehat{C} \left(n^{1+\frac{(\tau - 1)\varepsilon}{(2 + \varepsilon)}} +  n^{1/2 + \alpha} + n^{1 - \tau}  \right).
    \]
    Because $\tau \in (0, 1)$ and $\alpha < 1/2$, the result follows.
    
    \paragraph{Step II.} Let $\delta > 0$ be an arbitrary constant and set \[S_n := \left| \bigcup_{i = 1}^{n^{p}} \mathcal{C}_{(i)}^- \right|. \] Consider the event $\mathcal{Q}_n^-(\delta) := \{S_n > n \delta\}$. Note, if we were to add an extra arc to a graph $G$, then it will either not change the sizes of its strongly connected components, or merge two strongly connected components into one. In both cases, the ordered list of strong connected component sizes will change such that the size of the $i$-th largest strongly connected component before the added edge is smaller than or equal to the size of the $i$-largest connected component after adding the edge. Thus, we may conclude that $\mathcal{Q}_n^-(\delta)$ is increasing.

    Secondly, if we look at pairs of stable vertex-types in $\CCI_{n, \mu}^-$, then we note that all concentration lemmas (like e.g. Lemma~\ref{lem:deviation A double hat}) are still true, since probabilities of arcs being assigned to these vertex-types do not change. Arcs are just thrown away if they happen to be assigned to unstable vertex-types. Thus, we can use the result of Theorem~\ref{thm:CCI to IRD} to this slightly adapted model as well. Of course, Assumption~\ref{ass:CCI unstable vertices} is trivially satisfied for this model.

    Thirdly, suppose $\kappa_n'$ is a function that adheres to \eqref{eq:kern sequence CCI}. Denote by $|\mathcal{C}_{(i)}^{\IRD}|$ the $i$-th largest strongly connected component in $\IRD_n(T, \kappa_n')$, and by $S_n^{\IRD}$ its corresponding version of $S_n$. Then, due to Proposition~\ref{prop:IRD regularity inf} we can apply Theorem 3.9 in \cite{Cao2019OnDigraphs} to conclude that
    \begin{equation}\label{eq:subtr convergence GSCC}
    |\mathcal{C}_{(1)}^{\IRD}| / n \to \sum_{x = 1}^\infty \pi^+_x \pi^-_x q_x =: \alpha,
    \end{equation}
    in probability, where $q_x = \prob(T = x)$ and $\pi^\pm_x$ are defined through \eqref{eq:pi+} and \eqref{eq:pi-}. Moreover, by applying Theorem 3.11 in \cite{Cao2019OnDigraphs} we may conclude that $|\mathcal{C}_{(2)}^{\IRD}| \leq \log(n)^2$ with high probability, because otherwise it would be part of the giant. Thus, we may conclude that \[\left| \bigcup_{i = 2}^{n^{p}} \mathcal{C}_{(i)}^{\IRD} \right| \leq n^{p} \log(n)^2, \] with high probability. This means that indeed $S_n^{\IRD}/n \to \alpha$ in probability too.

    Combining these three points, using Theorem~\ref{thm:CCI to IRD}, we can conclude that $\prob(S_n \leq \delta n) \to 0$ when $\delta < \alpha$ and $\prob(S_n \leq \delta n ) \to 1$ when $\delta \geq \alpha$. Together, this means that $S_n/n \to \alpha$ in distribution, allowing us to conclude that $S_n/n \to \alpha$ in probability. Moreover, with the same argument we can also conclude from \eqref{eq:subtr convergence GSCC} that $|\mathcal{C}_{(1)}^-|/n \to \alpha$ in probability. 

    \paragraph{Step III.} We note from Step I that at most $n^{p}$ extra edges get added connecting to at least one unstable vertex in $\CCI_{n, \mu}$ with high probability. Each of these arcs can do one of two things:
    \begin{enumerate}
        \item Add all unstable vertices to the largest strongly connected component in $\CCI_{n, \mu}$.
        \item (Indirectly) connect $\mathcal{C}_{(i)}$ for some $i > 1$ to the largest strongly connected component in $\CCI_{n, \mu}$.
    \end{enumerate}
    If we denote by $N_n^\uparrow$ the number of unstable vertices, then the above two points show that with high probability we have that $|\mathcal{C}_{(1)}^-| \leq |\mathcal{C}_{\max}| \leq S_n + N_n^\uparrow$. Now, we will show that $N_n^\uparrow$ is sub-linear. For this, we recall from Lemma~\ref{prob:unstable} and assumption~\ref{ass:CCI} that
    \[
    \prob(T > u_n^\uparrow(\tau)) \leq n^{\frac{(\tau - 1) \varepsilon}{2 + \varepsilon}}.
    \]
    Thus, we have that \[N_n^\uparrow \preceq \texttt{Bin}\left(n, n^{\frac{(\tau - 1) \varepsilon}{2 + \varepsilon}}\right).\]
    By Chebyshev's inequality it holds that with high probability \[N_n^\uparrow \leq n^{1 + \frac{(\tau - 1) \varepsilon}{2 + \varepsilon}} + n^{3/4 + \frac{(\tau - 1) \varepsilon}{4 + 2\varepsilon}} = o(n).\]
    Thus, we can conclude that $N_n^\uparrow/n \to 0$ in probability. Together with the results of step II we now have that $|\mathcal{C}_{(1)}^-|/n \to \alpha$ and $(S_n + N_n^\uparrow)/n \to \alpha$ in probability, implying that also $|\mathcal{C}_{\max}|/n \to \alpha$ in probability.
\end{proof}

\subsection{Proofs of lemmas for Theorem~\ref{thm:IRD to ARD}}\label{sec:tool proof thm I}

\begin{proof}[Proof of Lemma~\ref{lem:no stable}]
    Since $\expec[T^{\delta}] < \infty$ we have that $\sum_{t = 1}^\infty t^{\delta} q_t < \infty$.
    In particular, this means (for $t$ large) that $q_t \leq t^{-1 - \delta}$. Moreover, since $q_t \to 0$ as $t \to \infty$ we must have that $u_n^\uparrow \to \infty$ as $n \to \infty$. Thus, we know (for $n$ large) that
    \[
    u_n^\uparrow(\tau) \leq \widehat{u}_n^\uparrow(\tau) := \inf\{ t : s^{-1 - \delta} < n^{-1 + \tau} \text{ for all } s \geq t\}.
    \]
    We can now calculate the value of $\widehat{u}_n^\uparrow(\tau)$ to find the desired result. We have for all $t$ that
    \[
    t^{-1 - \delta} < n^{-1 + \tau} \iff n^{(1 - \tau)/(1 + \alpha)} < t.
    \]
    Hence, we can conclude that
    \[
    u_n^\uparrow(\tau) \leq \widehat{u}_n^\uparrow(\tau)  = \left\lceil n^{(1 - \tau)/(1 + \delta)} \right\rceil.
    \]
\end{proof}

\begin{proof}[Proof of Lemma~\ref{lem:well-concentrated types}]
    Fix a pair $t, s \in \mathcal{S}$. First note we can rewrite the event we are interested in as follows:
    \[
    \neg \mathcal{V}_{ts} = \{|N_t - q_tn| > \log(n) \sqrt{q_t n}\} \cup \{|N_s - q_sn| > \log(n) \sqrt{q_s n}\}.
    \]
    By applying the union bound, we can then bound
    \[
    \prob(\neg \mathcal{V}_{ts} ) \leq \prob(|N_t - q_tn| > \log(n) \sqrt{q_t n}) + \prob(|N_s - q_sn| > \log(n) \sqrt{q_s n}).
    \]
    We will now show the result for the first probability, the argument for the second probability will be analogous. We write
    \begin{equation}\label{eq:symm split chernoff}
         \prob(|N_t - q_tn| > \log(n) \sqrt{q_t n}) =  \prob(N_t >  q_tn + \log(n) \sqrt{q_t n}) + \prob(N_t <  q_tn - \log(n) \sqrt{q_t n}). 
    \end{equation}
    By noting that $N_t \sim \texttt{Bin}(n, q_t)$ we can apply the Chernoff bound (see \cite{Hofstad2016RandomNetworks} Theorem 2.21) on both probabilities to find 
    \[
    \prob(|N_t - q_tn| > \log(n) \sqrt{q_t n}) \leq 2 \exp(-\log(n)^2 / 2).
    \]
\end{proof}

\begin{proof}[Proof of Lemma~\ref{lem:undershoot mixed binomials}]
     Fix two stable vertex types $t, s \in \mathcal{S}$ and a $\kappa'_n(t, s)$. We define \[\kappa^+_n(t, s) = \kappa(t, s) + C n^{-1/2 + \alpha}/\sqrt{q_t q_s}.\] Note that $\kappa^+_n(t, s) \leq \kappa'_n(t, s)$. Hence, by recalling Definition~\ref{def:arc to vertex type}, we have the following stochastic bound.
    \[
    A_n(t, s) \succeq A^+_n(t, s) \sim \texttt{Bin}(N_t N_s, \kappa^+_n(t, s) / n).
    \]
    From this stochastic bound we may conclude that
    \[
    \prob( A_n(t, s) < \Lambda_n(t, s)) \leq \prob(A_n^+(t, s) < \Lambda_n(t, s)).
    \]
    We will now show that the desired bound holds for $A_n^+(t, s)$. We do this in the following steps:
    \begin{enumerate}[label = \Roman*.]
        \item We intersect the event $\{A_n^+(t, s) < \Lambda_n(t, s)\}$ with $\mathcal{V}_{ts}$ and use the law of total probability to transform the mixed-binomial probability into several binomial ones where $\mathcal{V}_{ts}$ is satisfied.
        \item We show that on these binomial probabilities the Chernoff bound may be applied.
        \item We apply the Chernoff bound to achieve an upper-bound, and we show that this upper-bound converges to zero with the rate we require.
    \end{enumerate}

    \paragraph{Step I.} Intersecting with $\mathcal{V}_{ts}$ yields
    \[
    \prob(A_n^+(t, s) < \Lambda_n(t, s)) \leq \prob(\{A_n^+(t, s) < \Lambda_n(t, s)\}\cap \mathcal{V}_{ts}) + \prob(\neg \mathcal{V}_{ts}).
    \]
    Applying Lemma~\ref{lem:well-concentrated types} shows that
    \[
    \prob(A_n^+(t, s) < \Lambda_n(t, s)) \leq \prob(\{A_n^+(t, s) < \Lambda_n(t, s)\}\cap \mathcal{V}_{ts}) + 2 \exp(-\log(n)^2 / 2).
    \]
    Thus, the lemma is true when the leftover probability converges to zero with a faster rate than $ \exp(-\log(n)^2 / 2)$. We will now apply the law of total probability on this leftover probability, conditioning on the value of $N_t$ and $N_s$, to find
    \begin{equation}\label{eq:Law total prob undershoot}
    \prob(\{A_n^+(t, s) < \Lambda_n(t, s)\}\cap \mathcal{V}_{ts}) = \expec[\prob(\{A_n^+(t, s) < \Lambda_n(t, s)\}\cap \mathcal{V}_{ts} \;|\; N_t, N_s)].
    \end{equation}
    For the remainder of the proof we will focus on $\prob(\{A_n^+(t, s) < \Lambda_n(t, s)\}\cap \mathcal{V}_{ts} \;|\; N_t, N_s)$. Note this probability is zero when $N_t$ and $N_s$ are such that $\mathcal{V}_{ts}$ does not hold. If $\mathcal{V}_{ts}$ does hold we can remove the condition, and rewrite the probability as
    \begin{equation*}
    \prob(A_n^+(t, s) - \expec[A_n^+(t, s) \;|\; N_t, N_s] < \Lambda_n(t, s) - \expec[A_n^+(t, s) \;|\; N_t, N_s]\;|\; N_t, N_s).
    \end{equation*}
    Since under the conditioning $A_n^+(t, s)$ is a binomial random variable, we know that $\expec[A_n^+(t, s) \;|\; N_t, N_s] = \kappa^+_n(t, s) N_t N_s / n$. Thus, by setting $\theta(t, s) = \Lambda_n(t, s) - \kappa^+_n(t, s) N_t N_s / n$ we find that the probability we seek to control equals
    \begin{equation}\label{eq:binom deviation mixed binom proof}
    \prob(A_n^+(t, s) - \expec[A_n^+(t, s) \;|\; N_t, N_s] < \theta(t, s) \;|\; N_t, N_s).
    \end{equation}
    
    \paragraph{Step II.} To apply the Chernoff bound (i.e., Theorem 2.21 in \cite{Hofstad2016RandomNetworks}) we need to show that $\theta(t, s)$ is negative. First we bound
    \[
    \theta(t, s) = \lfloor n q_t q_s \kappa(t, s) \rfloor - \frac{\kappa^+_n(t, s) N_t N_s}{n}\leq n q_t q_s \kappa(t, s) - \frac{\kappa(t, s) N_t N_s}{n} - \frac{C n^{-1/2 + \alpha}N_t N_s}{n \sqrt{q_t q_s}}.
    \]
    Recall that we only consider settings in which $\mathcal{V}_{ts}$ is satisfied. Thus, we have that $N_t \geq q_t n - \log(n) \sqrt{q_t n}$ and $N_s \geq q_s n - \log(n) \sqrt{q_s n}$. Using these facts on the already existing upper-bound of $\theta(t, s)$ yields the larger upper-bound    
    \begin{align}
        & \kappa(t, s) q_t q_s n - \frac{\kappa(t, s) (q_t n - \log(n) \sqrt{q_t n} ) (q_s n - \log(n) \sqrt{q_s n})}{n} - \frac{C (q_t n - \log(n) \sqrt{q_t n} ) (q_s n - \log(n) \sqrt{q_s n})}{n^{3/2 - \alpha} \sqrt{q_t q_s}},\nonumber \\
        &\leq \kappa(t, s) \log(n) n^{1/2} (\sqrt{q_t} + \sqrt{q_s}) \sqrt{q_t q_s} - C n^{1/2 + \alpha} \sqrt{q_t q_s} + C \log(n) n^\alpha (\sqrt{q_t} + \sqrt{q_s}),\nonumber \\
        &= \sqrt{q_t q_s} \left( \kappa(t, s) \log(n) n^{1/2} (\sqrt{q_t} + \sqrt{q_s}) - C n^{1/2 + \alpha} + C \log(n) n^\alpha \cdot \frac{\sqrt{q_t} + \sqrt{q_s}}{\sqrt{q_t q_s}}  \right). \label{eq:goal upper binom}
    \end{align}
    In the final inequality we have removed all additional negative terms, and noted that the term $\kappa(t, s) q_t q_s n$ cancels out. We now deal with the two remaining positive terms and show they are dominated by the negative term. We first deal with the last positive term. Since we have assumed $t, s \leq u_n^\uparrow(\tau)$ we have that $q_t, q_s \geq n^{-1 + \tau}$. Thus, we have that
    \begin{equation}\label{eq:sqrt bound}
    \frac{\sqrt{q_t} + \sqrt{q_s}}{\sqrt{q_t q_s}} \leq q_t^{-1/2} + q_s^{-1/2} \leq 2 n^{1/2 - \tau / 2}.
    \end{equation}
    If we consider the first positive term, we note from Assumption~\ref{ass:kernel bound} that it implies $\kappa(t, s) \leq 1/\sqrt{q_t q_s}$ for $t, s \to \infty$. Hence, for $n$ large we have through \eqref{eq:sqrt bound} that
    \[
    \kappa(t, s) (\sqrt{q_t} + \sqrt{q_s} ) \leq \frac{\sqrt{q_t} + \sqrt{q_s}}{\sqrt{q_t q_s}} \leq 2 n^{1/2 - \tau /2}.
    \]
    Substituting this, together with \eqref{eq:sqrt bound}, back into \eqref{eq:goal upper binom} yields (for $n \to \infty$)
    \[
    \theta(t, s) \leq \sqrt{q_t q_s} \left( \log(n) n^{1 - \tau/2} - C n^{1/2 + \alpha} + C \log(n) n^{\alpha + 1/2 - \tau / 2} \right).
    \]
    Recall from Assumption~\ref{ass:kernel bound} that $1/2 - \tau/2 < \alpha < 1/2$. Thus, in particular we have that $1/2 + \alpha > 1 - \tau/2$, meaning it is the dominant term. In conclusion, we will have for some $\widehat{C} > 0$ that
    \[
    \theta(t, s) \leq - \widehat{C} n^{1/2 + \alpha} \sqrt{q_t q_s}.
    \]
    This is negative, and hence the Chernoff bound can be applied.

    \paragraph{Step III.} With the result from Step II we may now apply the Chernoff bound on \eqref{eq:binom deviation mixed binom proof} to find that
    \begin{equation}\label{eq:cher applied undershoot}
        \prob(A_n^+(t, s) < \Lambda_n(t, s) \;|\; N_t, N_s) \leq \exp\left( - \frac{\widehat{C}^2 n^{1 + 2 \alpha} q_t q_s}{2 \expec[A_n^+(t, s) \;|\; N_t, N_s]} \right).
    \end{equation}
    We now need to find a useful upper-bound on $\expec[A_n^+(t, s) \;|\; N_t, N_s]$ in order to show \eqref{eq:cher applied undershoot} converges to zero. To do this, we first recall that $\mathcal{V}_{ts}$ is satisfied, and hence that we can use a bound similar to the one applied in Step II. This approach shows $\expec[A_n^+(t, s) \; | \; N_t, N_s]$ can be upper-bounded by
    \[
    \frac{\kappa(t, s) (q_t n + \log(n) \sqrt{q_t n} ) (q_s n + \log(n) \sqrt{q_s n})}{n} + \frac{C (q_t n + \log(n) \sqrt{q_t n} ) (q_s n + \log(n) \sqrt{q_s n})}{n^{3/2 - \alpha} \sqrt{q_t q_s}}.
    \]
    Since $t, s \leq u_n^\uparrow$, note that $q_t n \geq \log(n) \sqrt{q_t n}$ and $q_s n \geq \log(n) \sqrt{q_s n}$. Recall also that $\kappa(t, s) \leq 1/\sqrt{q_t q_s}$ if Assumption~\ref{ass:kernel bound} is satisfied. Using these facts, the previously derived upper-bound becomes
    \begin{equation*} \label{eq:expec undershoot upper}
        \expec[A_n^+(t, s) \; | \; N_t, N_s] \leq 4 \sqrt{q_t, q_s} + 4 C n^{1/2 + \alpha} \sqrt{q_t q_s} \leq 5 C n^{1/2 + \alpha} \sqrt{q_t q_s}.
    \end{equation*}
    We will now substitute this upper-bound into \eqref{eq:cher applied undershoot} and show it still converges to zero. In this computation we will use the fact that for the stable vertex-type $t$ (at tolerance $\tau$) it is true that $q_t \geq n^{-1 + \tau}$. We find
    \[
    \prob(A_n^+(t, s) < \Lambda_n(t, s) \;|\; N_t, N_s) \leq \exp\left( - \frac{\widehat{C}^2 n^{1/2 + \alpha} \sqrt{q_t q_s}}{10} \right) \leq  \exp\left( - \frac{\widehat{C}^2 n^{-1/2 + \alpha + \tau}}{10} \right).
    \]
        
    By recalling from Assumption~\ref{ass:kernel bound} that $\alpha > 1/2 - \tau$, we find that there exists a number $\nu > 0$ such that
    \[
    \prob(A_n^+(t, s) < \Lambda_n(t, s) \;|\; N_t, N_s) \leq \exp\left( - n^\nu \right).
    \]
    Substituting this back into \eqref{eq:Law total prob undershoot} and noting that this bound is uniform in $N_t$ and $N_s$ yields
    \[
    \prob(\{A_n^+(t, s) < \Lambda_n(t, s)\} \cap \mathcal{V}_{t, s}) = \exp\left(-n^\nu\right).
    \]
    Hence, for the original target probability from Step I we find
    \[
    \prob(A_n^+(t, s) < \Lambda_n(t, s)) \leq \exp\left(-n^\nu \right) + 2 \exp(-\log(n)^2 / 2) \leq 2 \exp(-\log(n)^2 / 2).
    \]
    Hence, indeed we find that the statement is true.
\end{proof}

\begin{proof}[Proof of Lemma~\ref{lem:overshoot mixed binomials}]
    This proof is similar to the proof of Lemma~\ref{lem:undershoot mixed binomials}, hence we will not provide the same amount of detail as we did in its proof. Instead, we will mainly focus on the differences. For two stable $t, s \in \mathcal{S}$ we define
    \[
    \kappa_n^-(t, s) := ( \kappa(t, s) - C n^{-1/2 + \alpha}/\sqrt{q_t q_s} ) \wedge 0.
    \]
    By recalling Definition~\ref{def:arc to vertex type} we then we have the stochastic bound
    \[
    A_n(t, s) \preceq A^-_n(t, s) \sim \texttt{Bin}(N_t N_s, \kappa_n^-(t, s) / n),
    \]
    allowing us to conclude that
    \[
    \prob(A_n(t, s) > \Lambda_n(t, s)) \leq \prob(A^-_n(t, s) > \Lambda_n(t, s)).
    \]
    We will now apply same three steps as in the proof of Lemma~\ref{lem:undershoot mixed binomials}.

    \paragraph{Step I.} Intersecting with $\mathcal{V}_{ts}$, applying Lemma~\ref{lem:well-concentrated types} and applying the law of total probability yields
    \begin{equation}\label{eq:overshoot}
    \prob(A^-_n(t, s) > \Lambda_n(t, s)) \leq \expec[\prob(\{A_n^-(t, s) > \Lambda_n(t, s)\}\cap \mathcal{V}_{ts} \;|\; N_t, N_s)] + 2 \exp(-\log(n)^2 / 2).
    \end{equation}
    We will now focus on $\prob(\{A_n^-(t, s) > \Lambda_n(t, s)\}\cap \mathcal{V}_{ts} \;|\; N_t, N_s)$, and note due to the intersection that we can assume $\mathcal{V}_{t s}$ to be satisfied. Hence, similar to \eqref{eq:binom deviation mixed binom proof} we can write this probability as
    \[
    \prob(A_n^-(t, s) - \expec[A_n^-(t, s) \;|\; N_t, N_s] > \theta(t, s) \;|\; N_t, N_s),
    \]
    where $\theta(t, s) = \Lambda_n(t, s) - \expec[A_n^-(t, s) \;|\; N_t, N_s]$. To apply the Chernoff bound (i.e., Theorem 2.21 in \cite{Hofstad2016RandomNetworks}), we need to show this parameter is positive.

    \paragraph{Step II.} We first bound $\theta(t, s)$ as
    \[
    \theta(t, s) = \lfloor n q_t q_s \kappa(t, s) \rfloor - \frac{\kappa^-(t, s) N_t N_s}{n} \geq n q_t q_s \kappa(t, s) - \frac{\kappa(t, s) N_t, N_s}{n} + \frac{C N_t N_s}{n^{3/2 - \alpha} \sqrt{q_t q_s}} - 1.
    \]
    Now, since $\kappa_n^-(t, s) > 0$ and since we can assume $\mathcal{V}_{ts}$ to be satisfied, we note that we can create a further lower-bound by using $N_t \leq q_t n + \log(n) \sqrt{q_t n}$ and $N_s \leq q_s n + \log(n) \sqrt{q_s n}$. Substituting these bounds, and simplifying yields
    \[
    \theta(t, s) \geq \sqrt{q_t q_s} \left( - \kappa(t, s) \log(n) n^{1/2} (\sqrt{q_t} + \sqrt{q_s}) - \log(n)^2 + C n^{1/2 + \alpha} - \frac{1}{\sqrt{q_tq_s}} \right).
    \]
    Here, we have removed additional positive terms. We now want the sole positive term to dominate. We have already covered domination over the first negative term in the proof of Lemma~\ref{lem:undershoot mixed binomials}. The argument is given after \eqref{eq:sqrt bound}. We also trivially see that the positive term dominates the second negative term (the logarithm). To see domination over the third negative term, we use the fact that $t$ and $s$ are stable at tolerance $\tau$ to conclude $1/\sqrt{q_t q_s} \leq n^{1 - \tau}$. By recalling from Assumption~\ref{ass:kernel bound} that $\alpha > 1/2 - \tau/2$ we find that $n^{1/2 + \alpha} > n^{1 - \tau /2}$, which dominates $n^{1 - \tau}$. Hence, indeed we see there exists a constant $\widehat{C} > 0$ such that
    \[
    \theta(t, s) \geq \widehat{C} n^{1/2 + \alpha} \sqrt{q_t q_s}.
    \]
    This is the positivity we required.
    \paragraph{Step III.} Applying the Chernoff bound shows 
    \[
    \prob(A_n^-(t, s) > \Lambda_n(t, s) \;|\; N_t, N_s) \leq \exp\left( - \frac{\widehat{C}^2 n^{1 + 2 \alpha} q_i q_j}{2 \expec[A_n^-(t, s) \;|\; N_t, N_s] } \right).
    \]
    We will now upper-bound the expectation inside the exponential function. Recall the definition of $A^+_n(t, s)$ in the proof of Lemma~\ref{lem:undershoot mixed binomials} and note that
    \[
    \expec[A_n^-(t, s) \;|\; N_t, N_s] = \frac{\kappa^-_n(t, s) N_t N_s}{n} \leq \frac{\kappa^+_n(t, s) N_t N_s}{n} = \expec[A_n^+(t, s) \;|\; N_t, N_s].
    \]
    Substituting this into the Chernoff bound we just obtained unveils the bound in \eqref{eq:cher applied undershoot}. From here, repeating the arguments of Step III in the proof of Lemma~\ref{lem:undershoot mixed binomials} shows there is a $\nu > 0$ such that
    \[
    \prob(A_n^-(t, s) > \Lambda_n(t, s)  \;|\; N_t, N_s) \leq \exp\left( - n^\nu \right).
    \]
    Substituting this back into \eqref{eq:overshoot} yields the desired result.
\end{proof}

\begin{proof}[Proof of Lemma~\ref{lem:monotonicity ARD}]
     We will only show the increasing case, since the proof for the decreasing case is similar. Suppose $\mathcal{Q}_n$ is increasing. We couple $\ARD_n(T, \Lambda_n)$ and $\ARD_n(T, \Lambda_n')$ using the following procedure:
    \begin{enumerate}
        \item Generate the types of each vertex.
        \item First, for each pair of types $t, s \in \mathcal{S}$ choose $\Lambda_n(t, s)$ vertex pairs where the first vertex has type $t$ and the second $s$ uniformly at random from all possible pairs without replacement. This is the realisation of $\ARD_n(T, \Lambda_n)$.
        \item Then, for each pair of types $t, s \in \mathcal{S}$ choose $\Lambda'_n(t,s) - \Lambda_n(t, s) \geq 0$ of the remaining vertex pairs where the first has type $t$ and the second $s$ uniformly at random without replacement. This provides the realisation of $\ARD_n(T, \Lambda_n')$.
    \end{enumerate}
    Since $\Lambda'_n(t,s) - \Lambda_n(t, s) \geq 0$ we can see that step 2 and 3 in the above procedure are equivalent to choosing $\Lambda'_n(t,s)$ pairs of vertices where the first has type $t$ and the second $s$ uniformly at random without replacement. This is what is needed in Step 2 and 3 of $\texttt{ARD}_n(T, \Lambda_n')$ (cf. Section~\ref{sec:ARD}). Under this coupling we also have that $\ARD_n(T, \Lambda_n) \subset \ARD_n(T, \Lambda_n')$. Thus, by virtue of $\mathcal{Q}_n$ being an increasing event, we have that $\ARD_n(T, \Lambda_n) \in \mathcal{Q}_n$ implies $\ARD_n(T, \Lambda_n') \in \mathcal{Q}_n$, letting us conclude that $\prob(\ARD_n(T, \Lambda_n) \in \mathcal{Q}_n) \leq \prob(\ARD_n(T, \Lambda_n') \in \mathcal{Q}_n)$. 
\end{proof}

\subsection{Proofs of lemmas for Theorem~\ref{thm:CCI to IRD}}

\begin{proof}[Proof of Lemma~\ref{lem:kappa bounded}]
    Set $\lambda^\downarrow := \inf_i\{\lambda_i : \lambda_i > 0\}$ and similarly $\varrho^\downarrow := \inf_j \{ \varrho_j : \varrho_j > 0\}$. From Assumption~\ref{ass:CCI} we have that $\lambda^\downarrow, \varrho^\downarrow > 0$. Substituting these into \eqref{eq:asymp connection number} gives us the upper bound
    \[
    \kappa(t, s) \leq \frac{\mu}{\lambda^\downarrow \varrho^\downarrow} \sum_{i = 1}^\infty \sum_{j = 1}^\infty p_{ij} I(t, i) J(s, j).
    \]
    Finally, by noting that $I$ and $J$ are indicators, and that $(p_{ij})_{ij}$ is a probability mass function, we find an upper-bound that is uniform in $t$ and $s$, proving the claim.
    \[
    \kappa(t, s) \leq \frac{\mu}{\lambda^\downarrow \varrho^\downarrow} \sum_{i = 1}^\infty \sum_{j = 1}^\infty p_{ij} =\frac{\mu}{\lambda^\downarrow \varrho^\downarrow}.
    \]
\end{proof}

\begin{proof}[Proof of Lemma~\ref{prob:unstable}]
    Fix an arbitrary number $r \in (0, \delta)$. Since $\expec[T^\delta] < \infty $ we have that $q_t < t^{-1 - \delta}$ (for all $t$ sufficiently large), implying that
    \begin{equation}\label{eq:q power upper}
            q_t^{(1 + r)/(1 + \delta)} < t^{-1 - r}.
    \end{equation}
    We will now split up the desired probability as
    \[
    \prob(T > u_n^\uparrow(\tau)) = \sum_{t = u_n^\uparrow}^\infty q_t = \sum_{t = u_n^\uparrow}^\infty q_t^{(\delta - r)/(1 + \delta)} q_t^{(1 + r)/(1 + \delta)}.
    \]
    From Definition~\ref{def:stable vertices} we have that $q_t < n^{-1 + \tau}$. Using this fact yields
    \[
    \prob(T > u_n^\uparrow(\tau))= \sum_{t = u_n^\uparrow}^\infty q_t^{(\delta - r)/(1 + \delta)} q_t^{(1 + r)/(1 + \delta)} \leq n^{(\tau - 1)(\delta - r)/(1 + \delta)} \sum_{t = u_n^\uparrow}^\infty q_t^{(1 + r)/(1 + \delta)}.
    \]
    Using \eqref{eq:q power upper} in the remaining sum yields the desired result (by noticing that the leftover sum converges).
    \[
    \prob(T > u_n^\uparrow(\tau)) \leq n^{(\tau - 1)(\delta - r)/(1 + \delta)} \sum_{t = 1}^\infty q_t^{(1 + r)/(1 + \delta)} \leq n^{(\tau - 1)(\delta - r)/(1 + \delta)} \sum_{t = 1}^\infty \frac1{t^{1 + r}} = \widehat{C}_r \cdot n^{\frac{(\tau - 1)(\delta - r)}{1 + \delta}}.
    \]
\end{proof}

\begin{proof}[Proof of Lemma~\ref{lem:self-loops}]
    Fix a constant $\nu < \tau$. First, apply the union bound to find that
    \begin{equation}\label{eq:goal self-loops}
    \prob\left( \bigcup_{t = 1}^{u_n^\uparrow(\tau)} \left\{ S_t > n^{1 - \nu} \right\} \right) \leq \sum_{t = 1}^{u_n^\uparrow} \prob(S_t > n^{1 - \nu}).
    \end{equation}
    Now, conditioning on $\mathcal{V}_{tt}$ (cf. Definition~\ref{def:well-concentrated types}) and applying Lemma~\ref{lem:well-concentrated types} shows that
    \begin{equation}\label{eq:intersect Vtt}
    \prob(S_t > n^{1 - \nu}) \leq \prob(\{S_t > n^{1 - \nu} \} \cap \mathcal{V}_{tt}) + 2 \exp\left( - \log(n)^2/2 \right) .
    \end{equation}
    Now, we apply the law of total probability on the remaining probability condition on the number of vertices with type $t$. We find $\prob(S_t > n^{1 - \nu}) = \expec[\prob(\{S_t > n^{1 - \nu} \} \cap \mathcal{V}_{tt} \;|\; N_t)]$. Using this, we note that there will be $N_t^2$ pairs of vertices with type $t$ between which an arc can be placed and $N_t$ of these will create a self-loop. Thus, for each arc the probability that it creates a self-loop within type $t$ is upper bounded by $N_t / N_t^2 = 1/ N_t$. Thus, we have that $S_t \preceq \texttt{Bin}(\lfloor \mu n \rfloor, 1 / N_t)$.
    
    Furthermore, note that $\mathcal{V}_{tt}$ either happens or not based on the value of $N_t$. Specifically, it stipulates that $N_t \geq q_t n - \log(n) \sqrt{q_t n}$. This means we can further stochastically bound
    \[
    S_t \preceq \texttt{Bin}\left(\lfloor \mu n \rfloor, \frac{1}{q_t n - \log(n) \sqrt{q_t n}}  \right) =: B_t.
    \]
    Altogether, these arguments show that
    \[
    \expec[\prob(\{S_t > n^{1 - \nu} \} \cap \mathcal{V}_{tt} \;|\; N_t)] \leq \prob(B_t > n^{1 - \nu}) \leq \prob(|B_t - \expec[B_t] | > n^{1 - \nu} - \expec[B_t]).
    \]
    We now seek to apply the Chernoff bound on this binomial probability. For this, we show that $n^{1 - \nu} - \expec[B_t] > 0$. A straightforward calculation first yields
    \[
    n^{1 - \nu} - \expec[B_t] = n^{1 - \nu} - \frac{\lfloor \mu n \rfloor}{q_t - \log(n) \sqrt{q_t n}} \geq n^{1 - \nu} - \frac{ \mu n }{q_t - \log(n) \sqrt{q_t n}}.
    \]
    Now, note that $t \leq u_n^\uparrow$ which implies that $q_t \geq n^{-1 + \tau}$, resulting in the observation that for $n$ large $q_t n > 2 \log(n) \sqrt{q_t n} $. Thus, this shows us that
    \[
    n^{1 - \nu} - \expec[B_t] \geq n^{1 - \nu} - \frac{2 \mu}{q_t} \geq n^{1 - \nu} - 2 \mu n^{1 - \tau}.
    \]
    Now, since we have chosen $\nu < \tau$ we indeed find that $n^{1 - \nu} - \expec[B_t] > 0$. Moreover, if we apply the Chernoff bound, then the choice $\nu < \tau$ even ensures there exists a constant $\widehat{C} > 0$ such that
    \[
    \prob(|B_t - \expec[B_t] | > n^{1 - \nu} - \expec[B_t]) \leq \exp\left( - \widehat{C} n^{1 - \nu} \right).
    \]
    Substituting this together with \eqref{eq:intersect Vtt} into \eqref{eq:goal self-loops} yields
    \[
    \prob(S_t > n^{1 - \nu}) \leq u_n^\uparrow \exp\left( - \widehat{C} n^{1 - \nu} \right) + 2 u_n^\uparrow \exp\left( - \log(n)^2/2 \right) .
    \]
    By noting from Lemma~\ref{lem:no stable} that $u_n^\uparrow$ is polynomial, we see that both terms in this sum decay super-polynomially to zero, verifying the statement of the lemma. 
\end{proof}

\begin{proof}[Proof of Lemma~\ref{lem:multi-arcs}]
    The proof is very similar to the proof of Lemma~\ref{lem:self-loops}, so we will mainly highlight differences. Fix a number $\nu < 2 \tau - 1$. We first apply the union bound and intersect with $\mathcal{V}_{ts}$ (cf. Definition~\ref{def:stable vertices}) to find through Lemma~\ref{lem:well-concentrated types} for any large number $p$ that
    \begin{equation}\label{eq:multi-arc goal}
    \prob\left( \bigcup_{t = 1}^{u_n^\uparrow} \bigcup_{s = 1}^{u_n^\uparrow} \{M_{ts} > n^{1 - \nu}\} \right) \leq \sum_{t = 1}^{u_n^\uparrow} \sum_{s = 1}^{u_n^\uparrow} \prob(\{M_{ts} > n^{1 - \nu}\} \cap \mathcal{V}_{ts}) + 2 (u_n^\uparrow)^2 \exp\left( - \log(n)^2/2 \right).
    \end{equation}
    We now use the law of total probability to condition on the outcomes of both $N_t$ and $N_s$. If we know these values, then in our model the worst-case scenario for preventing multi-arcs would be the situation that $\mu n - 1$ arcs have already been placed between unique vertex-pairs with type $t$ and $s$, respectively. So, the probability for a new arc to form a multi-arc from a vertex with type $t$ to a vertex with type $s$ would be bounded by $\frac{\mu n}{N_t N_s}$. Once again $\mathcal{V}_{ts}$ stipulates that $N_t \geq q_t n - \log(n) \sqrt{q_t n}$ and $N_s \geq q_s n - \log(n) \sqrt{q_s n}$. Thus, together with the previously derived multi-arc probability upper-bound, we may conclude that
    \[
    M_{ts} \preceq \texttt{Bin}\left(\lfloor \mu n \rfloor, \frac{\mu n}{N_t N_s} \right) \preceq \texttt{Bin}\left(\lfloor \mu n \rfloor, \frac{\mu n}{(q_t n - \log(n) \sqrt{q_t n})(q_s n - \log(n) \sqrt{q_s n})} \right) =: B_{ts}.
    \]
    Hence, we find that
    \begin{align*}
           \expec[ \prob(\{M_{ts} > n^{1 - \nu}\} \cap \mathcal{V}_{ts}) \;|\; N_t, N_s] &\leq \prob(|B_{ts} - \expec[B_{ts}] | > n^{1 - \nu} - \expec[B_{ts}]).
    \end{align*}
    We now show that $B_{ts} - \expec[B_{ts}] > 0$. Recall from the proof of Lemma~\ref{lem:self-loops} that $q_t n \geq \sqrt{q_t n}$ and $q_s n \geq \sqrt{q_s n}$ when $t, s \leq u_n^\uparrow$. Moreover, we also have that $q_t, q_s \geq n^{-1 + \tau}$. Hence, we may conclude for $n$ large that
    \[
    n^{1 - \nu} - \expec[B_{ts}] \geq n^{1 - \nu} - \frac{\mu^2 n^2}{(q_t n - \log(n) \sqrt{q_t n})(q_s n - \log(n) \sqrt{q_s n})} \geq n^{1 - \nu} - \frac{2 \mu}{q_t q_s} \geq n^{1 - \nu} - 2 \mu n^{1 - ( 2\tau - 1)}.
    \]
    Now, since we have assumed that $\nu < 2 \tau - 1$ we indeed find that $n^{1 - \nu} - \expec[B_{ts}] > 0$. When we apply the Chernoff bound, the inequality $\nu < 2 \tau - 1$ specifically shows that there exists a $\widehat{C} > 0$ such that
    \[
    \prob(|B_{ts} - \expec[B_{ts}] | > n^{1 - \nu} - \expec[B_{ts}]) \leq \exp\left( - \widehat{C} n^{1 - \nu} \right).
    \]
    Substituting this back into \eqref{eq:multi-arc goal} shows
    \[
     \prob\left( \bigcup_{t = 1}^{u_n^\uparrow} \bigcup_{s = 1}^{u_n^\uparrow} \{M_{ts} > n^{1 - \nu}\} \right) \leq (u_n^\uparrow)^2 \exp\left( - \widehat{C} n^{1 - \nu} \right) + 2 (u_n^\uparrow)^2 \exp\left( - \log(n)^2/2 \right).
    \]
    Like at the end of the proof of Lemma~\ref{lem:self-loops}, the statement is true due to the fact that $u_n^\uparrow$ is polynomial in size (cf. Lemma~\ref{lem:no stable}).
\end{proof}

\begin{proof}[Proof of Lemma~\ref{lem:arc vertex prob stable}]
    Let $C_a$ denote the colour assigned to arc $a \in [\lfloor \mu n \rfloor]$, and denote by $(N_k)_{k \in \mathcal{S}}$ the sequence that records for each vertex-type the amount of vertices of said type. Finally, set a dummy upper-bound (cf. Definition~\ref{def:stable vertices}) \[M_n := u_n^\uparrow\left( 1 - \frac{2 + \varepsilon}{2(1 + \varepsilon)} \right).\] This dummy value serves to ensure that $N_k$ concentrates for as many vertex-types simultaneously as possible. Note from step 3--5 in the generation algorithm for $\texttt{CCI}_{n, \mu}$ that
    \begin{equation}\label{eq:goal arc vertex prob stable}
           \prob\left(\mathcal{A}_{ts}^{(a)} \; \Bigg{|}\; \bigcap_{k = 1}^{M_n} \bigcap_{k' = 1}^{M_n} \mathcal{V}_{kk'}, \mathcal{V}_{ts}, C_a = (i, j), (N_k)_{k\in \mathcal{S}}\right) = \frac{N_t N_s I(t, i) J(s, j)}{\left( \sum_{k \in \mathcal{S}} N_k I(k, i) \right) \left(\sum_{k' \in \mathcal{S}} N_{k'} J(k', j) \right) }, 
    \end{equation}
    Where $\mathcal{V}_{ts}$ is the event from Definition~\ref{def:well-concentrated types}.
    Now, since the intersection of the $\mathcal{V}_{kk'}$ events and $\mathcal{V}_{ts}$ occurs, we have for all types $k \leq M_n$ and $k \in \{t, s\}$ that
    \begin{equation}\label{eq:bounds well conentrated vertex types}
        n q_k - \log(n) \sqrt{n q_k} \leq N_k \leq n q_k + \log(n) \sqrt{n q_k}. 
    \end{equation}
    In essence, the rest of the proof consists of using \eqref{eq:bounds well conentrated vertex types} to find bounds on \eqref{eq:goal arc vertex prob stable} independent of $(N_k)_{k \in \mathcal{S}}$, and then use the law of total probability on these bounds to achieve the desired result. More precisely, we shall derive the following upper and lower bound:
    \begin{equation}\label{eq:goal Step I proof arc prob stable}
         \frac{N_t N_s I(t, i) J(s, j)}{\left( \sum_{k \in \mathcal{S}} N_k I(k, i) \right) 
         	\left(\sum_{k' \in \mathcal{S}} N_{k'} J(k', j) \right) }
         \le \frac{q_t q_s I(i, t) J(j, s)}{\lambda_i \varrho_j} + \tilde{C}^\uparrow \log(n) n^{- \frac{\tau}2},
    \end{equation}
    and
    \begin{equation}\label{eq:goal step II arc prob stable}
         \frac{N_t N_s I(t, i) J(s, j)}{\left( \sum_{k \in \mathcal{S}} N_k I(k, i) \right) 
         \left(\sum_{k' \in \mathcal{S}} N_{k'} J(k', j) \right) }
         \ge \frac{q_t q_s I(i, t) J(j, s)}{\lambda_i \varrho_j} - \tilde{C}^\downarrow  \log(n) n^{- \frac{\tau}2},
    \end{equation}
    for some constants $\tilde{C}^\uparrow$ and $\tilde{C}^\downarrow$. We first finish the proof using these bounds. We use the law of total probability to write
        \begin{equation}\label{eq:goal step III arc prob stable}
                \prob(\mathcal{A}^{(a)}_{ts}) = \sum_{i = 1}^\infty \sum_{j = 1}^\infty p_{ij} \prob(\mathcal{A}^{(a)}_{ts} \;|\; C_a = (i, j)).
        \end{equation}
        Now, we condition on the intersection of $\mathcal{V}_{kk'}$ events and $\mathcal{V}_{ts}$ to find that
        \begin{align*}
        \prob(\mathcal{A}^{(a)}_{ts} \;|\; C_a = (i, j)) &= \prob\left(\mathcal{A}^{(a)}_{ts} \;\Bigg{|}\; C_a = (i, j),  \bigcap_{k = 1}^{M_n} \bigcap_{k' = 1}^{M_n} \mathcal{V}_{kk'}, \mathcal{V}_{ts} \right) \prob\left( \bigcap_{k = 1}^{M_n} \bigcap_{k' = 1}^{M_n} \mathcal{V}_{kk'}, \mathcal{V}_{ts} \;\Bigg{|}\; C_a = (i, j)\right),\\
        &+ \prob\left(\mathcal{A}^{(a)}_{ts} \;\Bigg{|}\; C_a = (i, j),   \bigcup_{k = 1}^{M_n} \bigcup_{k' = 1}^{M_n} \neg \mathcal{V}_{kk'} \cup \neg \mathcal{V}_{ts} \right) \prob\left(  \bigcup_{k = 1}^{M_n} \bigcup_{k' = 1}^{M_n} \neg \mathcal{V}_{kk'} \cup \neg \mathcal{V}_{ts} \;\Bigg{|}\; C_a = (i, j)\right).
        \end{align*}
        By noting from Step 1 and 2 from the $\CCI_{n, \mu}$ generation algorithm that all $\mathcal{V}_{kk'}$ and $C_a$ are independent this reduces into
        \begin{equation}\label{eq:step III arc prob stable conditioning}
        \begin{aligned}
        \prob(\mathcal{A}^{(a)}_{ts} \;|\; C_a = (i, j)) &= \prob\left(\mathcal{A}^{(a)}_{ts} \;\Bigg{|}\; C_a = (i, j),  \bigcap_{k = 1}^{M_n} \bigcap_{k' = 1}^{M_n} \mathcal{V}_{kk'}, \mathcal{V}_{ts}\right) \prob\left( \bigcap_{k = 1}^{M_n} \bigcap_{k' = 1}^{M_n} \mathcal{V}_{kk'}, \mathcal{V}_{ts} \right),\\
        &+ \prob\left(\mathcal{A}^{(a)}_{ts} \;\Bigg{|}\; C_a = (i, j),   \bigcup_{k = 1}^{M_n} \bigcup_{k' = 1}^{M_n} \neg \mathcal{V}_{kk'} \cup \neg \mathcal{V}_{ts} \right) \prob\left(  \bigcup_{k = 1}^{M_n} \bigcup_{k' = 1}^{M_n} \neg \mathcal{V}_{kk'} \cup \neg \mathcal{V}_{ts} \right).
        \end{aligned}
        \end{equation}
        Now we use $(N_k)_{k \in \mathcal{S}}$-independent bounds \eqref{eq:goal Step I proof arc prob stable} and \eqref{eq:goal step II arc prob stable} to create upper- and lower-bounds of \eqref{eq:step III arc prob stable conditioning}. Specifically, it shows that 
        \begin{align*}
             \prob(\mathcal{A}^{(a)}_{ts} \;|\; C_a = (i, j)) &\leq \frac{q_t q_s I(i, t) J(j, s)}{\lambda_i \varrho_j} + \tilde{C}^\uparrow  \log(n) n^{- \frac{\tau}2}   +  \prob\left(  \bigcup_{k = 1}^{M_n} \bigcup_{k' = 1}^{M_n} \neg \mathcal{V}_{kk'}\cup \neg \mathcal{V}_{ts}\right), \text{ and} \\
             \prob(\mathcal{A}^{(a)}_{ts} \;|\; C_a = (i, j)) &\geq \frac{q_t q_s I(i, t) J(j, s)}{\lambda_i \varrho_j} - \tilde{C}^\downarrow  \log(n) n^{- \frac{\tau}2} -  \prob\left(  \bigcup_{k = 1}^{M_n} \bigcup_{k' = 1}^{M_n} \neg \mathcal{V}_{kk'} \cup \neg \mathcal{V}_{ts}\right).
        \end{align*}
        We use the union bound on the $\mathcal{V}_{kk'}$-terms and use Lemma~\ref{lem:well-concentrated types} to create further upper- and lower-bounds given by
        \begin{align*}
             \prob(\mathcal{A}^{(a)}_{ts} \;|\; C_a = (i, j)) &\leq \frac{q_t q_s I(i, t) J(j, s)}{\lambda_i \varrho_j} + \tilde{C}^\uparrow  \log(n) n^{- \frac{\tau}2}  +  2 (M_n^2 + 1) \exp\left( - \log(n)^2/2 \right), \text{ and}\\
             \prob(\mathcal{A}^{(a)}_{ts} \;|\; C_a = (i, j)) &\geq \frac{q_t q_s I(i, t) J(j, s)}{\lambda_i \varrho_j} - \tilde{C}^\downarrow  \log(n) n^{- \frac{\tau}2} - 2 (M_n^2 + 1) \exp\left( - \log(n)^2/2 \right).
        \end{align*}
        Using Lemma~\ref{lem:no stable} to conclude that $M_n \leq n$, we find that the final terms are not dominant. Substituting these bounds back into \eqref{eq:goal step III arc prob stable} yields
        \[
        	\frac{q_t q_s \kappa(t, s)}{\mu} - \widehat{C}  \log(n) n^{- \frac{\tau}2}
        	\le \prob(\mathcal{A}^{(a)}_{ts}) \le 
        	\frac{q_t q_s \kappa(t, s)}{\mu} + \widehat{C}   \log(n) n^{- \frac{\tau}2},
        \]
        for some $\widehat{C} > 0$. Together, these two bounds indeed show that
        \[
        \left| \prob(\mathcal{A}^{(a)}_{ts}) - \frac{q_t q_s \kappa(t, s)}{\mu}  \right| \leq \widehat{C}  \log(n) n^{- \frac{\tau}2},
        \]
    	which proves the main result. What is left now is to establish the upper and lower bound~\eqref{eq:goal Step I proof arc prob stable} and~\eqref{eq:goal step II arc prob stable}.

    \paragraph{The upper bound \eqref{eq:goal Step I proof arc prob stable}.} Define the indicator
    \[
    \hat{I}_k := \1\{q_k n - \log(n) \sqrt{n q_k } \geq 0\}.
    \]
    Substituting the upper-bound of \eqref{eq:bounds well conentrated vertex types} in the numerator of \eqref{eq:goal arc vertex prob stable}, and substituting the lower bound of \eqref{eq:bounds well conentrated vertex types} in the denominator of \eqref{eq:goal arc vertex prob stable} yields the upper-bound of \eqref{eq:goal arc vertex prob stable} given by
    \[
    \frac{(q_t n + \log(n) \sqrt{q_t n})(q_s n + \log(n) \sqrt{q_s n}) I(i, t) J(j, s)}{\left( \sum_{k \in \mathcal{S}} [ q_k n - \log(n) \sqrt{q_k n} ] I(i, k) \hat{I}_k \right) \left( \sum_{k' \in \mathcal{S}}[q_{k'} n - \log(n) \sqrt{q_{k'} n}]J(j, k') \hat{I}_{k'} \right)}.
    \]
    Here, the inclusion of the indicators $\hat{I}_k$ is possible in the denominator, since we know $N_k \geq 0$ for all $k \in \mathcal{S}$. Now, we will extend this upper bound by only considering the first $M_n$ terms in the sum. Note for these terms that $\hat{I}_k = 1$ (cf. Definition~\ref{def:stable vertices}). Thus, we find the further upper-bound
    \[
    \frac{(q_t n + \log(n) \sqrt{q_t n})(q_s n + \log(n) \sqrt{q_s n}) I(i, t) J(j, s)}{\left( \sum_{k = 1}^{M_n} [ q_k n - \log(n) \sqrt{q_k n} ] I(i, k)  \right) \left( \sum_{k' = 1}^{M_n}[q_{k'} n - \log(n) \sqrt{q_{k'} n}]J(j, k')  \right)}.
    \]
    Next, we expand the products in the upper-bound and remove the additional positive term from the denominator to derive a further upper-bound. It is given by
    \begin{equation*}
           \frac{(q_t q_s n^2 + (\sqrt{q_t} + \sqrt{q_s}) n \log(n) \sqrt{q_t q_s n} + n \log(n)^2 \sqrt{q_t q_s}) I(i, t) J(j, s)}{\sum_{k = 1}^{M_n}\sum_{k' = 1}^{M_n} q_{k}  q_{k'} n^2 I(i, k) J(j, k') - \sum_{k = 1}^{M_n}\sum_{k' = 1}^{M_n} n \log(n) \sqrt{q_k q_{k'} n} (\sqrt{q_k} + \sqrt{q_{k'}})}. 
    \end{equation*}
    Now, we further bound the negative term in the denominator to attain an error sum that only depends on $\sqrt{q_k q_{k'}}$. We find
    \begin{equation}\label{eq:upper bound goal arc prob stable with sqrt sums}
           \frac{(q_t q_s n^2 + (\sqrt{q_t} + \sqrt{q_s}) n \log(n) \sqrt{q_t q_s n} + n \log(n)^2 \sqrt{q_t q_s}) I(i, t) J(j, s)}{\sum_{k = 1}^{M_n}\sum_{k' = 1}^{M_n} q_{k}  q_{k'} n^2 I(i, k) J(j, k') - n \log(n) \sqrt{n} \sum_{k = 1}^{M_n}\sqrt{q_k} \sum_{k' = 1}^{M_n} \sqrt{q_{k'}}}. 
    \end{equation}
    Since $\expec[T^{1 + \varepsilon}] < \infty$ for some $\varepsilon > 0$ (cf. Assumption~\ref{ass:CCI}), it holds that $q_k \leq 1/k^{2 + \varepsilon}$ for $k$ large. Thus, we have that $\sum_{k = 1}^\infty \sqrt{q_k} < \infty$. Using this in \eqref{eq:upper bound goal arc prob stable with sqrt sums} we find there exists a constant $\widehat{C}' > 0$ such that it is upper-bounded by
    \begin{equation}\label{eq:upper bound goal arc prob stable}
           \frac{(q_t q_s n^2 + (\sqrt{q_t} + \sqrt{q_s}) n \log(n) \sqrt{q_t q_s n} + n \log(n)^2 \sqrt{q_t q_s}) I(i, t) J(j, s) }{\sum_{k = 1}^{M_n}\sum_{k' = 1}^{M_n} q_{k}  q_{k'} n^2 I(i, k) J(j, k') - \widehat{C}' \cdot n \log(n) \sqrt{n} }. 
    \end{equation}
    
    Note that the double sum in \eqref{eq:upper bound goal arc prob stable} is close to $\lambda_{i}\varrho_j$ (cf. \eqref{eq:fraction of vertices connecting to lig} and \eqref{eq:fraction of vertices connecting to rec}). We will now make these parameters visible in the denominator by adding and subtracting the remainders of the sum. We find for any $r \in (0, 1 + \varepsilon)$ (cf. Assumption~\ref{ass:CCI}) and some corresponding $\widehat{C}_r > 0$ that 
    \begin{align*}
     \sum_{k = 1}^{M_n}\sum_{k' = 1}^{M_n} q_{k}  q_{k'} n^2 I(i, k) J(j, k') &\geq  n^2 \lambda_i \varrho_j  -  \sum_{k = 1}^{\infty}\sum_{k' = M_n + 1}^{\infty} q_{k}  q_{k'} n^2 -  \sum_{k = u_n^\uparrow + 1}^{\infty}\sum_{k' = 1}^{\infty} q_{k}  q_{k'} n^2,\\
     &\geq n^2 \lambda_i \varrho_j - 2 n^2 \prob(T > M_n),\\
     &\geq n^2 \left( \lambda_i \varrho_j - 2 \widehat{C}_r \cdot n^{-\frac12 + \frac{r}{2(1 + \varepsilon)}} \right).
    \end{align*}
    We used Lemma~\ref{prob:unstable} in the final line of this string of inequalities. Substituting this back into \eqref{eq:upper bound goal arc prob stable} yields
    \[
    \frac{(q_t q_s n^2 + (\sqrt{q_t} + \sqrt{q_s}) n \log(n) \sqrt{q_t q_s n} + n \log(n)^2 \sqrt{q_t q_s}) I(i, t) J(j, s)}{n^2 \left( \lambda_i \varrho_j -  2 \widehat{C}_r \cdot n^{-\frac12 + \frac{r}{2(1 + \varepsilon)}} - \widehat{C}' \cdot \log(n)  n^{-\frac12}   \right)}.
    \]
    Next, we divide everything through by $n^2$ and extract the factor $q_t q_s I(i, t)J(j, s) / (\lambda_i \varrho_j)$ to find the following upper-bound
    \begin{equation}\label{eq:upper bound proof arc prob stable n dep den}
        \frac{q_t q_s I(i, t) J(j, s)}{\lambda_i \varrho_j}\cdot \frac{1 + (\sqrt{q_t} + \sqrt{q_s}) n^{-1/2} \log(n) /\sqrt{q_t q_s} + n^{-1} \log(n)^2 /\sqrt{q_t q_s}}{1 - \widehat{C}_1 \cdot n^{-\frac12 + \frac{r}{2(1 + \varepsilon)}} - \widehat{C}_2 \cdot \log(n) n^{- \frac12}}.
    \end{equation}
    Here, $\widehat{C}_1 = 2\widehat{C}_r/(\lambda^\downarrow \varrho^\downarrow)$ and $\widehat{C}_2 = \widehat{C}' /(\lambda^\downarrow \varrho^\downarrow)$ with $\lambda^\downarrow := \inf_i\{\lambda_i : \lambda_i > 0\}$ and $\varrho^\downarrow := \inf_j\{\varrho_j : \varrho_j > 0\}$. We recall that $\lambda^\downarrow,\varrho^\downarrow > 0$ due to Assumption~\ref{ass:CCI}. By noting that all the $n$-dependent terms in the denominator (and numerator, since $\tau > 0$) of \eqref{eq:upper bound proof arc prob stable n dep den} converge to zero, we can use the Taylor expansion of this error-fraction to conclude there exists a constant $\tilde{C} > 0$ such that it is bounded by
    \[
    1 + \tilde{C} \left( \frac{\log(n)}{\sqrt{n}} \left( \frac1{\sqrt{q_t}} + \frac1{\sqrt{q_s}} \right)  + \frac{\log(n)^2}{n \sqrt{q_t q_s}} + n^{-\frac12 + \frac{r}{2(1 + \varepsilon)}} + \log(n) n^{- \frac12}  \right).
    \]
    Now, recall that $t, s \leq u_n^\uparrow(\tau)$ and hence that $q_t, q_s \geq n^{-1 + \tau}$. Thus to remove $t$ and $s$ dependence we further bound this expression by
    \[
    1 + \tilde{C} \left(2 \log(n) n^{- \frac{\tau}2}  + \log(n)^2 n^{- \tau} + n^{-\frac12 + \frac{r}{2(1 + \varepsilon)}} + \log(n) n^{- \frac12}  \right).
    \]
    From this expression we note that only the first term can be dominant if we choose $r$ sufficiently close to zero. Thus, we find that there exists a constant $\tilde{C}^\uparrow > 0$ such that \eqref{eq:upper bound goal arc prob stable} is bounded from above by
    \[
        \frac{q_t q_s I(i, t) J(j, s)}{\lambda_i \varrho_j} + \tilde{C}^\uparrow \log(n) n^{- \frac{\tau}2},
    \]
    which proves~\eqref{eq:goal Step I proof arc prob stable}.

    \paragraph{The lower bound \eqref{eq:goal step II arc prob stable}.} The proof is similar to the upper bound. The biggest difference is the way we bound the denominator in \eqref{eq:goal arc vertex prob stable} using \eqref{eq:bounds well conentrated vertex types}. For this we first note $\sum_{k = 1}^\infty N_k = n.$ Thus, we can write
    \[
    \sum_{k = M_n + 1}^\infty N_k I(i, k) \leq \sum_{k = M_n + 1}^\infty N_k  = n - \sum_{k = 1}^{M_n} N_k = \sum_{k = 1}^\infty q_k n - \sum_{k = 1}^{M_n} N_k .
    \]
    Now, using the lower bound in \eqref{eq:bounds well conentrated vertex types} we can conclude that
    \begin{align*}
    \sum_{k = M_n + 1}^\infty N_k I(i, k) &\leq  \sum_{k = 1}^\infty q_k n - \sum_{k = 1}^{M_n} \left[ q_k n - \log(n) \sqrt{q_k n} \right] =  \sum_{k = 1}^\infty q_k n + \log(n)\sqrt{n} \sum_{k = 1}^n \sqrt{q_k} - \sum_{k =1}^{M_n} q_k n,\\
    &= \sum_{k = M_n + 1}^\infty q_k n + \log(n)\sqrt{n} \sum_{k = 1}^n \sqrt{q_k}  = \prob(T > M_n) n + \widehat{C}' \cdot \log(n) \sqrt{n}.
    \end{align*}
    Similarly, using the upper bound in \eqref{eq:bounds well conentrated vertex types} we have that
    \[
    \sum_{k = 1}^{M_n} N_k I(i, k) \leq \sum_{k = 1}^{M_n} q_k I(i, k) n + \widehat{C}' \cdot \log(n) \sqrt{n}.
    \]
    Now, using these bounds in the denominator of \eqref{eq:goal arc vertex prob stable} and using the lower-bound of \eqref{eq:bounds well conentrated vertex types} yields the following lower-bound
    \[
    \frac{(q_t n - \log(n) \sqrt{nq_t})(q_s n - \log(n) \sqrt{nq_s}) I(i, t) J(j, s)}{\left( \sum_{k = 1}^{M_n}  q_k n I(i, k) + 2 \widehat{C}'  \log(n) \sqrt{n} + n \prob(T > M_n)  \right) \left( \sum_{k' = 1}^{M_n}q_{k'} n J(i, k') + 2 \widehat{C}' \log(n) \sqrt{n} + n \prob(T > M_n )  \right)}.
    \]
    We expand the factors in this bound, remove the additional positive terms from the numerator, and only keep the dominant terms in the expansion of the denominator. This yields a lower-bound of \eqref{eq:goal arc vertex prob stable} similar to \eqref{eq:upper bound goal arc prob stable} given by
    \[
    \frac{(q_t q_s n^2 - (\sqrt{q_t} + \sqrt{q_s}) n \log(n) \sqrt{nq_tq_s}) I(i, t) J(j, s)}{\sum_{k = 1}^{M_n}\sum_{k' = 1}^{M_n} q_{k}  q_{k'} n^2 I(i, k) J(j, k') + \widetilde{C}' \left( n \log(n) \sqrt{n} + n^2 \prob(T > M_n))\right)},
    \]
    where $\widetilde{C}' > 0$ is some constant. Now, we create a further lower bound by running the sum in the denominator up to infinity (revealing $\lambda_i \varrho_j$), and applying Lemma~\ref{prob:unstable} for the remainder two terms. We find
    \[
     \frac{(q_t q_s n^2 - (q_t + q_s) n \log(n) \sqrt{n}) I(i, t) J(j, s)}{n^2 \left( \lambda_i \varrho_j +   \widehat{C}_r \cdot n^{-\frac12 + \frac{r}{2(1 + \varepsilon)}} + \widehat{C} n^{ - \frac12} \log(n) \right)}.
    \]
    Now, we extract $q_t q_s I(i, t) J(j, s) / (\lambda_i \varrho_j)$ from the fraction to find a lower-bound similar to \eqref{eq:upper bound proof arc prob stable n dep den} given by
    \[
    \frac{q_t q_s I(i, t) J(j, s)}{\lambda_i \varrho_j}\cdot \frac{1 - (\sqrt{q_t} + \sqrt{q_s}) n^{-1/2} \log(n)/\sqrt{q_t q_s}}{1 + \widehat{C}_1 \cdot n^{-\frac12 + \frac{r}{2(1 + \varepsilon)}} + \widehat{C}_2 \cdot \log(n) n^{- \frac12}},
    \]
    where $\widehat{C}_1, \widehat{C}_2 > 0$. Repeating the same arguments as in for the upper bound finally yields the desired lower-bound 
    \[
    	\frac{q_t q_s I(i, t) J(j, s)}{\lambda_i \varrho_j} - \tilde{C}^\downarrow  \log(n) n^{- \frac{\tau}2},
    \]
    for some $\tilde{C}^\downarrow > 0$ and any $r \in (0, 1 + \varepsilon)$.
\end{proof}

\begin{proof}[Proof of Lemma~\ref{lem:arc vertex prob no stable}]
    The approach is similar to the proof of Lemma~\ref{lem:arc vertex prob stable}. However, this time we only focus on an upper-bound. First, we set $M_n := u_n^\uparrow(\nu)$ for some $\nu$ close to zero, and write similar to \eqref{eq:goal arc vertex prob stable} that
    \begin{equation}\label{eq:goal arc vertex prob no stable}
   \prob\left(\mathcal{A}_{ts}^{(a)} \;\Bigg{|}\; \bigcap_{k = 1}^{M_n} \bigcap_{k' = 1}^{M_n} \mathcal{V}_{kk'}, \mathcal{V}_{ts}, C_a = (i, j), (N_k)_{k\in \mathcal{S}} \right) = \frac{N_t N_s I(t, i) J(s, j)}{\left( \sum_{k \in \mathcal{S}} N_k I(k, i) \right) \left(\sum_{k' \in \mathcal{S}} N_{k'} J(k', j) \right) }. 
    \end{equation}
    Now, since both the intersection and $\mathcal{V}_{ts}$ occurs, we have that
    \begin{align}
        n q_k - \log(n) \sqrt{q_k n} \leq N_k \leq n q_k + \log(n) \sqrt{q_k n} \enspace &\text{for } k \leq M_n \text{ and } k \in \{t, s\}. \label{eq:bounds conentrated vertex types good}
    \end{align}

    The rest of the proof consists of using \eqref{eq:bounds conentrated vertex types good} to find an upper-bound of \eqref{eq:goal arc vertex prob no stable} independent of $(N_k)_{k \in \mathcal{S}}$, and then use the law of total probability on these bounds to achieve the desired result. We will proceed in the following steps:
    \begin{enumerate}[label = \Roman*.]
        \item Derive a desirable upper-bound on \eqref{eq:goal arc vertex prob no stable}.
        \item Using the law of total probability on the bounds to find the desired result.
    \end{enumerate}

    \paragraph{Step I.} Repeat the argumentation for the upper bound in the proof of Lemma~\ref{lem:arc vertex prob stable} until \eqref{eq:upper bound proof arc prob stable n dep den}. We find the upper-bound
    \[
    \frac{q_t q_s I(i, t) J(j, s)}{\lambda_i \varrho_j}\cdot \frac{1 + (q_t^{-1/2} + q_s^{-1/2}) n^{-1/2} \log(n) + q_t^{-1/2} q_s^{-1/2} n^{-1} \log(n)^2}{1- \widehat{C}_1 \cdot n^{\frac{(\nu - 1)(1 + \varepsilon - r)}{2 + \varepsilon}} - \widehat{C}_2 \cdot \log(n) n^{- \frac12}}.
    \]
    Note we cannot repeat the Taylor expansion argument from the proof of Lemma~\ref{lem:arc vertex prob stable} here, since the two $n$-dependent terms in the numerator might diverge as $n \to \infty$, due to the instability of either $t$ or $s$. Thus, we slightly rewrite this expression to make the Taylor expansion argument viable again.
    \[
    \frac{\sqrt{q_t q_s} I(i, t) J(j, s)}{\lambda_i \varrho_j}\cdot \frac{\sqrt{q_t q_s} + (\sqrt{q_t} + \sqrt{q_s}) n^{-1/2} \log(n) + n^{-1} \log(n)^2}{1 - \widehat{C}_1 \cdot n^{\frac{(\nu - 1)(1 + \varepsilon - r)}{2 + \varepsilon}} - \widehat{C}_2 \cdot \log(n) n^{- \frac12}}.
    \]
    Now we continue repeating the rest of the arguments for the upper bound in the proof of Lemma~\ref{lem:arc vertex prob stable}, but keep the error-term multiplicative. This yields for some $\tilde{C} > 0$ and any $r \in (0, 1 + \varepsilon)$ that
    \[
    \frac{\sqrt{q_t q_s} I(i, t) J(j, s)}{\lambda_i \varrho_j} \left( \sqrt{q_t q_s} + \tilde{C} \left( \frac{(\sqrt{q_t} + \sqrt{q_s}) \log(n)}{\sqrt{n}} + \frac{\log(n)^2}{n} +  n^{\frac{(\nu - 1)(1 + \varepsilon - r)}{2 + \varepsilon}} +  \frac{\log(n)}{\sqrt{n}} \right)  \right).
    \]
    Due to Assumption~\ref{ass:CCI} we can bound $1/(\lambda_i \varrho_j)$ by a constant. Hence, if we take $\nu$ and $r$ close enough to zero, then there exists a constant $\widehat{C} > 0$ for which we can bound \eqref{eq:goal arc vertex prob no stable} from above by
    \begin{equation}\label{eq:step I proof arc prob no stable done}
            \widehat{C}' \sqrt{q_t q_s} \left(\sqrt{q_t q_s}  + \frac{\log(n)}{\sqrt{n}}   \right).
    \end{equation}

    \paragraph{Step II.}  We use the law of total probability to write
    \begin{equation}\label{eq:goal step III arc prob no stable}
            \prob(\mathcal{A}^{(a)}_{ts}) = \sum_{i = 1}^\infty \sum_{j = 1}^\infty p_{ij} \prob(\mathcal{A}^{(a)}_{ts} \;|\; C_a = (i, j)).
    \end{equation}
    Now, we condition on the intersection of $\mathcal{V}_{kk'}$ events and $\mathcal{V}_{ts}$ to bound
    \begin{align*}
    \prob(\mathcal{A}^{(a)}_{ts} \;|\; C_a = (i, j)) &\leq  \prob\left(\mathcal{A}^{(a)}_{ts} \;\Bigg{|}\; C_a = (i, j), \bigcap_{k = 1}^{u_n^\uparrow} \bigcap_{k' = 1}^{u_n^\uparrow} \mathcal{V}_{kk'}, \mathcal{V}_{ts} \right)  + \sum_{k = 1}^{u_n^\uparrow} \sum_{k' = 1}^{u_n^\uparrow} \prob(\neg \mathcal{V}_{kk'}) + \prob(\neg \mathcal{V}_{ts}).
    \end{align*}
    Using the $(N_k)_{k \in \mathcal{S}}$-independent bound \eqref{eq:step I proof arc prob no stable done} together with Lemma~\ref{lem:well-concentrated types} shows there exists a constant $\widehat{C} > 0$ such that
    \begin{align*}
         \prob(\mathcal{A}^{(a)}_{ts} \;|\; C_a = (i, j)) &\leq \widehat{C} \sqrt{q_t q_s} \left(\sqrt{q_t q_s} +   \frac{\log(n)}{\sqrt{n}}   \right) + 2 ((u_n^\uparrow)^2 + 1)  \exp\left( - \log(n)^2/2 \right).
    \end{align*}
    Substituting this back into \eqref{eq:goal step III arc prob no stable} and computing the sum shows the desired result, since the second term in the upper-bound is super-polynomial (cf. Lemma~\ref{lem:no stable}).
\end{proof}

\begin{proof}[Proof of Lemma~\ref{lem:deviation A double hat}]
    Fix a constant $\alpha > 3/8$. We fix two auxiliary constants $\tau_1 \in (2/3, 2 \alpha)$ and $\tau_2 \in (1/2 - \alpha, 1/2)$ and set $\zeta_n := u_n^\uparrow(\tau_1)$ and $\xi_n := u_n^\uparrow(\tau_2)$. Using these two constants, we will apply the union bound and split up the resulting bound on the target probability as follows: 
    \begin{align*}
       & \sum_{t = 1}^{\zeta_n}\sum_{s = 1}^{\zeta_n} \prob \left( \left|\bbar{A}_n(t, s) - \lfloor \kappa(t, s) q_t q_s n \rfloor \right| > C n^{1/2 + \alpha} \sqrt{q_t q_s} \right)\\
       &+ \sum_{t = \zeta_n}^{\xi_n}\sum_{s = 1}^{\xi_n} \prob \left( \left|\bbar{A}_n(t, s) - \lfloor \kappa(t, s) q_t q_s n \rfloor \right| > C n^{1/2 + \alpha} \sqrt{q_t q_s} \right) \\
       &+ \sum_{t = 1}^{\xi_n}\sum_{s = \zeta_n}^{\xi_n} \prob \left( \left|\bbar{A}_n(t, s) - \lfloor \kappa(t, s) q_t q_s n \rfloor \right| > C n^{1/2 + \alpha} \sqrt{q_t q_s} \right).
    \end{align*}
    Our goal is now to show that each of these sums converge to zero in order to prove the claim. This is what we will do in the rest of the proof.

    \paragraph{First double sum.} In essence, this computation will consist of an application of the Chernoff bound together with Lemma~\ref{lem:arc vertex prob stable}. We will first stochastically bound $\bbar{A}_n(t, s)$ in terms of binomial distributions. To do this, we write the probability inside this sum as
    \begin{equation}\label{eq:first double sum goal}
    \prob \left(\bbar{A}_n(t, s)   > C n^{1/2 + \alpha} \sqrt{q_t q_s} + \lfloor \kappa(t, s) q_t q_s n \rfloor \right) + \prob \left(\bbar{A}_n(t, s) < \lfloor \kappa(t, s) q_t q_s n \rfloor  -  C n^{1/2 + \alpha} \sqrt{q_t q_s}\right).
    \end{equation}
    Denote by $\mathcal{A}_{ts}^{(a)}$ the event that arc $a \in [\mu n]$ gets placed from a vertex with type $t$ to a vertex with type $s$. We note conditional on $(N_k)_{k \in \mathcal{S}}$ that 
    \[
    \bbar{A}_n(t, s) \sim \texttt{Bin}\left(\lfloor \mu n \rfloor, \prob(\mathcal{A}_{ts}^{(a)} ) \right).
    \]
    Thus, using Lemma~\ref{lem:arc vertex prob stable} we can stochastically bound this from above and below for some $\widehat{C} > 0$ by
    \[
   \underbrace{\texttt{Bin}\left(\lfloor \mu n \rfloor, \frac{q_t q_s \kappa(t, s)}{\mu} - \widehat{C}  \log(n) n^{- \frac{\tau}2}  \right)}_{B_n^-(t, s)} \preceq  \bbar{A}_n(t, s) \preceq \underbrace{\texttt{Bin}\left(\lfloor \mu n \rfloor, \frac{q_t q_s \kappa(t, s)}{\mu} + \widehat{C}  \log(n) n^{- \frac{\tau}2}  \right)}_{B_n^+(t, s)} ,
    \]
   where we use $B_n^-(t,s)$ and $B_n^+(t,s)$ to denote the random variables on, respectively, the left and right hand side. Using this in \eqref{eq:first double sum goal} allows us to bound the terms in the first sum as
    \begin{equation}\label{eq:stochastic bounded probability}
    \prob \left(B_n^+(t, s)  > C n^{1/2 + \alpha} \sqrt{q_t q_s} + \lfloor \kappa(t, s) q_t q_s n \rfloor \right) + \prob \left(B_n^-(t, s) < \lfloor \kappa(t, s) q_t q_s n \rfloor  -  C n^{1/2 + \alpha} \sqrt{q_t q_s}\right).
    \end{equation}
    We now apply the Chernoff bound on both these probabilities. We will only work out the first of the two, since the argument for the second is analogous. Like for all the previous applications of the Chernoff bound, we first sequentially bound $C n^{1/2 + \alpha} \sqrt{q_t q_s} + \lfloor \kappa(t, s) q_t q_s n \rfloor - \expec[B^+_n(t, s)]$ from below to show that it is positive. First, a direct calculation of this expression and using $\lfloor \mu n \rfloor \leq \mu n$ shows that it is bounded by
    \[
    C n^{\frac12 + \alpha} \sqrt{q_t q_s} + \lfloor \kappa(t, s) q_t q_s n \rfloor - \kappa(t, s) q_t q_s n - \mu  \widehat{C}  \log(n) n^{1 - \frac{\tau_1}2} \geq C n^{\frac{1}{2} + \alpha} \sqrt{q_t q_s} - 1 -  \mu  \widehat{C}  \log(n) n^{1 - \frac{\tau_1}2} .
    \]
    Using the fact that $t, s \leq \zeta_n$ allows us to further lower-bound this expression by
    \[
    C n^{-\frac{1}{2} + \alpha + \tau_1} - 1 -  \mu  \widehat{C}  \log(n) n^{1 - \frac{\tau_1}2}.
    \]
    Note that by our condition on $\alpha$ and $\tau_1$ it holds that $3 \tau_1 + 2\alpha > 3$, which implies that $\tau_1 + \alpha - 1/2 > 1-\tau_1/2$. Therefore, the first term in the expression above dominates. Thus, indeed there exists a $\widetilde{C} > 0$ such that for large enough $n$
\[
C n^{1/2 + \alpha} \sqrt{q_t q_s} + \lfloor \kappa(t, s) q_t q_s n \rfloor - \expec[B^+_n(t, s)] \geq \widetilde{C} n^{\alpha + \tau_1 - \frac12} > 0.
\]
    We can now apply the Chernoff bound, which shows the following bound for some constant $C^- > 0$:
    \[
     \prob \left(B_n^+(t, s)  > C n^{1/2 + \alpha} \sqrt{q_t q_s} + \lfloor \kappa(t, s) q_t q_s n \rfloor \right) \leq \exp\left(- C^- n^{\alpha + \tau_1 - \frac12}  \right).
    \]
    We repeat the same arguments, and find a similar bound for the probability involving $B_n^-(t, s)$ in \eqref{eq:stochastic bounded probability}. Then, we substitute both bounds in \eqref{eq:first double sum goal}, revealing a uniform bound for all probabilities in the furst sum. Thus, we see that there exists a $C^\pm > 0$ such that the first sum is bounded by
    \[
    2 \zeta_n^2 \cdot \exp\left(- C^\pm n^{\alpha + \tau_1 - \frac12}  \right) \to 0,
    \]
    by Lemma~\ref{lem:no stable}.
    \paragraph{Second and third double sum.} We will only give the argument for the second sum, because the argument for the third is the same. Similar to~\eqref{eq:first double sum goal}, we start by noting that each term in the second sum is equal to
    \[
    \prob \left(\bbar{A}_n(t, s)   > C n^{1/2 + \alpha} \sqrt{q_t q_s} + \lfloor \kappa(t, s) q_t q_s n \rfloor \right) + \prob \left(\bbar{A}_n(t, s) < \lfloor \kappa(t, s) q_t q_s n \rfloor  -  C n^{1/2 + \alpha} \sqrt{q_t q_s}\right).
    \]
    We will proceed to show that the second probability is zero by showing that $\lfloor \kappa(t, s) q_t q_s n \rfloor  -  C n^{1/2 + \alpha} \sqrt{q_t q_s}$ is negative. To do this, we subsequently bound
    \begin{align*}
        \lfloor \kappa(t, s) q_t q_s n \rfloor  -  C n^{1/2 + \alpha} \sqrt{q_t q_s} &\leq n^{1/2 + \alpha} \sqrt{q_t q_s} \left( \kappa(t, s) n^{1/2 - \alpha} \sqrt{q_t q_s} - C \right),\\
        &\leq n^{1/2 + \alpha} \sqrt{q_t q_s} \left(\kappa^+ n^{\tau_1/2 - \alpha} - C \right) \leq 0.
    \end{align*}
    Here, in the second inequality we have used that either $t$ or $s$ is larger than $\zeta_n$ and that $\kappa$ is bounded (Lemma~\ref{lem:kappa bounded}). In the final inequality we used that $\tau_1 < 2 \alpha$. Thus, to bound the second (and third) sum, we only need to bound
    \begin{equation}\label{eq:second sum goal}
        \prob \left(\bbar{A}_n(t, s)   > C n^{1/2 + \alpha} \sqrt{q_t q_s} + \lfloor \kappa(t, s) q_t q_s n \rfloor \right).
    \end{equation}
    Like in the previous step, we again have
    \[
     \bbar{A}_n(t, s) \sim \texttt{Bin}\left( \lfloor \mu n \rfloor, \prob(\mathcal{A}_{ts}^{(a)} ) \right),
    \]
    which we can stochastically bound using Lemma~\ref{lem:arc vertex prob no stable} for some $\widehat{C} > 0$ and all fixed $r > 0$ as
    \[
    \bbar{A}_n(t, s) \preceq \underbrace{\texttt{Bin}\left(\lfloor \mu n \rfloor, \widehat{C} \sqrt{q_t q_s} \left(\sqrt{q_t q_s} + \frac{\log(n)}{\sqrt{n}}   \right) + \frac{\widehat{C}}{n^r}\right)}_{B_n^+(t, s)}, 
    \]
    where we now use $B_n^+(t,s)$ to denote the random variable on the right hand side. This show we can bound \eqref{eq:second sum goal} by
    \begin{equation*}
        \prob \left(B_n^+(t, s)   > C n^{1/2 + \alpha} \sqrt{q_t q_s} + \lfloor \kappa(t, s) q_t q_s n \rfloor \right).
    \end{equation*}
    As always, we seek to apply the Chernoff bound, so we will show that indeed $\theta_{ts} := C n^{1/2 + \alpha} \sqrt{q_t q_s} + \lfloor \kappa(t, s) q_t q_s n \rfloor - \expec[B_n^+(t, s)] \geq 0$. Substituting the expectation and bounding $\lfloor \mu n \rfloor \leq \mu n$ yields
    \[
    \theta_{ts} = C n^{1/2 + \alpha} \sqrt{q_t q_s} + \lfloor \kappa(t, s) q_t q_s n \rfloor - \mu \widehat{C} n q_t q_s  - \mu \widehat{C} \log(n) \sqrt{q_t q_s n} - \mu \widehat{C} n^{1- r} .
    \]
    Since both $t, s \leq \xi_n$ we have that the first term in this expression is larger that $C n^{-1/2 + \alpha + \tau_2}$. Hence, by choosing $r$ sufficiently large in the last term, we see that the first term dominates it. Moreover, we also see that the first term dominates the fourth. In principle, this means that the fourth and final term are insignificant. Therefore, we rewrite to obtain the following lower bound:
    \[
    \theta_{ts} \geq n^{1/2 + \alpha} \sqrt{q_t q_s} \left( C - \mu \widehat{C} n^{1/2 - \alpha} \sqrt{q_t q_s} \right)  - \mu \widehat{C} \log(n) \sqrt{q_t q_s n} - \mu \widehat{C} n^{1- r} .    \]
    Without loss of generality, we know in these sums that $q_t \leq n^{-1 + \tau_1}$ (otherwise, this would be true for $q_s$). Hence, we can further bound
    \[
    \theta_{ts} \geq n^{1/2 + \alpha} \sqrt{q_t q_s} \left( C - \mu \widehat{C} n^{\tau_1/2 - \alpha}  \right)  - \mu \widehat{C} \log(n) \sqrt{q_t q_s n} - \mu \widehat{C} n^{1- r} .
    \]
    Again, since $\tau_1 < 2 \alpha$ we know from all previous arguments there exists a $\widetilde{C} > 0$ such that
    \[
    \theta_{ts} \geq \widetilde{C} n^{1/2 + \alpha} \sqrt{q_s q_t} > 0.
    \]
    This means we can apply the Chernoff bound. It shows for some $C^+ > 0$ that
    \[
     \prob \left(B_n^+(t, s)   > C n^{1/2 + \alpha} \sqrt{q_t q_s} + \lfloor \kappa(t, s) q_t q_s n \rfloor \right) \leq \exp\left(- C^+ n^{1/2 + \alpha} \sqrt{q_t q_s} \right).
    \]
    Since $t, s \leq \xi_n$, meaning $q_t, q_s \geq n^{-1 + \tau_2}$, we can further bound this as
    \[
    \prob \left(B_n^+(t, s)   > C n^{1/2 + \alpha} \sqrt{q_t q_s} + \lfloor \kappa(t, s) q_t q_s n \rfloor \right) \leq \exp\left(- C^+ n^{ \alpha + \tau_2 - 1/2} \right).
    \]
    Since we have assumed that $\alpha > 1/2 - \tau_2$, we have that $ \alpha + \tau_2 - 1/2 > 0$, showing that \eqref{eq:second sum goal} converges exponentially to zero. Thus, for the second (and third) sum we find it is bounded by
    \[
    \xi_n^2 \cdot \exp\left(- C^+ n^{ \alpha + \tau_2 - 1/2} \right) \to 0,
    \]
    by Lemma~\ref{lem:no stable}. Since all the sums converge to zero, we also have that the target probability converges to zero.
\end{proof}

\paragraph{Acknowledgments.} The research of Mike van Santvoort is funded by the Institute for Complex Molecular Systems (ICMS) at Eindhoven University of Technology.

\bibliographystyle{plain} 
\bibliography{references.bib} 

\end{document}